\documentclass[12pt]{article}
\usepackage{amssymb}
\usepackage{bm}
\usepackage{amsmath,amsthm}
\usepackage{a4wide}
\usepackage{parskip}
\usepackage{hyperref}
\usepackage{tikz}
\usepackage{mathtools}
\usepackage{accents}
\usetikzlibrary{arrows,decorations.markings,matrix}
\usetikzlibrary{patterns}
\usetikzlibrary{math}
\usetikzlibrary{calc}
\usetikzlibrary{shapes.multipart}
\usetikzlibrary{decorations.pathmorphing,shapes}
\usetikzlibrary{decorations.pathreplacing}

\tikzset{->-/.style={decoration={
              markings,
              mark=at position .5 with {\arrow{>}}},postaction={decorate}}}

\tikzset{-<-/.style={decoration={
              markings,
              mark=at position .5 with {\arrow{<}}},postaction={decorate}}}

\usepackage{comment}

\theoremstyle{plain}
\newtheorem{theorem}{Theorem}[subsection]
\newtheorem*{theorem*}{Theorem}
\newtheorem{lemma}[theorem]{Lemma}
\newtheorem{proposition}[theorem]{Proposition}
\newtheorem{corollary}[theorem]{Corollary}
\theoremstyle{definition}
\newtheorem{definition}[theorem]{Definition}
\newtheorem{example}[theorem]{Example}
\newtheorem{notation}[theorem]{Notation}

\newtheorem{question}[theorem]{Question}

\newtheorem{remark}[theorem]{Remark}
\newtheorem{remarks}[theorem]{Remarks}

\numberwithin{equation}{section}

\makeatletter
\def\thm@space@setup{%
  \thm@preskip=\parskip \thm@postskip=0pt
}
\makeatother

\newcommand{\Diff}{\operatorname{Diff}}

\def\Z{\mathbb{Z}}

\newcommand{\CC}{\mathbb{C}}
\newcommand{\mC}{\mathcal{C}}

\newcommand{\eE}{\mathcal{E}}

\newcommand{\QQ}{\mathbb{Q}}
\newcommand{\FF}{\mathbb{F}}

\newcommand{\JJ}{\mathcal{J}}

\newcommand{\MM}{\mathcal{M}}
\newcommand{\RR}{\mathbb{R}}

\newcommand{\PP}{\mathbb{P}}

\newcommand{\ZZ}{\mathbb{Z}}
\newcommand{\CP}{\mathbb{CP}}
\newcommand{\RP}{\mathbb{RP}}
\newcommand{\OP}[1]{\mathrm{#1}}
\newcommand{\til}{\widetilde}
\newcommand{\ha}{\widehat}
\newcommand{\ev}{\OP{ev}}
\newcommand{\ff}{\mathfrak{f}}

\newcommand{\Pav}{\OP{Pav}}

\newcommand{\Orb}{\mathcal{O}}
\newcommand{\Off}{\OP{Off}}
\newcommand{\ATF}{\mathfrak{A}}
\newcommand{\Vianna}{\mathfrak{D}}
\newcommand{\cul}{\mathrm{culet}}
\newcommand{\Tri}{\Delta}
\newcommand{\node}{\mathfrak{n}}
\newcommand{\vtx}{\mathfrak{v}}
\newcommand{\edge}{\mathfrak{e}}
\newcommand{\pt}{\mathfrak{p}}
\newcommand{\bc}{\mathfrak{b}}
\newcommand{\sth}{\, : \,}
\newcommand{\stair}{\operatorname{Stair}}

\newcommand{\id}{\operatorname{id}}
\newcommand{\Int}{\operatorname{Int}}
\newcommand{\FS}{\operatorname{FS}}
\newcommand{\vis}{\operatorname{vis}}

\newcommand{\se}{\overset%
{\raisebox{-.2ex}[0ex][-.2ex]{\mbox{$\scriptstyle s$}} \mskip 1mu}\hookrightarrow}

\def\ga{\alpha}
\def\gb{\beta}

\def\gf{\varphi}

\def\go{\omega}

\def\gs{\sigma}

\def\ca{{\mathcal A}}

\def\cc{{\mathcal C}}

\def\cf{{\mathcal F}}

\def\cn{{\mathcal N}}

\def\ct{{\mathcal T}}

\newcommand{\sbullet}{% 
{\scalebox{0.7}{$\bullet$}}}

\title{Markov staircases}

\author{Nikolas Adaloglou, Jo\'{e} Brendel, Jonny Evans,\\ Johannes Hauber and Felix Schlenk}

\newcommand{\Addresses}{{
  \bigskip
  \footnotesize

  N.~Adaloglou, \textsc{Mathematisch Instituut, Leiden},\\ \texttt{n.adaloglou@math.leidenuniv.nl}

  \medskip

  J.~Brendel, \textsc{D-MATH, ETH Z\"{u}rich},\\ \texttt{joe.brendel@math.ethz.ch}

  \medskip

  J.~Evans, \textsc{School of Mathematical Sciences, Lancaster
  University},\\ \texttt{j.d.evans@lancaster.ac.uk}
  
  \medskip

  J.~Hauber, \textsc{Institut de Math\'{e}matiques, Universit\'{e} de Neuch\^{a}tel},\\
  \texttt{johannes.hauber@unine.ch}

  \medskip
  
  F.~Schlenk, \textsc{Institut de Math\'{e}matiques, Universit\'{e} de Neuch\^{a}tel},\\
  \texttt{schlenk@unine.ch}
}}

\begin{document}

\maketitle

\begin{abstract}
Rational homology ellipsoids are certain Liouville domains diffeomorphic to rational homology balls and having Lagrangian pin-wheels as their skeleta. From the point of view of almost toric fibrations, they are a natural generalisation of usual symplectic ellipsoids.
We study symplectic embeddings of rational homology ellipsoids
into the complex projective plane and we show that for each Markov triple, 
this problem gives rise to an infinite staircase. 
A key ingredient in the proof is the result that any
two such embeddings are Hamiltonian isotopic. 
We also prove constraints on sizes for pairs of disjoint embeddings.
\end{abstract}

\tableofcontents

\section{Introduction}

\addtocontents{toc}{\protect\setcounter{tocdepth}{1}}
\subsection{Old questions for new domains}
\addtocontents{toc}{\protect\setcounter{tocdepth}{2}}

Let $(X,\omega)$ be a symplectic 4-manifold. 
By Darboux's theorem, small enough patches of~$X$ are symplectomorphic to subsets of the standard symplectic vector space~$(\CC^2, \omega_0)$. 
In particular one can always symplectically embed small enough symplectic ellipsoids
\begin{equation}
    \label{eq:stdellipsoid}
    E(\alpha,\beta)
    =
    \left\{
        (z_1,z_2) \in \CC^2 \sth  \frac{\pi \vert z_1 \vert^2}{\alpha} + \frac{\pi \vert z_2 \vert^2}{\beta} \leq 1
    \right\}
\end{equation}
into $(X,\omega)$.
This raises several natural questions: 

\begin{itemize}
    \item[(Q1)] {\bf Isotopy problem.} Given two such embeddings of the same size, are they ambiently isotopic?
    \item[(Q2)] {\bf Embedding problem.} For which values of the parameters $\alpha, \beta > 0$ do symplectic embeddings of $E(\alpha,\beta)$ exist?
    \item[(Q3)] {\bf Ellipsoid packing problem.} Given several embeddings of differently-sized ellipsoids, when can they be made disjoint?
\end{itemize}

In the case of symplectic balls $B(\alpha) = E(\alpha,\alpha)$, these questions were raised in Gromov's foundational work~\cite{Gro85}, and, ever since, the idea of using such embeddings to measure the symplectic sizes and quantitative properties of~$(X,\omega)$ has been a central theme in symplectic topology. 
In the work of McDuff and Schlenk \cite{McDSch12}, it was shown that the set of pairs $(\alpha,\beta)$ for which there is a symplectic embedding of $E(\alpha,\beta)$ into the ball takes the shape of an infinite staircase, whose step-sizes are related to the sequence of odd-indexed Fibonacci numbers.
Since then, the ellipsoid embedding question into various target spaces~$X$ has attracted a lot of attention, see for instance the survey~\cite{Sch18}. 
In this paper, instead of looking at new target spaces, we consider new domains. 

Our goal is to study symplectic embedding questions for a family of domains we call \emph{pin-ellipsoids}\/ $E_{p,q}(\alpha, \beta)$, where $p$ and $q$ are positive coprime integers with~$q \leq p$.   
See Definition~\ref{def:Epq} for their definition.
These domains arise naturally (1)~in Milnor fibres of smoothings of certain surface singularities called \emph{cyclic quotient T-singularities}, and (2)~as building blocks of \emph{almost toric fibrations}.
The name pin-ellipsoid is derived from the fact that every pin-ellipsoid contains a so-called \emph{Lagrangian pin-wheel}. 
The case $p=q=1$ corresponds to the case of the \emph{standard ellipsoid}\/ in Equation \eqref{eq:stdellipsoid}, whereas $E_{2,1}(\alpha,\beta)$ can be identified with 
a neighbourhood of the zero-section of~$T^*\RP^2$. 
We address the above questions (Q1)--(Q3) for all pin-ellipsoid embeddings into $X=\CP^2$.
Theorem~\ref{thm:uniqueness} affirmatively answers the isotopy question~(Q1), 
Theorem~\ref{thm:Markovstairs} 
solves the embedding problem~(Q2) in the ``visible range" (see Remark~\ref{rmk:pin_ellipsoids}~(d)),
and Theorem~\ref{thm:twoball} solves the multiple pin-ellipsoid problem~(Q3) 
for the special case of {\em pin-balls} where \(\alpha=\beta\), 
analogous to Gromov's Two Ball Theorem {\cite[\S\,0.3.B]{Gro85}.
We begin by defining pin-ellipsoids and also quickly recall the Markov numbers, which will play a key role in the statements of all of our theorems.

\addtocontents{toc}{\protect\setcounter{tocdepth}{1}}
\subsection{Pin-ellipsoids, pin-balls and pin-wheels}
\addtocontents{toc}{\protect\setcounter{tocdepth}{2}}

Our definition of the pin-ellipsoid \(E_{p,q}(\alpha,\beta)\) uses the theory of {\em almost toric fibrations} (ATFs), where the symplectic geometry of a 4-manifold is encoded in a 2-dimensional {\em almost toric base diagram}. 
These diagrams were introduced by Symington \cite{Symington1}; see the book~\cite{ELTF} for more details about how to decode these diagrams.

\begin{definition}\label{def:Epq}
    Given positive integers \(1\leq q\leq p\)
    with \(\gcd(p,q)=1\) and positive real numbers \(\alpha,\beta\),
    let the \emph{pin-ellipsoid}\/ \(E_{p,q}(\alpha,\beta)\) be the compact symplectic manifold
    corresponding to the almost toric base diagram
    \(\ATF_{p,q}(\alpha,\beta)\) shown in Figure~\ref{fig:atbd}. 
    This base diagram has a slanted edge pointing in the
    \((p^2,pq-1)\)-direction, having integral affine length \(\alpha\), and a
    vertical edge having integral affine length \(\beta\); it has a single
    focus-focus fibre joined to the origin by a branch cut parallel to the
    \((p,q)\)-direction. 
    The top side, drawn \tikz[baseline=-0.5ex]{\draw[very thick,dash pattern=on 9pt off 4pt] (0,0) -- (0.8,0);} is included in the diagram but it is not part of the toric boundary; the preimage of this side under the almost toric fibration is the boundary of \(E_{p,q}(\alpha,\beta)\), which is a 3-manifold diffeomorphic to the lens space \(L(p^2,pq-1)\).
\end{definition}

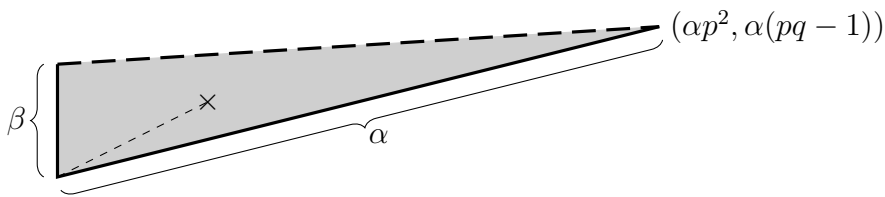
\begin{figure}[htb]
  \begin{center}   
    \begin{tikzpicture}
      \filldraw[lightgray,opacity=0.75] (0,1.5) -- (0,0) -- (8,2) -- cycle;
      \draw[very thick] (0,1.5) -- (0,0) -- (8,2);
      \draw[dashed] (0,0) -- (2,1);
      \draw [decorate,decoration={brace,amplitude=5pt,raise=1ex}] (0,0) -- (0,1.5) node[midway,xshift=-3ex]{\(\beta\)};
      \draw [decorate,decoration={brace,amplitude=5pt,raise=1ex}] (8,2) -- (0,0) node[midway,xshift=1.5ex,yshift=-2.5ex]{\(\alpha\)};
      \node at (8,2) [right] {\((\alpha p^2,\alpha(pq-1))\)};
      \node at (2,1) {\(\times\)};
      \draw[very thick,dash pattern=on 9pt off 4pt] (0,1.5) -- (8,2);
    \end{tikzpicture}
    \caption{The almost toric base diagram
      \(\ATF_{p,q}(\alpha,\beta)\) for
      \(E_{p,q}(\alpha,\beta)\). The branch cut is parallel to the \((p,q)\)-direction.  
      The toric boundary comprises
      the two solid edges (not the top side), whose affine
      lengths are \(\alpha\) and \(\beta\); these are two
      segments of an unbroken straight line with respect to the integral
      affine structure on the base diagram.}
    \label{fig:atbd}
  \end{center}
\end{figure}

\begin{remark}\label{rmk:pin_ellipsoids}
\begin{itemize}
    \item[(a)] We distinguish between {\em edges} of an almost toric base diagram (which we draw as thick, unbroken lines \tikz[baseline=-0.5ex]{\draw[very thick] (0,0) -- (1,0);}) which are part of the toric boundary (that is, the fibre over a point of an edge is a circle or a point), 
    and {\em sides} (drawn as thick, long-dashed lines \tikz[baseline=-0.5ex]{\draw[very thick,dash pattern=on 9pt off 4pt] (0,0) -- (1,0);}) which are not part of the toric boundary (that is, the fibre over a point of a side is a 2-torus). 
    Thin, short-dashed lines \tikz[baseline=-0.5ex]{\draw[dashed] (0,0) -- (1,0);} on the interior of the polygon will represent branch cuts.
    
    \item[(b)] Note that \(E_{p,q}(\alpha,\beta)\) is symplectomorphic to \(E_{p,p-q}(\beta,\alpha)\): the almost toric base diagrams are related by an integral affine transformation with determinant \(-1\), which lifts to a symplectomorphism of the total spaces.
    It follows that pin-ellipsoid embedding problems in the case $(p,q)=(2,1)$ will be symmetric with respect to swapping $\alpha$ and~$\beta$, just like for standard ellipsoids. 
    For other $(p,q)$ this will no longer hold true, as will be evident from Theorem~\ref{thm:Markovstairs}. Compare also Remark~\ref{rmk:staircases}. 
    
    \item[(c)] We call the pin-ellipsoid \(E_{p,q}(\alpha,\alpha)\) the {\em pin-ball} \(B_{p,q}(\alpha)\).
    
    \item[(d)] Given an almost toric base diagram which contains a subset integral-affine equivalent to \(\ATF_{p,q}(\alpha,\beta)\), its preimage under the almost toric fibration is an embedded copy of \(E_{p,q}(\alpha,\beta)\) in the associated almost toric manifold.
    We refer to these as {\em visible pin-ellipsoids}.
\end{itemize}
\end{remark}

\begin{definition}
    Inside \(E_{p,q}(\alpha,\beta)\) there is an immersed Lagrangian disc \(L_{p,q}\) which lives over the branch cut. 
    This immersion fails to be an embedding along its boundary circle which meets the toric boundary; here it is \(p\)-to-\(1\) and modelled on a standard model immersion, described in {\cite[Equations (3.4) and (3.5)]{Khod1}} or {\cite[Example 5.14]{ELTF}}. 
    This immersed disc is called a {\em Lagrangian \((p,q)\)-pin-wheel}.
\end{definition}

\begin{remark}
    More generally, we call any immersed Lagrangian disc which is embedded away from its boundary and modelled on \(L_{p,q}\) along its boundary a Lagrangian pin-wheel. 
    Khodorovskiy {\cite[Lemma 3.4]{Khod1}} shows that any Lagrangian pin-wheel admits a neighbourhood symplectomorphic to \(E_{p,q}(\alpha,\beta)\) for any sufficiently small \(\alpha,\beta>0\).
    We will call a Lagrangian pin-wheel {\em visible}\/ if it projects to an arc in the base of some almost toric fibration.
\end{remark}

\addtocontents{toc}{\protect\setcounter{tocdepth}{1}}
\subsection{Markov numbers and their companions}
\addtocontents{toc}{\protect\setcounter{tocdepth}{2}}

\begin{definition}\label{par:markovnum} 
    Integer solutions to the Markov equation
    \begin{equation*}
        p_1^2 + p_2^2 + p_3^2 = 3p_1p_2p_3
    \end{equation*}
    are called \emph{Markov triples}. 
    One can produce new Markov triples from a given one by three \emph{mutations}, each of which substitutes a member of the triple by another Markov number. 
    For example the mutation $(p_1,p_2,p_3) \rightarrow (p_1,p_2,3p_1p_2-p_3)$ replaces $p_3$ by $p_3' = 3p_1p_2-p_3$. 
    All Markov triples are obtained in this way from iterated mutations of~$(1,1,1)$. 
    Thus they are naturally arranged according to the graph depicted in Figure \ref{fig:Markovtree}. 
\end{definition}

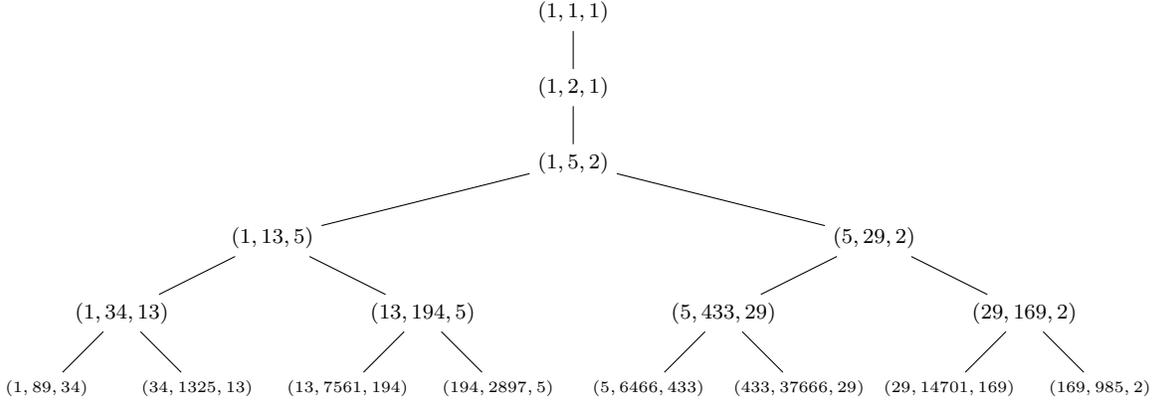
\begin{figure}[ht]
\scriptsize
\centering
\tikzstyle{level 1}=[level distance = 1cm, sibling distance=8cm]
\tikzstyle{level 2}=[level distance = 1cm, sibling distance=8cm]
\tikzstyle{level 3}=[level distance = 1cm, sibling distance=8cm]
\tikzstyle{level 4}=[sibling distance=4cm]
\tikzstyle{level 5}=[sibling distance=2cm]
\begin{tikzpicture}[]
  \node {$(1, 1, 1)$}
    child {node{$(1,2,1)$}
    child {node{$(1,5,2)$}
    child {node {$(1, 13, 5)$}
      child {node {$(1, 34, 13)$}
        child {node[align=right] {{\tiny$(1, 89, 34)$}}}
        child {node[align=right] {{\tiny$(34, 1325, 13)$}}}
      }
      child {node {$(13, 194, 5)$}
        child {node[align=right] {{\tiny$(13, 7561, 194)$}}}
        child {node[align=right] {{\tiny$(194, 2897, 5)$}}}
      }
    }
    child {node {$(5, 29, 2)$}
      child {node {$(5, 433, 29)$}
        child {node[align=right] {{\tiny$(5, 6466, 433)$}}}
        child {node[align=right] {{\tiny$(433, 37666, 29)$}}}
      }
      child {node {$(29, 169, 2)$}
        child {node[align=right] {\tiny{$(29, 14701, 169)$}}}
        child {node[align=right] {\tiny{$(169, 985, 2)$}}}
      }
    }}};
\end{tikzpicture}
\caption{The beginning of the Markov graph, where we have omitted the repetitions at the first two Markov triples.} \label{fig:Markovtree}
\end{figure}

\begin{definition}[Companion numbers]\label{par:comp}
    Let $p$ be a Markov number.
    A number $q$ is called \emph{companion number of $p$} if there exists a Markov triple $(p,p_2,p_3)$ such that 
    \begin{equation}
        \label{eq:companion}
        q \equiv \pm 3p_2p_3^{-1} \mod p.
    \end{equation}
\end{definition}

\begin{remarks}
    \begin{itemize}
        \item[(a)] For a fixed triple, Equation \eqref{eq:companion} defines two companion numbers with $1 \leq q \leq p$ which are related to one another by $q \rightarrow p - q$.
        If $p \in \{1,2\}$, then the two companion numbers are both equal to~$1$, and for all other Markov numbers they differ and are strictly less than~$p$.
        Mutations of Markov triples containing $p$ preserve its companion numbers. 
        \item[(b)] Conversely, there is a unique mutation branch in the Markov tree in which $p$ appears with companion number~$q$, see for example {\cite[Proposition 3.15]{Aig}}.
        It is conjectured that each Markov number has only one pair of companion numbers (see~{\cite[\S 2.3]{Aig}}), but since this conjecture is famously open, we need to specify both \(p\) and \(q\) in everything that follows. 
    \end{itemize}
\end{remarks}

As was noticed by Vianna \cite{Via14}, the almost toric base diagrams of $\CP^2$ (up to integral affine equivalence and nodal slides) are in bijection with unordered Markov triples. 
This is closely related to the fact -- due to Hacking--Prokhorov \cite{HacPro10} -- that the singular fibre of any \(\QQ\)-Gorenstein degeneration of \(\CP^2\) is (a partial smoothing of) a weighted projective plane \(\PP(p_1^2,p_2^2,p_3^2)\) where \((p_1,p_2,p_3)\) is a Markov triple. 
Given a Markov triple, let us write \(\Vianna(p_1,p_2,p_3)\) for the corresponding almost toric base diagram. 
We call these almost toric base diagrams {\em Vianna triangles}: see Section~\ref{subsec:Viannatriangles} for a full description. 
For now, the important feature is that Vianna triangles allow us to see {\em visible}\/ pin-ellipsoids
\begin{equation}
    \label{eq:vianna_ellipsoids}
    E_{p_i,q_i}\left(\alpha_i,\beta_i\right)\subset\CP^2\quad \mbox{for any}\quad\alpha_i<\frac{p_{i+2}}{p_ip_{i+1}},\quad\beta_i<\frac{p_{i+1}}{p_ip_{i+2}},
\end{equation}
where \(q_i\) is the companion number \(3p_{i+1}p_{i+2}^{-1}\mod p_i\).
Note that reordering the Markov triple by swapping \(p_{i+1}\) and \(p_{i+2}\) will swap the order of the bounds on \(\alpha_i\) and \(\beta_i\) and also swap the companion numbers; this is consistent with Remark~\ref{rmk:pin_ellipsoids}(b).

\addtocontents{toc}{\protect\setcounter{tocdepth}{1}}
\subsection{The isotopy problem}
\addtocontents{toc}{\protect\setcounter{tocdepth}{2}}

Evans and Smith \cite{ES} showed that $E_{p,q}(\alpha,\beta)$ admits a symplectic embedding into \(\CP^2\) for some \(\alpha,\beta>0\) 
if and only if \(p\) is a Markov number and \(q\) is one of its companion numbers.
Provided that \(\alpha<\frac{p_3}{pp_2}\) and \(\beta<\frac{p_2}{pp_3}\) for some Markov triple \((p,p_2,p_3)\) with \(q=3p_2p_3^{-1}\mod p\), there is a visible embedding \(E_{p,q}(\alpha,\beta)\hookrightarrow \CP^2\) coming from the Vianna triangle \(\Vianna(p,p_2,p_3)\).
But there exist symplectic embeddings also outside the visible range;
for instance, symplectic folding can be used to show that for any $\alpha >0$ there exists $\beta >0$ such that
\(E_{p,q}(\alpha,\beta)\) symplectically embeds into~\(\CP^2\).

\begin{theorem}[Isotopy Theorem]\label{thm:uniqueness} 
    Any two symplectic embeddings of a pin-ellipsoid \(E_{p,q}(\alpha,\beta)\) into \(\CP^2\)
           are isotopic through symplectic embeddings.
\end{theorem}

In other words, the space of symplectic embeddings of a pin-ellipsoid into \(\CP^2\), if non-empty, 
is path-connected.

\begin{remarks}\label{rmks:isotopy_remarks}
    \begin{itemize}
        \item[(a)] 
            The problem whether the space of symplectic embeddings of $E_{p,q}(\ga,\gb)$ 
            into~$\CP^2$ is non-empty is addressed 
            in $\S$~\ref{ss:embedding}. 
            
        \item[(b)]
            Since $H^1(E_{p,q}(\alpha,\beta);\RR)$ vanishes, it follows that any two embeddings which are symplectically isotopic are related by a Hamiltonian diffeomorphism of $\CP^2$.
            
        \item[(c)]
            In fact, we will reprove the results of Evans and Smith \cite{ES} along the way, avoiding the use of orbifold holomorphic curves. 
            See Section~\ref{subsec:ES}. 
    \end{itemize}
\end{remarks}

\begin{corollary}
    \label{cor:pinwheeluniqueness}
    Lagrangian $(p,q)$-pin-wheels in $\CP^2$ are unique up to Hamiltonian isotopy.
\end{corollary}

\begin{proof}
    Indeed, Khodorovskiy \cite[Lemma 3.4]{Khod1} shows that a Lagrangian pin-wheel has a standard neighbourhood symplectomorphic to an embedded pin-ellipsoid. 
    Therefore, two pin-wheels yield two embedded pin-ellipsoids. 
    If we choose them sufficiently small and of equal sizes $\ga,\gb$, Theorem~\ref{thm:uniqueness} shows that these pin-ellipsoids are Hamiltonian isotopic, and the restriction of this isotopy to the pin-wheel gives the desired isotopy of pin-wheels.
\end{proof}

Since the $(2,1)$-pin-wheel is $\RP^2$, this generalises the known result about the uniqueness of Lagrangian $\RP^2$s in $\CP^2$ due to Hind~\cite{Hi10} and Li--Wu~\cite[Section~6.4.1]{LiWu12}, see also Borman--Li--Wu \cite[Theorem 1.3]{BormanLiWu} and Adaloglou~\cite{Ada25} for an approach closer to ours.

\begin{remark}
    Every Lagrangian pin-wheel $L\subset\CP^2$ is a non-trivial Lagrangian barrier in the sense of Biran~\cite{BiranBarriers}: the Gromov width of the complement \(\CP^2\setminus L\) is strictly less than the Gromov with of \(\CP^2\).
    Indeed, by Corollary \ref{cor:pinwheeluniqueness} any pin-wheel in \(\CP^2\) is Hamiltonian isotopic to one of the {\em visible}\/ Lagrangian pin-wheels studied by 
    Brendel and Schlenk~\cite{BrendelSchlenk}, who proved that these are barriers and moreover computed the Gromov width of the complement.
\end{remark}

\addtocontents{toc}{\protect\setcounter{tocdepth}{1}}
\subsection{The embedding problem} \label{ss:embedding}
\addtocontents{toc}{\protect\setcounter{tocdepth}{2}}

\begin{notation}
    \label{par:notation} Answering the embedding problem (Q2) for pin-ellipsoid embeddings into $\CP^2$ means determining the set
    \begin{equation}
        \mathcal{A}_{p,q}
        \coloneqq
        \left\{
            (\alpha,\beta) \in \RR^2_{>0} \sth
             E_{p,q}(\alpha,\beta) \se \CP^2
        \right\},
    \end{equation}
    where $p$ is a Markov number and $q$ one of its companion numbers.
    By the definition of symplectic embedding, the set $\ca_{p,q}$ is open in~$\RR_{>0}^2$.
    Let $(p,m_0,m_1)$ be a Markov triple with $q \equiv 3m_0m_1^{-1} \mod p$ and consider the recursion
    \begin{equation}\label{eq:mi}
        m_{i+2} = 3pm_{i+1} - m_i,
    \end{equation}
    defining a sequence $\{m_i\}_{i \in \ZZ}$. 
    The set $\{(p,m_i,m_{i+1}) \mid i\in\ZZ\}$ consists of all Markov triples which can be obtained from $(p,m_0,m_1)$ 
    by mutations preserving~$p$: 
    in the Markov tree in Figure~\ref{fig:Markovtree} this set forms a $\bigwedge$-shaped tree.
    Let 
    \begin{equation}
        \label{eq:stairpq}
        \square_i(p,q) \coloneqq \left(0, \frac{m_{i+1}}{pm_i} \right) \times \left(0, \frac{m_i}{pm_{i+1}} \right),\quad
        \stair(p,q) \coloneqq \bigcup_{i \in \ZZ} \square_i(p,q),
    \end{equation}
    and
    \begin{equation*}
        \sigma_p 
        \coloneqq
        \frac{1}{2} \left( 3 + \sqrt{9 - \frac{4}{p^2}} \right).
    \end{equation*}
\end{notation}

\begin{theorem}[Staircase Theorem]\label{thm:Markovstairs} 
    Let $p$ be a Markov number and $q$ one of its companion numbers. 
    Then 
    \begin{equation}
        \label{eq:Markovstairs}
        \mathcal{A}_{p,q} \cap (0,\sigma_p)^2 = \stair(p,q).
    \end{equation}
\end{theorem}

\begin{remarks}\label{rmk:staircases}
\begin{itemize}
    \item[(a)] Since $\operatorname{vol}(E_{p,q}(\alpha,\beta)) = \frac{1}{2}p^2\alpha\beta$, the volume constraint for the symplectic embedding problem yields that $\mathcal{A}_{p,q}$ lies below the curve $\{p^2\alpha\beta = 1\}$.
    
    \item[(b)] All the pin-ellipsoids in this staircase can be realised as visible pin-ellipsoids coming from Vianna triangles (compare with Equation~\eqref{eq:vianna_ellipsoids})
    and by 
    Theorem~\ref{thm:uniqueness} are isotopic to a visible pin-ellipsoid.
    Conversely, all visible embeddings of \(E_{p,q}(\alpha,\beta)\) satisfy \(\max \{\alpha,\beta\} < \sigma_p.\)
    It is striking that visible constructions are often sharp for such quantitative problems.
    
    \item[(c)] The sets in Equation \eqref{eq:Markovstairs} are infinite staircases whose outer corners lie on the volume constraint and accumulate at the boundary of the square $(0,\sigma_p)^2 \subset \RR^2$, see Figures~\ref{fig:staircase_2_1} and~\ref{fig:staircase_5_1}. 
    The first such staircase, which we recover as the case $(p,q)=(1,1)$, was found by McDuff and Schlenk~\cite{McDSch12}. 
    Since then, similar staircases have been found for many different target spaces -- we refer to~\cite{Sch18} and the references therein for more on this and related problems. 
    The usual convention is to use the normalisation freedom to reduce to the case $E(1,\beta) \hookrightarrow \CP^2(A)$, where $A>0$ is the symplectic area of the line. 
    In our case, we prefer the normalisation convention $E(\alpha,\beta) \hookrightarrow \CP^2(1)$. 
    This is in part justified by the fact that the staircases we find in Theorem~\ref{thm:Markovstairs} are not symmetric with respect to swapping $\alpha$ and~$\beta$, making it more natural to keep $\alpha$ and~$\beta$ on equal footing. 
    In fact, the set obtained by swapping coordinates in $\stair(p,q)$ describes $\mathcal{A}_{p,p-q}$, the staircase associated to the second companion number $p-q$, by Remark~\ref{rmk:pin_ellipsoids}(b).

    \item[(d)] We do not know the set $\ca_{p,q}$ for $\ga \geq \gs_p$ or 
    $\gb \geq \gs_p$, except for $p=q=1$ where it can be derived from~\cite[Sections~4 and~5]{McDSch12}. 
    For $p\geq 2$ we even do not know whether the intersection of the closure~$\overline{\ca_{p,q}}$ with the vertical segment at $\ga = \gs_p$ is a point or an interval of positive length. For $p=1$ this intersection is a point, since the upper boundary of $\ca_{1,1}$ on $\{ \tau^2 \leq \ga \leq \frac{21}{8}\}$
    is given by $\beta = 3-\ga$.
\end{itemize}
\end{remarks}

For many (though not all, \cite[Remark 5.1]{CrHi24}) 
symplectic 4-manifolds $(X,\go)$, 
there are full symplectic fillings $E(\alpha, \beta) \hookrightarrow (X,\go)$ by standard ellipsoids whenever $\frac \ga \gb$ is sufficiently large or sufficiently small.
For instance, the boundary of $\ca_{1,1}$ over $\{ \frac{17}{6} \leq \ga < \infty\}$ is given by $\beta = \frac 1 \ga$.
Is this so for all Markov numbers~$p$? More precisely:

\begin{question}
    Let $(p,q)$ be as above. Are there full symplectic fillings $E_{p,q}(\alpha,\beta) \hookrightarrow \CP^2$ for all sufficiently small and all sufficiently large $\frac{\alpha}{\beta}$\,?
\end{question}

\begin{example}
    Let us illustrate Theorem \ref{thm:Markovstairs} by the example $p=2,q=1$, see Figures~\ref{fig:stairs_2_vianna} 
    and~\ref{fig:staircase_2_1}.
    We start with the Vianna triangle $\Vianna(1,1,1)$ appearing as the image of $\CP^2$ under the standard toric moment map. 
    After one mutation, we obtain the triangle $\Vianna(2,1,1)$, containing a vertex which is integral affine equivalent to the cone spanned by the vectors $(0,-1)$ and $(4,-1)$ adjacent to edges which both have integral affine length~$\frac{1}{2}$.
    Thus we deduce the existence of a volume filling embedding $E_{2,1}(\frac{1}{2},\frac{1}{2}) \hookrightarrow \CP^2$.
    This domain is symplectomorphic to the codisc bundle $D_{1/2\pi} ^*\RP^2$ of radius $\frac{1}{2\pi}$ in the round metric of curvature~1 on~$\RP^2$.
    One more mutation at another vertex yields the triangles $\Vianna(2,1,5)$ and $\Vianna(2,5,1)$ and hence volume filling embeddings of $E_{2,1}(\frac{5}{2},\frac{1}{10})$ and $E_{2,1}(\frac{1}{10},\frac{5}{2})$ into $\CP^2$, respectively.
    This generates three steps of the $(2,1)$-staircase, as illustrated in Figure~\ref{fig:staircase_2_1}.
    
    The rest of the \(p=2\), \(q=1\) staircase comes from the sequence:
        \begin{equation*}
            \{m_i\}_{i \in \ZZ} = \{\ldots,29,5,1,1,5,29,169,\ldots\}
        \end{equation*}
    of odd-index Pell numbers, which yields the staircase depicted in Figure~\ref{fig:staircase_2_1}. 
    This staircase is symmetric with respect to swapping the coordinate axes.
\end{example}

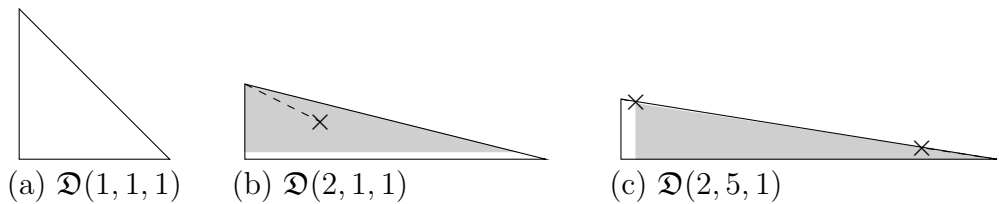
\begin{figure}[htb]
  \begin{center}
    \begin{tikzpicture}
      \draw (0,0) -- (2,0) -- (0,2) -- cycle;
      \node at (1,0) [below] {(a) \(\Vianna(1,1,1)\)};
      \begin{scope}[shift={(3,0)}]
      \filldraw[lightgray,opacity=0.75] (0,0.1) -- (4-0.4,0.1) -- (0,1) -- cycle;
      \draw (0,0) -- (4,0) -- (0,1) -- cycle;
      \node at (1,0) [below] {(b) \(\Vianna(2,1,1)\)};
      \draw[dashed] (0,1) -- (1,0.5) node {\(\times\)};
      \end{scope}
      \begin{scope}[shift={(8,0)}]
      \filldraw[lightgray,opacity=0.75] (0.2,2*46/125) -- (5,0) -- (0.2,0) -- cycle;
      \draw (0,0) -- (5,0) -- (0,4/5) -- cycle;
      \node at (1,0) [below] {(c) \(\Vianna(2,5,1)\)};
      \draw[dashed] (0,4/5) -- ++ (5/25,-1/25) node {\(\times\)};
      \draw[dashed] (5,0) -- ++ (-13/13,2/13) node {\(\times\)};
      \end{scope}
    \end{tikzpicture}
    \caption{The first three Vianna triangles (a) \(\Vianna(1,1,1)\), (b) \(\Vianna(2,1,1)\) and (c) \(\Vianna(2,5,1)\) together with the first two visible pin-ellipsoids (shaded in (b) and~(c)) for the \(p=2,q=1\) staircase. These visible embeddings can be made to fill an arbitrarily large portion of the triangle; in case (c) you need to use a nodal slide to shrink the other branch cut.}
    \label{fig:stairs_2_vianna}
  \end{center}
\end{figure}

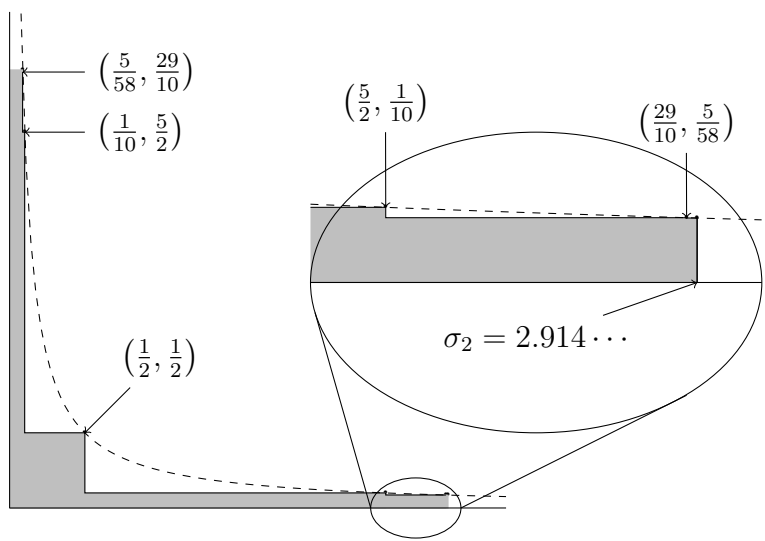
\begin{figure}[htb]
  \begin{center}
    \begin{tikzpicture}[scale=2]
      \node at ({169/(2*29)},{29/(169*2)}) {\(\cdot\)};
      \node at ({5/(2*29)},{29/(5*2)}) {\(\cdot\)};
      \node at ({1/(2*5)},{5/(1*2)}) {\(\cdot\)};
      \node at ({1/(2*1)},{1/(1*2)}) {\(\cdot\)};
      \node at ({5/(2*1)},{1/(2*5)}) {\(\cdot\)};
      \node at ({29/(2*5)},{5/(2*29)}) {\(\cdot\)};
      \node at ({29/(2*169)},{169/(2*29)}) {\(\cdot\)};
      \filldraw[lightgray] (0,1/2) -- (1/2,1/2) -- (1/2,0) -- (0,0) -- cycle;
      \filldraw[lightgray] (0,1/10) -- (5/2,1/10) -- (5/2,0) -- (0,0) -- cycle;
      \filldraw[lightgray] (1/10,0) -- (1/10,5/2) -- (0,5/2) -- (0,0) -- cycle;
      \filldraw[lightgray] (0,29/10) -- (5/58,29/10) -- (5/58,0) -- (0,0) -- cycle;
      \filldraw[lightgray] (29/10,0) -- (29/10,5/58) -- (0,5/58) -- (0,0) -- cycle;
      \filldraw[lightgray] (0,{(3-sqrt(8))/2}) -- ({(3+sqrt(8))/2},{(3-sqrt(8))/2}) -- ({(3+sqrt(8))/2},0) -- (0,0) -- cycle;
      \filldraw[lightgray] ({(3-sqrt(8))/2},0) -- ({(3-sqrt(8))/2},{(3+sqrt(8))/2}) -- (0,{(3+sqrt(8))/2}) -- (0,0) -- cycle;
      \draw (0,0) -- (3.3,0);
      \draw (0,0) -- (0,3.3);
      \draw (5/58,29/10) -- (5/58,5/2) -- (1/10,5/2) --
      (1/10,1/2) -- (1/2,1/2) -- (1/2,1/10) -- (5/2,1/10) --
      (5/2,5/58) -- (29/10,5/58);
      \draw[dashed,samples=100,domain=0.076:3.3,variable=\x]
      plot ({\x},{1/(4*\x)});
      \draw (2.7,0) circle [x radius=0.3,y radius=0.2];
      \draw (2.4,0) -- (2.03,1.3);
      \draw (3,0) -- (4.5,0.75);
      \node (c) at (1,1)
      {\(\left(\frac{1}{2},\frac{1}{2}\right)\)};
      \draw[->] (c) -- (0.5,0.5);
      \node (d) at (1/2,29/10) [right]
      {\(\left(\frac{5}{58},\frac{29}{10}\right)\)};
      \draw[->] (d) -- (5/58,29/10);
      \node (e) at (5/10,5/2) [right]
      {\(\left(\frac{1}{10},\frac{5}{2}\right)\)};
      \draw[->] (e) -- (1/10,5/2);
      \begin{scope}[scale=5,shift={(-2,0.3)}]
        \filldraw[lightgray,draw=black] (2.4,1/10) -- (5/2,1/10) -- (5/2,5/58) --
        (29/10,5/58) -- (29/10,{29/(2*169)}) --
        ({169/(2*29)},{29/(2*169)}) -- ({169/(2*29)},0) --
        (2.4,0);
        \draw[dashed,samples=100,domain=2.4:3,variable=\x]
        plot ({\x},{1/(4*\x)});
        \draw (2.4,0) -- (3,0);
        \node at ({169/(2*29)},{29/(169*2)}) {\(\cdot\)};
        \node at ({29/(5*2)},{5/(2*29)}) {\(\cdot\)};
        \node at ({5/(1*2)},{1/(2*5)}) {\(\cdot\)};
        \draw ({(3+sqrt(8))/2},{(3-sqrt(8))/2}) --
        ({(3+sqrt(8))/2},0);
        \draw (2.7,0) circle [x radius=0.3,y radius=0.2];
        \node (f) at (2.7,-0.075) {\(\sigma_2=2.914\cdots\)};
        \draw[->] (f) -- ({(3+sqrt(8))/2},0);
        \node (a) at (29/10,10/58) [above]
        {\(\left(\frac{29}{10},\frac{5}{58}\right)\)};
        \draw[->] (a) -- (29/10,5/58);
        \node (b) at (5/2,2/10) [above]
        {\(\left(\frac{5}{2},\frac{1}{10}\right)\)};
        \draw[->] (b) -- (5/2,1/10);
      \end{scope}
    \end{tikzpicture}
    \caption{The staircase for \(p=2\), \(q=1\).}
    \label{fig:staircase_2_1}
  \end{center}
\end{figure}

\begin{example}
    For $(p,q)=(5,1)$, we obtain 
    \begin{equation*}
        \{m_i\}_{i \in \ZZ} = \{\ldots,433,29,2,1,13,194,2897,\ldots\},
    \end{equation*}
    yielding the staircase depicted in Figure \ref{fig:staircase_5_1}.
    This staircase is not symmetric with respect to swapping the coordinate axes.
\end{example}

\begin{figure}[htb]
  \begin{center}
    \begin{tikzpicture}[scale=2]
      \node at (194/14485,2897/970) {\(\cdot\)};
      \node at (13/970,194/65) {\(\cdot\)};
      \node at (1/65,13/5) {\(\cdot\)};
      \node at (2/5,1/10) {\(\cdot\)};
      \node at (29/10,2/145) {\(\cdot\)};
      \node at (433/145,29/2165) {\(\cdot\)};
        \node (d) at (2/5+0.4,1/10+0.4) {\(\left(\frac{2}{5},\frac{1}{10}\right)\)};
      \filldraw[lightgray] (0,2897/970) -- (194/14485,2897/970)
      -- (194/14485,0) -- (0,0) -- cycle;
      \filldraw[lightgray] (0,194/65) -- (13/970,194/65) --
      (13/970,0) -- (0,0) -- cycle;
      \filldraw[lightgray] (0,13/5) -- (1/65,13/5) --
      (1/65,0) -- (0,0) -- cycle;
      \filldraw[lightgray] (0,1/10) -- (2/5,1/10) --
      (2/5,0) -- (0,0) -- cycle;
      \filldraw[lightgray] (0,2/145) -- (29/10,2/145) --
      (29/10,0) -- (0,0) -- cycle;
      \filldraw[lightgray] (0,29/2165) -- (433/145,29/2165) --
      (0,29/2165) -- (0,0) -- cycle;
      \draw (0,0) -- (3.3,0);
      \draw (0,0) -- (0,3.3);
      \draw (194/14485,2897/970) -- (194/14485,194/65) --
      (13/970,194/65) -- (13/970,13/5) -- (1/65,13/5) --
      (1/65,1/10) -- (2/5,1/10) -- (2/5,2/145) -- (29/10,2/145)
      -- (29/10,29/2165) -- (433/145,29/2165);
      \draw[dashed,samples=100,domain=0.013:3.3,variable=\x]
      plot ({\x},{1/(25*\x)});
      \draw (2.9,0.013) circle [x radius=0.1,y radius=0.1];
      \draw (2.8,0.013) -- (2.23,0.3);
      \draw (3,0.013) -- (3.4,0.315);
        \draw[->] (d) -- (2/5,1/10);
      \begin{scope}[scale=14,shift={(-2.7,0.1)}]
        \filldraw[lightgray] (2.8,2/145) -- 
        (29/10,2/145) -- (29/10,29/2165) --
        (433/145,29/2165) -- (433/145,0) --
        (2.8,0) -- cycle;
        \draw (2.8,2/145) -- 
        (29/10,2/145) node {\(\cdot\)} -- (29/10,29/2165) --
        (433/145,29/2165) node {\(\cdot\)};
        \filldraw[lightgray] (2.8,0) -- (2.987,0) --
        (2.987,0.013) -- (2.8,0.013);
        \draw (2.987,0) -- (2.987,0.013);
        \draw[dashed,samples=100,domain=2.8:3,variable=\x]
        plot ({\x},{1/(25*\x)});
        \draw (2.8,0) -- (3,0);
        \draw (2.9,0.013) circle [x radius=0.1,y radius=0.1];
        \node (a) at (29/10,10/145)
        {\(\left(\frac{29}{10},\frac{2}{145}\right)\)};
        \node (b) at (433/145,145/2165) {\(\left(\frac{433}{145},\frac{29}{2165}\right)\)};
        \node (c) at (2.9,-0.05) {\(\sigma_5=2.987\cdots\)};
        \draw[->] (a) -- (29/10,2/145);
        \draw[->] (b) -- (433/145,29/2165);
        \draw[->] (c) -- (2.987,0);
      \end{scope}
    \end{tikzpicture}
    \caption{The staircase for \(p=5\), \(q=1\).}
    \label{fig:staircase_5_1}
  \end{center}
\end{figure}
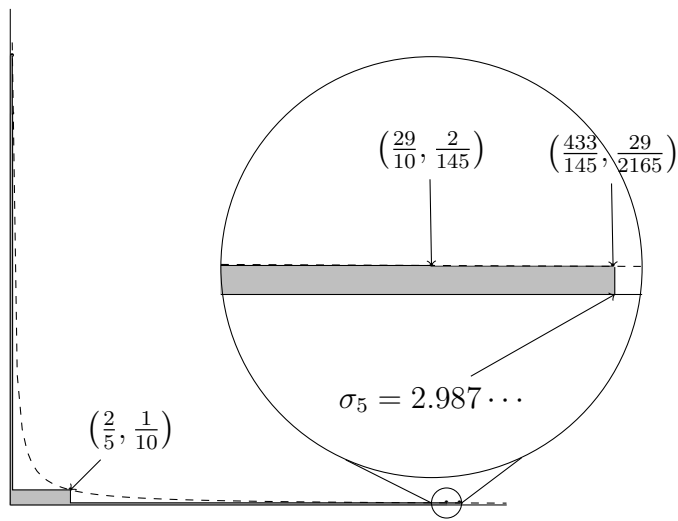

\begin{remark} \label{rk:pqgromovwidth} 
    In particular, for a Markov number $p$ with companion number $q$, Theorem~\ref{thm:Markovstairs} computes
    the largest $(p,q)$-pin-ball that symplectically embeds into $\CP^2$:
    $$
    \sup \{ \alpha > 0 \sth  B_{p,q}(\alpha) \se X \} \,=\, 
    \min\left\{ \frac{p_2}{pp_3} , \frac{p_3}{pp_2} \right\}
    $$ 
    where $(p,p_2,p_3)$ denotes the Markov triple in which $p$ appears as largest entry with companion number~$q$. 
    The only pin-balls admitting volume-filling embeddings into~$\CP^2$ 
    are therefore those with $(p,q)=(1,1)$ and $(p,q)=(2,1)$. 
\end{remark}

\addtocontents{toc}{\protect\setcounter{tocdepth}{1}}
\subsection{Pin-ball packing problem}
\addtocontents{toc}{\protect\setcounter{tocdepth}{2}}

The following result generalizes Gromov's Two Ball Theorem in dimension four.

\begin{theorem}[Two Pin-Ball Theorem]\label{thm:twoball} 
    Let \(p_1,p_2\) be Markov numbers appearing in the same Markov triple\footnote{Whenever such a triple exists, there are precisely two such triples, and they are related by a single mutation.}, 
    and let \(p_3\) be the smaller of the two Markov numbers completing the triple.
    Let \(q_i= \pm 3p_{i+1}p_{i+2}^{-1}\mod p_i\) for \(i=1,2\) 
    be companion numbers.
    Then there exists a symplectic embedding
    \begin{equation}
        \label{eq:twoball}
        B_{p_1,q_1}(\alpha_1) \sqcup B_{p_2,q_2}(\alpha_2) \hookrightarrow \CP^2,
    \end{equation}
    if and only if $\alpha_1 < \frac{p_2}{p_1p_3}$ and $\alpha_2 < \frac{p_1}{p_2p_3}$ and $\alpha_1 + \alpha_2 < \frac{p_3}{p_1p_2}$.
\end{theorem}

\begin{remark}\label{rmk:Twoballobstruction}
    Note that one of the inequalities on the $\alpha_j$s is implied by the bound on $\alpha_1+\alpha_2$. 
    To see this, assume that 
    $p_2 = \max \{ p_1, p_2, p_3 \}$; the case $p_1 = \max \{ p_1, p_2, p_3 \}$ follows by exactly the same argument.
    Then, assuming that \(\alpha_1+\alpha_2<p_3/(p_1p_2)\), we get 
    \[\alpha_1 < \alpha_1 + \alpha_2 < \frac{p_3}{p_1 p_2} \leq \frac{p_2}{p_1 p_3},\]
    which proves the inequality for $\alpha_1$.
    Moreover, Theorem~\ref{thm:Markovstairs} proves that $\alpha_2 < \frac{p_1}{p_2 p_3}$.
    This means that to establish the necessity of the inequalities in Theorem \ref{thm:twoball}, we only need to prove the bound on \(\alpha_1+\alpha_2\).
\end{remark}

\begin{remark}\label{rmk:Twoballconstruction}
    The embeddings in Theorem \ref{thm:twoball} again appear as visible embeddings in Vianna triangles, see Figure~\ref{fig:23ball}. 
    In general, if an embedding of the form \eqref{eq:twoball} exists and $p_1,p_2 \geq 2$, then $p_1$ and $p_2$ are necessarily Markov numbers appearing in the same Markov triple and $q_1$ and~$q_2$ 
    are companion numbers of $p_1$ and~$p_2$, respectively,
    see \cite[Theorem~1.2]{ES}. 
    However if, say, $p_1 = 1$, then $B_{p_1,q_1}(\alpha_1)$ is a standard ball and embeddings of the type \eqref{eq:twoball} exist even if the Markov numbers $1,p_2$ do not appear in the same triple. 
    In that case, constraints on $\alpha_1$ and $\alpha_2$ can be derived by the same methods as in \cite{BrendelSchlenk}.
\end{remark}

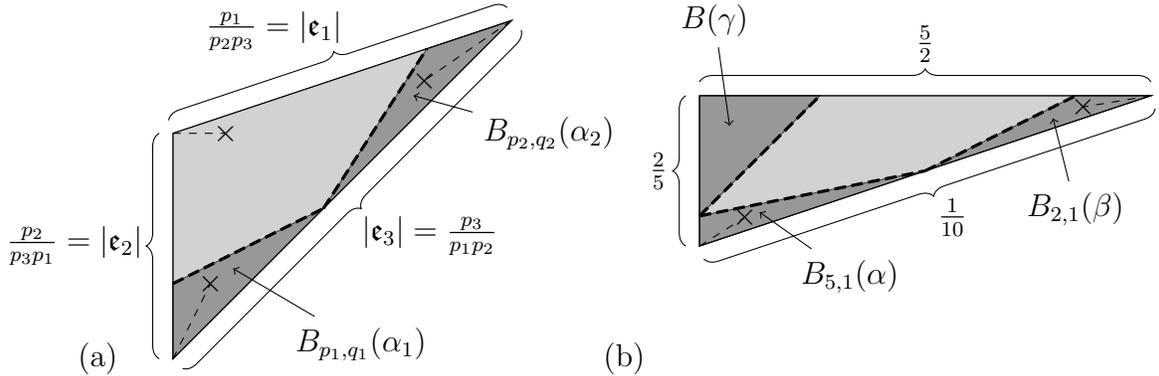
\begin{figure}[htb]
  \begin{center}
    \begin{tikzpicture}
        \filldraw[fill=lightgray,opacity=0.7] (0,0) -- (4.5,4.5) -- (0,3) -- cycle;
        \filldraw[fill=gray,opacity=0.7] (0,0) -- (2,2) -- (0,1) -- cycle;
        \filldraw[fill=gray,opacity=0.7] (2,2) -- (4.5,4.5) -- (4.5-4.5/4,4.5-1.5/4) -- cycle;
        \draw[very thick,dash pattern=on 9/2pt off 4/2pt] (0,1) -- (2,2);
        \draw[very thick,dash pattern=on 9/2pt off 4/2pt] (2,2) -- (4.5-4.5/4,4.5-1.5/4);
        \draw (0,0) -- (4.5,4.5) -- (0,3) -- cycle;
        \draw[dashed] (0,0) -- ++ (0.25*2,0.5*2) node (a) {\(\times\)};
        \draw[dashed] (0,3) -- ++ (0.7,0) node (b) {\(\times\)};
        \draw[dashed] (4.5,4.5) -- ++ (-1.16,-0.8) node (d) {\(\times\)};
        \draw [decorate,decoration={brace,amplitude=5pt,raise=1ex}] (0,0) -- (0,3) node[midway,xshift=-7.2ex]{\(\frac{p_2}{p_3p_1}=|\edge_2|\)};
        \draw [decorate,decoration={brace,amplitude=5pt,raise=1ex}] (4.5,4.5) -- (0,0) node[midway,xshift=6.4ex,yshift=-3.4ex]{\(|\edge_3|=\frac{p_3}{p_1 p_2}\)};
        \draw [decorate,decoration={brace,amplitude=5pt,raise=1ex}] (0,3) -- (4.5,4.5) node[midway,xshift=-5ex,yshift=3.4ex]{\(\frac{p_1}{p_2 p_3}=|\edge_1|\)};
        \node (p) at (2.5,0.2) {\(B_{p_1,q_1}(\alpha_1)\)};
        \node (r) at (5,3) {\(B_{p_2,q_2}(\alpha_2)\)};
        \draw[->] (p) -- (0.8,1.2);
        \draw[->] (r) -- (3.2,3.5);
        \node at (-1,0) {(a)};
        \begin{scope}[shift={(7,1.5)}]
            \filldraw[fill=lightgray,opacity=0.7] (0,0) -- (6,2) -- (0,2) -- cycle;
            \filldraw[fill=gray,opacity=0.7] (0,0.4) -- (1.6,2) -- (0,2) -- cycle;
            \filldraw[fill=gray,opacity=0.7] (0,0) -- (3,1) -- (0,0.4) -- cycle;
            \filldraw[fill=gray,opacity=0.7] (3,1) -- (6,2) -- (5,2) -- cycle;
            \draw[very thick,dash pattern=on 9/2pt off 4/2pt] (0,0.4) -- (1.6,2);
            \draw[very thick,dash pattern=on 9/2pt off 4/2pt] (0,0.4) -- (3,1);
            \draw[very thick,dash pattern=on 9/2pt off 4/2pt] (3,1) -- (5,2);
            \draw (0,0) -- (6,2) -- (0,2) -- cycle;
            \draw[dashed] (0,0) -- ++ (6*0.1,3.8*0.1) node (a) {\(\times\)};
            \draw[dashed] (6,2) -- ++ (-6*0.15,-0.15) node (d) {\(\times\)};
            \draw [decorate,decoration={brace,amplitude=5pt,raise=1ex}] (0,0) -- (0,2) node[midway,xshift=-3ex]{\(\frac{2}{5}\)};
            \draw [decorate,decoration={brace,amplitude=5pt,raise=1ex}] (6,2) -- (0,0) node[midway,xshift=2ex,yshift=-3.4ex]{\(\frac{1}{10}\)};
            \draw [decorate,decoration={brace,amplitude=5pt,raise=1ex}] (0,2) -- (6,2) node[midway,yshift=3.8ex]{\(\frac{5}{2}\)};
            \node (p) at (2,-0.4) {\(B_{5,1}(\alpha)\)};
            \node (r) at (5,0.5) {\(B_{2,1}(\beta)\)};
            \node (q) at (0.2,3) {\(B(\gamma)\)};
            \draw[->] (p) -- (0.8,0.45);
            \draw[->] (r) -- (4.7,1.7);
            \draw[->] (q) -- (0.4,1.6);
            \node at (-1,-1.5) {(b)};
        \end{scope}
    \end{tikzpicture}
    \caption{(a) The Vianna triangle \(\Vianna(p_1,p_2,p_3)\) and an illustration of the constructive side of the Two Pin-Ball Theorem~\ref{thm:twoball}. 
    (b) The Vianna triangle \(\Vianna(5,2,1)\) and the pin-balls for the constructive side of the Two Pin-Ball Theorem for $\{p_1,p_2\} = \{1,5\}$ and for $\{p_1,p_2\} = \{2,5\}$
    and of the Three Pin-Ball Theorem~\ref{cor:threeball}.}
    \label{fig:23ball}
  \end{center}
\end{figure}

\begin{corollary}[Three Pin-Ball Theorem]\label{cor:threeball}
    Let $(p_1,p_2,p_3)$ be a Markov triple with companion numbers $q_1,q_2,q_3$. 
    Then there exists a symplectic embedding 
    \begin{equation}
        B_{p_1,q_1}(\alpha_1) \sqcup B_{p_2,q_2}(\alpha_2) \sqcup B_{p_3,q_3}(\alpha_3) \hookrightarrow \CP^2,
    \end{equation}
    if and only if $\alpha_1 + \alpha_2 < \frac{p_3}{p_1p_2}$, $\alpha_1 + \alpha_3 < \frac{p_2}{p_1p_3}$ and $\alpha_2 + \alpha_3 < \frac{p_1}{p_2p_3}$.
\end{corollary}

\begin{proof} 
If $p_1=p_2=p_3 =1$, the claim follows directly from Theorem~\ref{thm:twoball}
(or also from Gromov's 2-ball theorem).
We can therefore assume that $p_3 \leq p_2 < p_1$.
By Theorem~\ref{thm:twoball}, 
$$
\ga_1+\ga_2 < \frac{p_3}{p_1p_2}.
$$
If we mutate $(p_1,p_2,p_3)$ at $p_2$, then $p_2' = 3p_1p_3-p_2 > p_2$.
Hence Theorem~\ref{thm:twoball} applied to $p_1,p_3$ yields
$$
\ga_1+\ga_3 < \frac{p_2}{p_1p_3}.
$$
Finally, if we mutate $(p_1,p_2,p_3)$ at $p_1$, then $p_1' = 3p_2p_3-p_1 <p_1$.
Hence Theorem~\ref{thm:twoball} applied to $p_2,p_3$ yields
\[
\ga_2+\ga_3 < \frac{p_1'}{p_2p_3} < \frac{p_1}{p_2p_3}.
\qedhere\]
\end{proof}

\begin{remark}
    Note that, by \cite[Theorem 1.2]{ES}, it is impossible to pack more than three pin-balls with \(p_i\geq 2\).
\end{remark}

\addtocontents{toc}{\protect\setcounter{tocdepth}{1}}
\subsection{Idea of proof for Theorem \ref{thm:uniqueness}} 
\addtocontents{toc}{\protect\setcounter{tocdepth}{2}}

Our proof builds on the idea of Hind~\cite{Hi10} and Li--Wu~\cite[\S\,6.1.1]{LiWu12}
who showed the Lagrangian uniqueness of~$\RP^2$ in~$\CP^2$ by turning this problem
into the uniqueness problem of symplectic $-4$-spheres in $S^2 \times S^2$
obtained by a symplectic cut on the boundary of small codisc-bundle of an embedded~$\RP^2$ in~$\CP^2$.
Suppose we have two symplectic embeddings \(\iota_1,\iota_2\colon E_{p,q}(\alpha,\beta)\hookrightarrow \CP^2\). 
We shall perform a rational blow-up 
along each of these pin-ellipsoids; this yields two 4-manifolds \(\til{X}_1\) and \(\til{X}_2\) containing configurations \(\mathcal{C}_1\) and \(\mathcal{C}_2\) of embedded symplectic spheres.
We show that, for \(i=1,2\), \(\til{X}_i\) is ruled, and that:
\begin{itemize}
    \item in each case, one of the curves from \(\mathcal{C}_i\) is a section of this ruled surface,
    
    \item the other curves of \(\mathcal{C}_i\) (together with some additional curves) form one or two broken rulings.
    
    \item the configuration \(\mathcal{T}_i\) consisting of \(\mathcal{C}_i\) and these additional curves (and possibly one or two additional smooth rulings) is completely determined by \(p,q\) and the choice of data of the rational blow-up. In other words \(\mathcal{T}_1\) and \(\mathcal{T}_2\) have symplectomorphic neighbourhoods \(\nu_1\cong \nu_2\).
    
    \item the common boundary \(\partial\nu_i\) is contactomorphic to \(S^1\times S^2\) and the complement of \(\nu_i\) in \(\til{X}_i\) is minimal. In particular, uniqueness of symplectic fillings of these manifolds tells us that the symplectic completion of \(\til{X}_i\setminus\nu_i\) is symplectomorphic to either~$\RR^4$ or~\(T^*S^1\times \CC\).
\end{itemize}
Since \(\til{X}_1\) and \(\til{X}_2\) are essentially obtained by gluing together symplectomorphic pieces, and there is only one way to glue these together, we find that \(\til{X}_1\) and \(\til{X}_2\) are symplectomorphic via a symplectomorphism taking \(\mathcal{C}_1\) to \(\mathcal{C}_2\).
This descends to give a symplectomorphism of the rational blow-down \(\CP^2\to\CP^2\) taking one pin-ellipsoid to the other. 
Finally, this symplectomorphism is Hamiltonian due to Gromov's classical result stating that the symplectomorphism group of~$\CP^2$ is connected.

\addtocontents{toc}{\protect\setcounter{tocdepth}{1}}
\subsection{Idea of proof for Theorems \ref{thm:Markovstairs} and \ref{thm:twoball}} 
\addtocontents{toc}{\protect\setcounter{tocdepth}{2}}

To explain the proof of the Staircase Theorem (Theorem~\ref{thm:Markovstairs}), let $p$ be a Markov number and $q$ a companion number. As mentioned in 
Remark~\ref{rmk:staircases}~(b), the points in $\stair(p,q)$ are realised by visible embeddings coming from Vianna triangles. 
It is therefore enough to show that for all $i \in \Z$, $E_{p,q}\left(\alpha,\beta\right)$ does not symplectically embed into $\CP^2$ whenever $\alpha > \frac{m_i}{pm_{i-1}}$ and $\beta > \frac{m_{i}}{pm_{i+1}}$. Here $m_i$ denote the members of the sequence defined in~\eqref{eq:mi}. 
Note that these points correspond to the inner corners 
of the staircase~$\stair(p,q)$.

Assuming that such an embedding $\iota: E_{p,q}(\alpha,\beta)\hookrightarrow \CP^2$ exists, we construct a one-parameter family of symplectic manifolds $(\til{X},\til{\omega}_t)$ by rational symplectic blow-up\footnote{Technically, we need to take a blow-up of the rational blow-up which we refer to as a {\em pavilion blow-up}.} 
of a small neighbourhood of 
$\iota\left(t E_{p,q}\left(\alpha,\beta\right) \right)$, 
where the parameter~$t \in (0,1]$ measures the size of the blow-up. 
For small enough~$t$, we can assume, by Theorem \ref{thm:uniqueness}, that the embedding coincides with the visible embedding coming from the Vianna triangle $\Vianna(p,m_{i-1},m_i)$. 
In that case, we construct a \emph{visible}\/ symplectic $-1$-sphere in the base diagram of $(\til{X},\til{\omega}_t)$, 
see Figure~\ref{fig:visible_sphere}.\footnote{In the case $p=1$ this curve corresponds to the unicuspidal curve in~$\CP^2$ used in~\cite[Figure~4.2.1]{McDuffSiegel24}, but here it is found directly in~$\widetilde X$ and hence it is a smoothly embedded $-1$-curve.}
Hence we find a class $S \in H_2(\til{X};\ZZ)$ whose Gromov invariant is one.
Furthermore, homological computations show that the symplectic area of~$S$ as measured by~$\til{\omega}_t$ becomes negative for~$t$ close enough to~$1$.
However, the Gromov invariant is a deformation invariant, 
meaning that $S$ should be represented by a holomorphic curve 
for all~$\til{\omega}_t$. If $\alpha > \frac{m_i}{pm_{i-1}}$ and $\beta > \frac{m_{i}}{pm_{i+1}}$ for some \(i\) then we show that \(t\) can be chosen large enough that the symplectic area of \(S\) is negative.
Since holomorphic curves have positive area, we obtain a contradiction.

This idea of using visible curves to obstruct embeddings of standard ellipsoids can be found in McDuff--Siegel {\cite[Figure 4.2.1 and Corollary 6.1.4]{McDuffSiegel24}}, so the main new input here is the Isotopy Theorem. 
Note that in the proof we only use a weak version of the Isotopy Theorem~\ref{thm:uniqueness},
namely that two pin-ellipsoid embeddings can be made to agree on an arbitrarily small pin-subellipsoid. 
For symplectic embeddings of balls ($p=1$, where the pin-wheel becomes a point) this is an elementary fact; 
as a consequence, the proof of the Fibonacci staircase (in the visible range, below the aspect ratio \(\phi^4\)) is immediate from almost toric geometry: 
the optimal embeddings are visible and the obstructions come from visible curves. This is implicit in~\cite{McDuffSiegel24}.

The proof of the Two Pin-Ball Theorem (Theorem~\ref{thm:twoball}) also comes from positivity of area for a pseudoholomorphic curve, and can be seen as an elaborate version of Gromov's proof of his Two Ball Theorem.
Suppose we have disjoint symplectically embedded pin-balls \(B_{p_1,q_1}(\alpha_1)\) and \(B_{p_2,q_2}(\alpha_2)\) 
in~\(\CP^2\). In this case, we rationally blow-up both pin-balls; this produces a manifold~\(\til{X}\) containing two chains of symplectic spheres.
We find a symplectic \(-1\)-sphere~\(\til{C}\) in~\(\til{X}\) which connects these chains and whose symplectic area is bounded by \(\frac{p_3}{p_1p_2}\) where \(p_3\) is the smaller Markov number making \((p_1,p_2,p_3)\) a Markov triple.
Contracting the chains yields an orbifold with two orbifold points, and the image of~\(\til{C}\) passes through both of them.
We then lift to the uniformising cover of each orbifold chart and apply the monotonicity formula for areas of minimal surfaces to get a lower bound of \(\alpha_1+\alpha_2\) on the symplectic area of~\(\til{C}\). 
In the case when the pin-balls are simultaneously visible, the sphere~\(\til{C}\) is also visible as part of the toric boundary; see Figure~\ref{fig:twoballcurve}.
One could prove the existence of~\(\til{C}\) in general by proving an isotopy theorem for pairs of pin-balls, but to save space we instead appeal to the paper of Evans and Smith {\cite[Theorem 4.16]{ES}}.

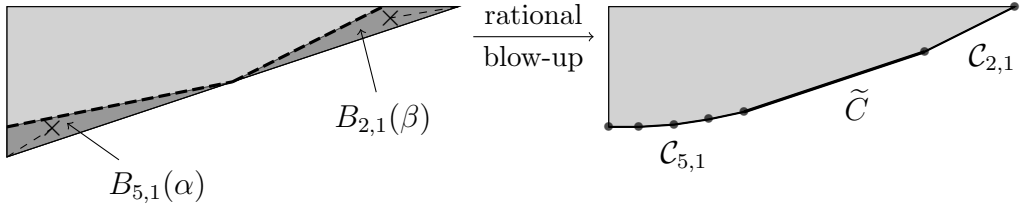
\begin{figure}[htb]
  \begin{center}
    \begin{tikzpicture}
            \filldraw[fill=lightgray,opacity=0.7] (0,0) -- (6,2) -- (0,2) -- cycle;
            \filldraw[fill=gray,opacity=0.7] (0,0) -- (3,1) -- (0,0.4) -- cycle;
            \filldraw[fill=gray,opacity=0.7] (3,1) -- (6,2) -- (5,2) -- cycle;
            \draw[very thick,dash pattern=on 9/2pt off 4/2pt] (0,0.4) -- (3,1);
            \draw[very thick,dash pattern=on 9/2pt off 4/2pt] (3,1) -- (5,2);
            \draw (0,0) -- (6,2) -- (0,2) -- cycle;
            \draw[dashed] (0,0) -- ++ (6*0.1,3.8*0.1) node (a) {\(\times\)};
            \draw[dashed] (6,2) -- ++ (-6*0.15,-0.15) node (d) {\(\times\)};
            \node (p) at (2,-0.4) {\(B_{5,1}(\alpha)\)};
            \node (r) at (5,0.5) {\(B_{2,1}(\beta)\)};
            \draw[->] (p) -- (0.8,0.45);
            \draw[->] (r) -- (4.7,1.7);
            \begin{scope}[shift={(8,0)}]
                \filldraw[fill=lightgray,opacity=0.7] (0,0.4) node {\(\sbullet\)} -- (0.4,0.4) node {\(\sbullet\)} -- (0.87,0.43) node {\(\sbullet\)} -- (1.33,0.5) node {\(\sbullet\)} -- (1.8,0.6) node {\(\sbullet\)} -- (6*0.7,2*0.7) node {\(\sbullet\)} -- (6*0.7+1.2,2) node {\(\sbullet\)} -- (0,2) -- cycle;
                \draw (0,0.4) -- (0.4,0.4) -- (0.87,0.43) -- (1.33,0.5) -- (6*0.3,2*0.3) -- (6*0.7,2*0.7) -- (6*0.7+1.2,2) -- (0,2) -- cycle;
                \draw[thick] (0,0.4) -- (0.4,0.4) -- (0.87,0.43) -- (1.33,0.5) -- (6*0.3,2*0.3) -- (6*0.7,2*0.7) -- (6*0.7+1.2,2);
                \draw[very thick] (6*0.3,2*0.3) -- (6*0.7,2*0.7);
                \node at (6*0.5+0.3,2*0.5-0.3) {$\til{C}$};
                \node at (1,0) {$\cc_{5,1}$};
                \node at (5.1,1.3) {$\cc_{2,1}$};
            \end{scope}
    \draw[->] (6.2,1.6) to (7.8,1.6);
    \node at (7,1.9) {\small rational};
    \node at (7,1.3) {\small blow-up};
    \end{tikzpicture}
    \caption{$\widetilde C$ is the exceptional curve used to prove the obstructive side of the Two Pin-Ball Theorem~\ref{thm:twoball}. The situation above is the visible situation for the $(5,2)$-case. The Wahl chains $\cc_{5,1}$ and $\cc_{2,1}$ are connected by the exceptional curve $\widetilde{C}$. In general we will find the curve~$\til{C}$ by a neck stretching procedure.}
    \label{fig:twoballcurve}
  \end{center}
\end{figure}

\addtocontents{toc}{\protect\setcounter{tocdepth}{1}}
\subsection{Outline of the paper}
\addtocontents{toc}{\protect\setcounter{tocdepth}{2}}

In Section \ref{sct:notation}, we introduce the basic constructions from almost toric geometry which are needed for the rest of the paper, including the Vianna triangles (almost toric base diagrams for \(\CP^2\)) which allow us to see the visible pin-ellipsoids, the {\em pavilion blow-up} (a mild generalisation of rational blow-up) and the {\em culet curve} (a distinguished curve in the pavilion blow-up of a pin-ellipsoid in \(\CP^2\)).
In Section~\ref{sct:rat_pav}, we explain in detail how to equip the pavilion blow-up with a symplectic structure and compatible almost complex structure and how to perform computations in the homology after blowing up.
In Section~\ref{sct:ruled_manifolds}, we collect some basic results about ruled symplectic 4-manifolds which will be used in the main argument.
Section~\ref{sct:pin_in_p2} contains the proofs of the main results: 
\begin{itemize}
    \item The main technical results are: the existence of a square zero holomorphic sphere in the rational blow-up 
    (Theorem~\ref{thm:input_from_es}, proved in Section~\ref{subsec:ES}) and the description of how this sphere can degenerate to give broken rulings (Theorem \ref{thm:regulation_properties}, proved in Section~\ref{subsec:regulation_properties}).
    
    \item The proofs of the Isotopy Theorem \ref{thm:uniqueness} and the Staircase Theorem~\ref{thm:Markovstairs}, assuming these technical results, are given in Section~\ref{sct:isotopy_proof} and~\ref{subsec:staircase_proof} respectively.
    
    \item The Two Pin-Ball Theorem (Theorem \ref{thm:twoball}) is proved in Section~\ref{sct:twoball}.
\end{itemize}

\subsection*{Acknowledgements}

The authors would like to thank Gerard Bargall\'{o} i G\'{o}mez, Marco Golla, Richard Hind, Paul Levy, Grisha Mikhalkin, Dusa McDuff, Federica Pasquotto, George Politopoulos, Joel Schmitz, Kyler Siegel, Laura Starkston, Giancarlo Urz\'{u}a, Renato Vianna and Weiyi Zhang
for helpful conversations.

JB is supported by SNSF Ambizione Grant PZ00P2-223460.

JE is supported by EPSRC Standard Grant EP/W015749/1.

\section{Notation and definitions}
\label{sct:notation}

\subsection{Wahl singularities and their resolutions}

\begin{definition}[Triangle \(\Tri_{p,q}(\alpha,\beta)\)]
  If we ignore the node and branch cut for the almost toric base
  diagram \(\ATF_{p,q}(\alpha,\beta)\) from Definition
  \ref{def:Epq} then we get a triangle (see Figure
  \ref{fig:tri_pq}):
  \begin{equation}\label{eq:tri_pq}\Tri_{p,q}(\alpha,\beta)\coloneqq \left\{
      x\in\RR^2 \mid \rho_0\cdot x\geq 0,\quad\rho_{m+1}\cdot x\geq
      0,\quad \gamma\cdot x \geq - \alpha\beta p^2\right\}\end{equation}
  where
  \(\rho_0\), \(\rho_{m+1}\) and \(\gamma\) are the vectors
  \[\rho_0=(1,0),\quad \rho_{m+1}=(1-pq,p^2),\quad \mbox{and}\quad
    \gamma=(\alpha pq-\alpha-\beta,-\alpha p^2).\]
  We use the index \(m+1\) here because we will soon add further edges with inward normals \(\rho_1,\ldots,\rho_m\); see Definition~\ref{dfn:pavilion}.
\end{definition}

\begin{definition}[Girdled orbifold \(\Orb_{p,q}(\alpha,\beta)\)]
    This triangle is the moment map image of a compact toric
    orbifold \(\Orb_{p,q}(\alpha,\beta)\) with boundary. As in
    Remark \ref{rmk:pin_ellipsoids}(a), the top side of
    \(\Tri_{p,q}(\alpha,\beta)\) with inward normal~\(\gamma\) is
    considered to be part of the boundary, not the toric
    boundary\footnote{The boundary has real codimension \(1\); the toric boundary is a divisor, with complex codimension
    \(1\).}. 
    We refer to the top side as the {\em
    girdle}\footnote{The terminology we introduce in this paper
    (girdle, pavilion, culet) is standard for cut gems
    \cite{encycbrit}.} and we call this bounded toric orbifold
    the {\em girdled orbifold}.
\end{definition}

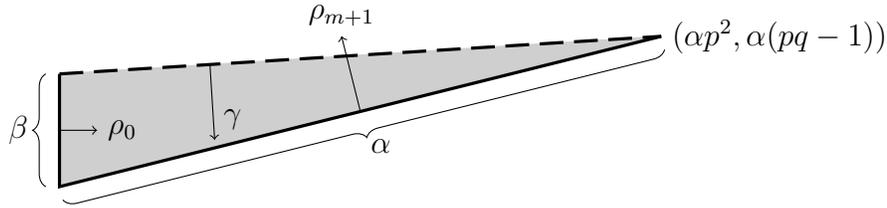
\begin{figure}[htb]
  \begin{center}
    \begin{tikzpicture}
      \filldraw[lightgray,opacity=0.75] (0,1.5) -- (0,0) -- (8,2) -- cycle;
      \draw[very thick] (0,1.5) -- (0,0) -- (8,2);
      \draw [decorate,decoration={brace,amplitude=5pt,raise=1ex}] (0,0) -- (0,1.5) node[midway,xshift=-3ex]{\(\beta\)};
      \draw [decorate,decoration={brace,amplitude=5pt,raise=1ex}] (8,2) -- (0,0) node[midway,xshift=1.5ex,yshift=-2.5ex]{\(\alpha\)};
      \node at (8,2) [right] {\((\alpha p^2,\alpha(pq-1))\)};
      \draw[very thick,dash pattern=on 9pt off 4pt] (0,1.5) -- (8,2);
      \draw[->] (0,0.75) -- (0.5,0.75) node [right] {\(\rho_0\)};
      \draw[->] (4,1) -- ++ (-0.25,1) node [above] {\(\rho_{m+1}\)};
      \draw[->] (2,1.625) -- ++ (0.0625,-1);
      \node at (2.3,0.9) {\(\gamma\)};
    \end{tikzpicture}
    \caption{The moment image \(\Tri_{p,q}(\alpha,\beta)\) of a
      toric orbifold. The toric boundary comprises
      the two solid edges, {\em not} the top side; the orbifold point is at the vertex
      of the polygon. We have also labelled the inward normals~\(\rho_0\) and \(\rho_{m+1}\) of the toric boundary and \(\gamma\) of the top side.}
    \label{fig:tri_pq}
  \end{center}
\end{figure}

The girdled orbifold \(\Orb_{p,q}(\alpha,\beta)\) has a
singularity \(x\) living over the vertex of
\(\Tri_{p,q}(\alpha,\beta)\), locally modelled on the cyclic
quotient singularity \(\frac{1}{p^2}(1,pq-1)\). This is the
quotient of \(\CC^2\) by the action of
\(\bm{\mu}_{p^2}=\{\xi\in\CC^\times \mid \xi^{p^2}=1\}\) where
\(\xi\cdot(x,y)=(\xi x,\xi^{pq-1} y)\). Singularities of this
form are called {\em Wahl singularities}. These are precisely
the cyclic quotient singularities which admit a
\(\QQ\)-Gorenstein smoothing of Milnor number zero. The Milnor
fibre of this smoothing is a pin-ellipsoid, so whenever a smooth
surface admits a \(\QQ\)-Gorenstein degeneration with Wahl
singularities, we can find symplectically embedded
pin-ellipsoids in the smooth fibres.

\begin{definition}[Resolutions]
    \label{dfn:resolution}
    A {\em resolution} of \(\Orb_{p,q}(\alpha,\beta)\) is a
    smooth complex manifold \((\til{U},\til{J})\) together with
    a holomorphic map
    \(\pi\colon\til{U}\to \Orb_{p,q}(\alpha,\beta)\) satisfying
    the following property. There exists a configuration \(\mC\)
    of complex curves \(C\subset\til{U}\) such that \(\pi(C)=x\)
    for all \(C\in\mC\) and such that \(\pi\) restricts to a biholomorphism
    \(\til{U}\setminus\bigcup_{C\in\mC} C\to
    \Orb_{p,q}(\alpha,\beta)\setminus\{x\}\). We say that the resolution is {\em
      chain-shaped} if the curves \(C\in\mC\) are embedded
    spheres and form a chain, that is,
    \(\mC=\{C_1,C_2,\ldots,C_m\}\) with
\[C_i\cdot C_j=\begin{cases}1&\mbox{ if }|i-j|=1\\0&\mbox{
      otherwise.}\end{cases}\]
\end{definition}

\begin{definition}[Girdled resolution \(\til{U}_{\bm{\rho}}\)]
  There is a standard family of chain-shaped resolutions of
  \(\Orb_{p,q}(\alpha,\beta)\) which are themselves toric. They
  arise from subdivisions of the inward normal fan for the
  girdled orbifold. Recall that the inward normals to the two
  edges of \(\Tri_{p,q}(\alpha,\beta)\) are \(\rho_0\) and
  \(\rho_{m+1}\). A subdivision \(\bm{\rho}\) is a sequence of
  primitive integer vectors \(\rho_1,\ldots,\rho_m\) in the
  convex sector bounded by \(\rho_0\) and \(\rho_{m+1}\) with
  the property that \(\rho_i\wedge \rho_{i+1}\geq 1\) for
  \(i=0,1,\ldots,m\). Here, \(v\wedge w\) denotes the
  determinant of the 2-by-2 matrix with columns \(v\) and
  \(w\). If the subdivision \(\bm{\rho}\) satisfies \(\rho_i\wedge \rho_{i+1}= 1\) the subdivided fan \(\bm{\rho}\) determines the
  resolution as a normal toric variety, see for example
  {\cite[Section~2.6]{Fulton}}. Since we are resolving a girdled
  orbifold, we will call this the {\em girdled resolution}
  \(\til{U}_{\bm{\rho}}\) associated to~\(\bm{\rho}\).
\end{definition}

A general prescription for finding resolutions of a
cyclic quotient singularity~$\frac 1n (a,1)$ works as follows. The exceptional set of the
minimal resolution is a chain of holomorphic spheres
\(C_1,\ldots,C_m\) with \(C_i^2=-b_i\) where
\([b_1,\ldots,b_m]\) is the Hirzebruch--Jung continued fraction
expansion of \(n/a\), that is:
\[[b_1,\ldots,b_m]\coloneqq b_1-
\frac{1}{b_2 - \frac{1}{\ddots - \frac{1}{b_m}}}.\]

\begin{remark}\label{rmk:HJ}
  Girdled resolutions all \emph{dominate}\/ a unique 
  {\em minimal resolution}, meaning that they are obtained from it
  by a sequence of complex blow-ups. For the minimal resolution,
  the curves \(C_i\) are embedded spheres with self-intersection
  \(C_i^2=-b_i\) where \([b_1,\ldots,b_m]\) is the
  Hirzebruch--Jung continued fraction expansion of
  \(p^2/(pq-1)\). Non-minimal resolutions are obtained from the
  minimal one by blowing up at intersection points between the
  exceptional spheres \(C_i\).
\end{remark}

\begin{definition}[Wahl chains and duals]
  In the special case of Wahl singularities, we call
  \([b_1,\ldots,b_m]\) a {\em Wahl chain}. Recall that the {\em
    dual} of the singularity \(\frac{1}{n}(1,a)\) is the
  singularity \(\frac{1}{n}(1,\overline{a})\) where
  \(a\overline{a}=1\mod n\); the continued fraction of
  \(n/\overline{a}\) is \([b_m,\ldots,b_1]\), the {\em dual} of
  the chain for \(n/a\).
\end{definition}

\begin{definition}[Pavilion]\label{dfn:pavilion}
    Let \(\bm{\rho}\) be a subdivision and, for each \(\rho_i\in\bm{\rho}\), choose a positive real number \(\lambda_i>0\).
    We then consider the polygon 
    \begin{equation}
    \label{eq:pavilion_def}
    \Pav_{\bm{\rho},\bm{\lambda}}\left(\Tri_{p,q}(\alpha,\beta)\right)\coloneqq
    \left\{x\in \Tri_{p,q}(\alpha,\beta) \mid \rho_i \cdot x \geq
      \lambda_i\mbox{ for }i=1,\ldots,m\right\},
    \end{equation}
    see Figure~\ref{fig:pavilion}.
    We require that \(\rho_i \cdot x>\lambda_i\) for all points
    \(x\) on the girdle (recall that this is the top side of the
    moment triangle). A choice of \(\bm{\rho}\) and
    \(\bm{\lambda}\) satisfying this requirement is called a {\em
    pavilion}. If, moreover, for every \(i=1,2,\ldots,m\) we have that (a) \(\rho_i\wedge \rho_{i+1}= 1\) and (b) there
    is an edge \(\ff_i\) of the polygon with inward normal~\(\rho_i\) and positive length, then we call the pavilion {\em Delzant}.
    We call a pavilion {\em minimal}\/ if the corresponding toric girdled 
    resolution~\(\til{U}_{\bm{\rho}}\) is the minimal resolution.
\end{definition}

\begin{figure}[htb]
  \begin{center}
    \begin{tikzpicture}
      \draw[very thick,dash pattern=on 9pt off 4pt] (0,2) -- (5.2,2.6);
      \draw[draw=none] (0,2) -- (5.1,2.6) node[midway,sloped,above] {girdle};
      \filldraw[fill=lightgray] (0,2) -- (0,0) -- (5.2,2.6);
      \draw[very thick] (0,2) -- (0,0) -- (5.2,2.6);
      \draw[->,very thick] (0,0) -- (2,1) node [above] {\(v\)};
      \draw[->,very thick] (0,0) -- (0,1) node [right] {\(w\)};
      \draw [decorate,decoration={brace,amplitude=5pt,raise=1ex}] (0,0) -- (0,2) node[midway,xshift=-3ex]{\(\beta\)};
      \draw [decorate,decoration={brace,amplitude=5pt,raise=1ex}] (5.2,2.6) -- (0,0) node[midway,xshift=1.5ex,yshift=-2.5ex]{\(\alpha\)};
      \node at (-1,3) {(a)};
      \begin{scope}[shift={(7,0)}]
        \filldraw[pattern=north west lines] (0,2) -- (0,0) -- (5.2,2.6);
        \draw[very thick,dash pattern=on 9pt off 4pt] (0,2) -- (5.2,2.6);
        \filldraw[fill=lightgray] (0,2) -- (0,1.5) -- (1,1.5) -- (4,2) -- (5.2,2.6);
        \draw[very thick] (0,2) -- (0,1.5) -- (1,1.5) -- (4,2) -- (5.2,2.6);
        \node[ellipse,inner sep=0pt,fill=white] at (0.5,1.5) [below] {\(\ff_1\)};
        \node[ellipse,inner sep=0pt,fill=white] at (2,1.65) [below] {\(\ff_2\)};
        \draw[->,very thick] (0,1.5) -- (1,1.5) node [above left] {\(u_1\)};
        \draw[->,very thick] (1,1.5) -- ++ (3/2,1/4) node [above left] {\(u_2\)};
        \draw[dotted,thick] (-1,0) -- (0,0) node [below,pos=0.1,sloped] {\(\rho_1=0\)};
        \draw[dotted,thick] (-1,1.5) -- (0,1.5) node [below,pos=0.1,sloped] {\(\rho_1=\lambda_1\)};
        \draw[dotted,thick] (0,0) -- (6,1) node [below,pos=0.9,sloped] {\(\rho_2=0\)};
        \draw[dotted,thick] (4,2) -- ++ (1.8*6/6,1.8*1/6) node [below,pos=0.8,sloped] {\(\rho_2=\lambda_2\)};
        \node at (-1,3) {(b)};
      \end{scope}
    \end{tikzpicture}
    \caption{(a): \(\Tri_{p,q}(\alpha,\beta)\) with its girdle. 
             (b): A pavilion
      \(\Pav_{\bm{\rho},\bm{\lambda}}\left(\Tri_{p,q}(\alpha,\beta)\right)\)
      for \(\Tri_{p,q}(\alpha,\beta)\).}
    \label{fig:pavilion}
  \end{center}
\end{figure}
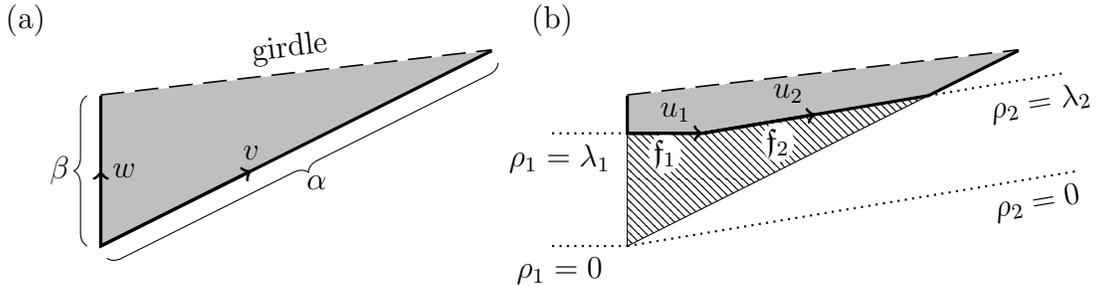

If we choose a Delzant pavilion and a strictly convex function
\(\psi\) on it (with suitable behaviour along the edges of the
polygon) then we get a toric K\"{a}hler metric on the resolution
constructed using the Hessian of \(\psi\) (see {\cite[Eq.\
  (2.1)]{AbreuToricKaehler}}; the boundary behaviour along edges
is obtained from {\cite[Eq.\ (2.8)]{AbreuToricKaehler}} by
taking Legendre transform). Therefore,
\(\Pav_{\bm{\rho},\bm{\lambda}}\left(\Tri_{p,q}(\alpha,\beta)\right)\)
is the moment image for a toric K\"{a}hler structure on the
girdled resolution \(\til{U}_{\bm{\rho}}\). We write
\(\til{\omega}_{\bm{\rho},\bm{\lambda},\psi}\) for the
K\"{a}hler form.

\begin{definition}[Pavilion blow-up]\label{dfn:pav_local}
  We call the symplectic manifold
  \(\left(\til{U}_{\bm{\rho}},\til{\omega}_{\bm{\rho},\bm{\lambda},\psi}\right)\)
  the {\em pavilion blow-up of \(E_{p,q}(\alpha,\beta)\)}.
\end{definition}

\begin{remark}
    The specific choices of \(\alpha\), \(\beta\) and
    \(\bm{\lambda}\) will be important for some of the
    quantitative questions we consider.
    The freedom to choose
    non-minimal resolutions will help with finding holomorphic
    curves which give obstructions to pin-ellipsoid embeddings.
\end{remark}

\begin{remark}
  Although we construct the symplectic form in this way, one could also think of the pavilion blow-up as coming from a sequence of symplectic cuts. This would have the disadvantage that different choices of \(\bm{\lambda}\) would yield symplectic forms {\em on different manifolds} (all diffeomorphic, but still different) making it harder to compare almost complex structures. Nonetheless, we will sometimes abuse language and talk about ``making a pavilion cut''; we just mean picking a particular \(\bm{\rho}\) and \(\bm{\lambda}\).
\end{remark}

\subsection{Offcuts}

\begin{definition}
  Given a pavilion \(\bm{\rho},\bm{\lambda}\) for
  \(\Tri_{p,q}(\alpha,\beta)\), we can think of the pavilion as a
  subset of the almost toric base diagram
  \(\ATF_{p,q}(\alpha,\beta)\), and nodally slide so that the
  focus-focus singularity lies below
  \(\Pav_{\bm{\rho},\bm{\lambda}}\left(\Tri_{p,q}(\alpha,\beta)\right)\). Define
  the {\em offcut}
  \[\Off_{\bm{\rho},\bm{\lambda}}\left(\ATF_{p,q}(\alpha,\beta)\right)\coloneqq
    \ATF_{p,q}(\alpha,\beta) \setminus
    \Pav_{\bm{\rho},\bm{\lambda}}\left(\Tri_{p,q}(\alpha,\beta)\right).\]
  We call the associated almost toric domain the {\em almost
    toric offcut}
  \(\Off_{\bm{\rho},\bm{\lambda}}\left(E_{p,q}(\alpha,\beta)\right)\); see
  Figure \ref{fig:offcut}.
\end{definition}

\begin{figure}[htb]
  \begin{center}
    \begin{tikzpicture}
      \begin{scope}[shift={(0,0.5)}]
        \node at (-1,3) {(a)};
        \node (n1) at (2,3) {\(\ATF_{p,q}(\alpha,\beta)\)};
        \draw[very thick,dashed] (0,2) -- (5.2,2.6);
        \draw [dashed] (0,2) -- (5.2,2.6);
        \filldraw[fill=lightgray] (0,2) -- (0,0) -- (5.2,2.6);
        \draw[very thick] (0,2) -- (0,0) -- (5.2,2.6);
        \draw[->,very thick] (0,0) -- (2,1);
        \draw[->,very thick] (0,0) -- (0,1);
        \draw[dashed] (0,0) -- (1,1) node {\(\times\)};
        \node at (1,1) [above] {\((p,q)\)};
        \draw [decorate,decoration={brace,amplitude=5pt,raise=1ex}] (0,0) -- (0,2) node[midway,xshift=-3ex]{\(\beta\)};
        \draw [decorate,decoration={brace,amplitude=5pt,raise=1ex}] (5.2,2.6) -- (0,0) node[midway,xshift=1.5ex,yshift=-2.5ex]{\(\alpha\)};
      \end{scope}
      \begin{scope}[shift={(7,-2)}]
        \node (n2) at (2,2.5) {\(\Off_{\bm{\rho},\bm{\lambda}}\left(\ATF_{p,q}(\alpha,\beta)\right)\)};
        \filldraw[fill=lightgray,draw=none] (0,1.5) -- (0,0) -- (4,2) -- (1,1.5) -- (0,1.5);
        \draw[very thick] (0,1.5) -- (0,0) -- (4,2);
        \draw[dashed] (0,0) -- (1,1) node {\(\times\)};
        \draw[dashed] (0,1.5) -- (1,1.5) -- (4,2);
        \node at (-1,2.5) {(c)};
      \end{scope}
      \begin{scope}[shift={(7,0.5)}]
        \node at (-1,3) {(b)};
        \node (n3) at (2,3) {\(\Pav_{\bm{\rho},\bm{\lambda}}\left(\Tri_{p,q}(\alpha,\beta)\right)\)};
        \draw[very thick,dashed] (0,2) -- (5.2,2.6);
        \filldraw[fill=lightgray] (0,2) -- (0,1.5) -- (1,1.5) -- (4,2) -- (5.2,2.6);
        \draw[very thick] (0,2) -- (0,1.5) -- (1,1.5) -- (4,2) -- (5.2,2.6);
      \end{scope}
    \end{tikzpicture}
    \caption{(a): The almost toric base diagram
      \(\ATF_{p,q}(\alpha,\beta)\). (b): The moment image of the
      pavilion blow-up \(\til{U}_{\bm{\rho}}\). (c): The almost
      toric off\-cut
      \(\Off_{\bm{\rho},\bm{\lambda}}\left(E_{p,q}(\alpha,\beta)\right)\).}
    \label{fig:offcut}
  \end{center}
\end{figure}
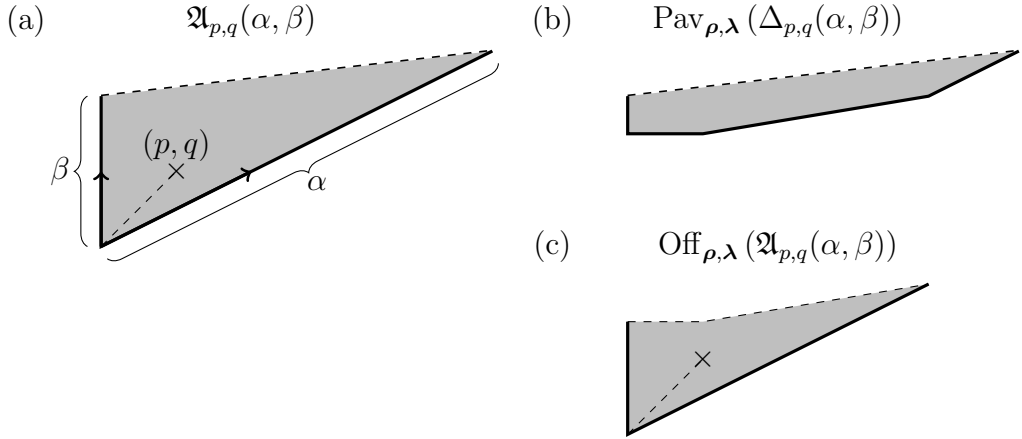

\subsection{Vianna triangles}
\label{subsec:Viannatriangles}

Let \(\Vianna(1,1,1)\) be the triangle with vertices
\(\vtx_1=(0,0)\), \(\vtx_2=(1,0)\) and \(\vtx_3=(0,1)\). This is
the moment image of a Hamiltonian torus action on
\(\CP^2\). We can make nodal trades at the three corners to
get an almost toric base diagram with three nodes
\(\node_1,\node_2,\node_3\) where \(\node_i\) is connected to the
vertex~\(\vtx_i\) by a branch cut~\(\bc_i\). The result is the
base of an almost toric fibration on \(\CP^2\). We can
perform mutations of this almost toric base diagram and we
obtain a sequence of almost toric base diagrams called 
{\em Vianna triangles}. 
For a general introduction to Vianna triangles and their elementary properties see \cite[Appendix I]{ELTF}. 
We continue to write
\(\vtx_1,\vtx_2,\vtx_3\) for the vertices (ordered
anticlockwise) and \(\bc_i\) for the branch cut connecting
\(v_i\) to the node \(\node_i\). Recall that the
determinant\footnote{If the primitive integer vectors along the
  edges emanating from \(\vtx_i\) are \(u\) and~\(v\), then the
  determinant of \(\vtx_i\) is defined to be \(|u\wedge v|\),
  that is the absolute value of the determinant of the matrix
  whose columns are \(u\) and \(v\).} at the vertex~\(\vtx_i\) is
equal to \(p_i^2\) for some positive integer \(p_i\), and
together, \((p_1,p_2,p_3)\) form a {\em Markov triple}. Recall
that \((p_1,p_2,p_3)\) is called a Markov triple if
\begin{equation}\label{eq:markov}
  p_1^2+p_2^2+p_3^2=3p_1p_2p_3.
\end{equation}
We will write
\(\Vianna(p_1,p_2,p_3)\) for the Vianna triangle corresponding to
\((p_1,p_2,p_3)\). If we mutate \(\Vianna(p_1,p_2,p_3)\) at the vertex
\(p_2\), say, then the result is \(\Vianna(p'_1,p'_2,p'_3)\) where
\begin{equation}\label{eq:mutation}
  p'_1=p_1,\quad p'_2=p_3,\quad p'_3=3p_1p_3-p_2.
\end{equation} Label the
edges of \(\Vianna(p_1,p_2,p_3)\) as \(\edge_1,\edge_2,\edge_3\),
where \(\edge_i\) is opposite vertex \(\vtx_i\). The affine
length of \(\edge_i\) is \(\frac{p_i}{p_{i+1}p_{i+2}}\),
see~\cite[Corollary~I.13]{ELTF}.

If we ignore the nodes and branch cuts, then
\(\Vianna(p_1,p_2,p_3)\) is the moment image for the toric orbifold
surface \(\PP(p_1^2,p_2^2,p_3^2)\) which has a
\(\frac{1}{p_i^2}(p_{i+2}^2,p_{i+1}^2)\) cyclic quotient
singularity\footnote{Note that
  there is a confusing ordering issue here. The conventions we
  are using for the moment polygon ensure that the
  Hirzebruch--Jung chain appears in order if you read the
  self-intersections off anticlockwise. However, with this
  convention, the moment-preimage of the vertical edge in
  \(\Delta_{p,q}(\alpha,\beta)\) is contained in the \(x\)-axis
  and the moment-preimage of the slanted edge is contained in
  the \(y\)-axis. Indeed, the affine change of coordinates in
  {\cite[Example 3.21]{ELTF}} which finds the moment polygon for
  an orbifold singularity has negative determinant. This is why
  we write \(\frac{1}{p_i^2}(p_{i+2}^2,p_{i+1}^2)\) rather than
  \(\frac{1}{p_i^2}(p_{i+1}^2,p_{i+2}^2)\).} living over the vertex \(\vtx_i\). Note that (taking
indices modulo \(3\))
\(\frac{1}{p_i^2}(p_{i+2}^2,p_{i+1}^2)\cong
\frac{1}{p_i^2}(p_{i+1}^2p_{i+2}^{-2},1)\), where the inverse is
computed modulo \(p_i^2\). Equation~\eqref{eq:markov} tells us
that
\(p_{i+1}^2p_{i+2}^{-2}=3p_ip_{i+2}^{-1}p_{i+1}-1\mod p_i^2\).
Furthermore, $p_{i+1}^2p_{i+2}^{-2} = p_iq_i-1 \mod p_i^2$ since the singularity is Wahl. 
Therefore, 
\begin{equation}\label{eq:formula_for_companion}q_i= +3p_{i+1}p_{i+2}^{-1}\mod p_i.\end{equation}

\begin{figure}[htb]
  \begin{center}
    \begin{tikzpicture}
      \filldraw[fill=lightgray,draw=black] (0,0) -- (3,3) -- (0,5) -- cycle;
      \filldraw[fill=gray,draw=none] (0,0) -- (2,2) -- (0,3) -- cycle;
      \draw[dashed] (0,0) -- ++ (0.16*3,0.5*3) node (a) {\(\times\)};
      \draw[dashed] (3,3) -- ++ (-1.5,0) node (b) {\(\times\)};
      \draw[dashed] (0,5) -- ++ (0.25*3,-0.5*3) node (c) {\(\times\)};
      \node at (a) [above] {\(\node_1\)};
      \node at (b) [below] {\(\node_2\)};
      \node at (c) [below] {\(\node_3\)};
      \node at (0.16*1.5,0.5*1.5) [above right] {\(\bc_1\)};
      \node at (2,3) [above] {\(\bc_2\)};
      \node at (0.25*1.5,{5-0.5*1.5}) [below] {\(\bc_3\)};
      \node at (1.5,4) [above right] {\(|\edge_1|=\frac{p_1}{p_2p_3}\)};
      \node at (0.2,0) [below] {\(\vtx_1,p_1\)};
      \node at (3,3) [right] {\(\vtx_2,p_2\)};
      \node at (0,5) [above] {\(\vtx_3,p_3\)};
      \node at (1,2.5) {\(\ell\)};
      \draw [decorate,decoration={brace,amplitude=5pt,raise=1ex}] (0,0) -- (0,3) node[midway,xshift=-3ex]{\(\beta\)};
      \draw [decorate,decoration={brace,amplitude=20pt,raise=3ex}] (0,0) -- (0,5) node[midway,sloped,yshift=9ex]{\(|\edge_2|=\frac{p_2}{p_3p_1}\)};
      \draw [decorate,decoration={brace,amplitude=5pt,raise=1ex}] (2,2) -- (0,0) node[midway,xshift=1.5ex,yshift=-2.5ex]{\(\alpha\)};
      \draw [decorate,decoration={brace,amplitude=20pt,raise=4ex}]
      (3,3) -- (0,0) node[midway,sloped,xshift=1.5ex,yshift=-10ex]{\(|\edge_3|=\frac{p_3}{p_1p_2}\)};
      \begin{scope}[shift={(7,0)}]
        \filldraw[fill=lightgray,draw=black] (0,0) -- (4.5,4.5) -- (0,3) -- cycle;
        \draw[dashed] (0,0) -- ++ (0.16*3,0.5*3) node (a) {\(\times\)};
        \draw[dashed] (0,3) -- ++ (1.5,0) node (b) {\(\times\)};
        \draw[dashed] (4.5,4.5) -- ++ (-1.16*2,-0.5*2) node (c) {\(\times\)};
        \node at (0,0) [below] {\(\vtx'_1,p'_1=p_1\)};
        \node at (4.5,4.5) [right] {\(\vtx'_2,p'_2=p_3\)};
        \draw[draw=none] (-0.1,3) -- (1,3.5) node[sloped,midway,above] {\(\vtx'_3,p'_3=3p_3p_1-p_2\)};
      \end{scope}
    \end{tikzpicture}
    \caption{A Vianna triangle and its mutation at \(\vtx_2\);
      quantities for the mutated triangle are written with
      primes. The darker shaded region bounded by the line
      \(\ell\) is the visible
      \(\ATF_{p_1,q_1}(\alpha,\beta)\). Next to each vertex
      \(\vtx_i\) we also give the Markov number \(p_i\) with
      \(\det(\vtx_i)=p_i^2\).}
    \label{fig:vianna_triangle}
  \end{center}
\end{figure}
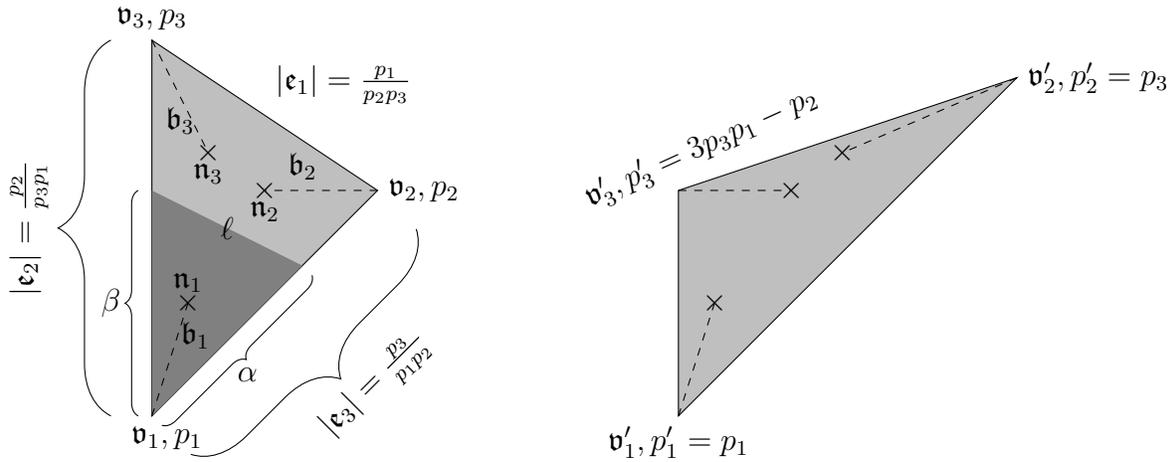

\begin{lemma} Given a Markov triple \((p_1,p_2,p_3)\), take the companion number
  \(q_i=3p_{i+1}p_{i+2}^{-1}\mod p_i\) for \(i=1,2,3\)
  (indices taken modulo \(3\)). Then \(\CP^2\) contains a
  visible symplectic \(E_{p_i,q_i}(\alpha,\beta)\) for any
  \(\alpha<\frac{p_{i+2}}{p_{i}p_{i+1}}\) and
  \(\beta<\frac{p_{i+1}}{p_ip_{i+2}}\).
\end{lemma}
\begin{proof}
  The constraints on \(\alpha\) and \(\beta\) mean that we can
  pick the points \(\pt_{i+2}\) on \(\edge_{i+2}\) an affine
  length \(\alpha\) from \(v_i\) and \(\pt_{i+1}\) on
  \(\edge_{i+1}\) an integral affine length \(\beta\) from
  \(v_i\). Connect them with a straight line \(\ell\); nodally
  slide \(\node_i\) so that it lives on one side of \(\ell\) and
  \(\node_{i+1}\) and \(\node_{i+2}\) so that they live on the
  far side of \(\ell\). The line \(\ell\) divides
  \(\Vianna(p_1,p_2,p_3)\) into two pieces; the piece containing
  \(\node_i\) is integral-affine equivalent to
  \(\ATF_{p_i,q_i}(\alpha,\beta)\), so we obtain a visible
  embedding of \(E_{p_i,q_i}(\alpha,\beta)\) as required; see
  Figure \ref{fig:vianna_triangle}.
\end{proof}

\subsection{The culet}\label{pg:culet}
Once again, ignore the almost toric data and consider
\(\Vianna(p_1,p_2,p_3)\) as the moment image of the toric orbifold
\(\PP(p_1^2,p_2^2,p_3^2)\). Suppose that \(p_1\) is the biggest
number in the Markov triple \((p_1,p_2,p_3)\) and use an
integral affine transformation to make the edge \(\edge_1\) of
\(\Vianna(p_1,p_2,p_3)\) horizontal, with the rest of the triangle
below.

\begin{definition}
  Let us subdivide the fan of inward normals for the Vianna
  triangle by adding a ray pointing vertically upwards. This
  gives partial resolution of the
  \(\frac{1}{p_1^2}(p_2^2,p_3^2)\) singularity introducing a
  single exceptional curve which we call the {\em culet
    curve}.\footnote{We will continue to refer to the proper
    transform of this curve as the culet curve in any resolution
    dominating this one.} If we fix a sufficiently small
  \(\lambda\) then we get a (non-Delzant) pavilion with a new
  horizontal edge at the bottom, which we call the {\em culet};
  see Figure~\ref{fig:wps}.
\end{definition}

There are still (up to) two orbifold singularities on this edge
at the new vertices \(\vtx'_1\) and \(\vtx''_1\); these are
toric, with their moment images modelled on the {\em alternate
  angles} for the corners \(\vtx_2\) and \(\vtx_3\) and so are
{\em dual} to these Wahl singularities 
{\cite[Proposition~6]{MikhalkinShkolnikov}}. We can further (minimally) resolve
these singularities by completing our subdivision to a Delzant
pavilion~\(\bm{\rho}\). Corresponding to these new rays in the
fan, we have a chain of exceptional curves of the form
\([W^\vee_0,w,W^\vee_1]\) where \(W_i^\vee\) are dual Wahl
chains and \(-w\) is the self-intersection of the culet
curve. By a theorem of Urz\'{u}a and Z\'{u}\~{n}iga
{\cite[Proposition~4.1]{UrzuaZuniga}}, this is only possible if:
\begin{itemize}
\item both \(W^{\vee}_i\) are nonempty chains and \(w=10\),
\item precisely one of the \(W^{\vee}_i\) is nonempty and
  \(w=7\), or
\item both of the \(W^\vee_i\) are empty and \(w=4\).
\end{itemize}
Hence the chain \([W^\vee_0,w,W^\vee_1]\) does not contain any
$(-1)$-spheres and is therefore minimal.

\begin{remark}\label{rmk:non_minimal_culet}
    If we pick a non-minimal resolution, it dominates the minimal resolution. We will still refer to the proper transform of the culet curve as the culet curve, so the notion makes sense for arbitrary resolutions.
\end{remark}

\begin{definition}
    Given a chain-shaped resolution of the \(\frac{1}{p_1^2}(p_2^2,p_3^2)\) singularity, we write
    \(C_1,\ldots,C_m\) for the exceptional curves of this resolution and
    write \(i_{\cul}\) for the index of the culet curve.
\end{definition}

\begin{definition}\label{dfn:manetti_weight}
  Let \(w\) be \(-C_{i_\cul}^2\) for the minimal resolution.
  We call this the {\em Manetti weight} of the pin-ellipsoid (note that it only makes sense when \(p\) is a Markov number and \(q\) a companion of \(p\)).
\end{definition}

\begin{remark}
    Recall that \(w\in\{4,7,10\}\) by the aforementioned result of Urz\'{u}a and Z\'{u}\~{n}iga
    {\cite[Proposition~4.1]{UrzuaZuniga}}. The Manetti weights depend on the pairs $(p,q)$ as follows. The pair $(p,q)=(2,1)$ is the only one which yields $w=4$. The Manetti weight is $w=7$ if and only if $(p,q)$ comes from a Markov triple containing a $1$. In every other case it is $w=10$.
    Note that the square of the culet curve in a non-minimal resolution will be bigger than the weight if our non-minimal resolution is obtained from the minimal one by blow up of points on the culet curve.
\end{remark}

\begin{figure}[htb]
  \begin{center}
    \begin{tikzpicture}
      \filldraw[pattern=north west lines,draw=black] (0,0) -- (-3.616628176,1.674364896) -- (11.31440631,1.674364896) -- (0,0);
      \filldraw[fill=lightgray,draw=none] (-3.616628176/4,1.674364896/4) -- (-3.616628176,1.674364896) -- (11.31440631,1.674364896) -- (11.31440631/4,1.674364896/4) -- cycle;
      \draw[very thick] (-3.616628176/4,1.674364896/4) -- (-3.616628176,1.674364896) -- (11.31440631,1.674364896) -- (11.31440631/4,1.674364896/4) -- cycle;
      \draw[dashed] (-4,1.674364896/4) -- (12,1.674364896/4);
      \draw[dashed] (-4,1.674364896) -- (12,1.674364896);
      \draw (-3.616628176/4,1.674364896/4) ++ (0.39,0) arc [radius=0.39,start angle=0,end angle=155];
      \draw (-3.616628176,1.674364896) ++ (-0.39,0) arc [radius=0.39,start angle=180,end angle=335];
      \draw[double] (11.31440631/4,1.674364896/4) ++ (-0.39,0) arc [radius=0.39,start angle=180,end angle=10];
      \draw[double] (11.31440631,1.674364896) ++ (0.39,0) arc [radius=0.39,start angle=0,end angle=-170];
      \node at (-3.616628176,1.674364896) [above] {\(\vtx_3\)};
      \node at (11.31440631,1.674364896) [above] {\(\vtx_2\)};
      \node at (1,1.67/4) [above] {culet, \(C_{i_\cul}\)};
      \node at (0,0) [below] {\(\vtx_1\)};
      \node at (0,2) {\(\edge_1\)};
    \end{tikzpicture}
    \caption{The culet cut is horizontal, parallel to the edge
      \(\edge_1\) of affine length \(p_1/(p_2p_3)\). The new corners
      are integral-affine dual to the corners at \(\vtx_2\) and
      \(\vtx_3\).}
    \label{fig:wps}
  \end{center}
\end{figure}
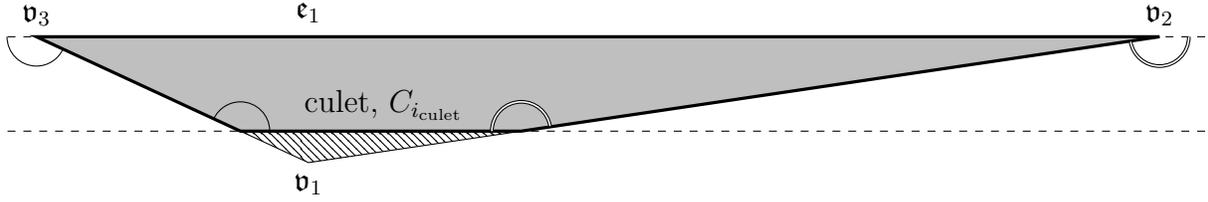

Since the edge \(\edge_1\) is parallel to the culet, there is a
visible symplectic sphere in the sense of Symington
{\cite[Definition 7.2]{Symington3}} having square zero and
living over a vertical arc connecting \(\edge_1\) and the
culet. One of the main ideas in what follows will be
to find the analogue of this square zero sphere for an arbitrary
symplectic embedding \(E_{p,q}(\alpha,\beta)\hookrightarrow\CP^2\).

\section{Geometry and topology of pavilion blow-ups}
\label{sct:rat_pav}

In Section \ref{sct:pav}, we explain in detail the almost complex and symplectic structures with
which we will equip pavilion blow-ups. In Section \ref{ss:top}, we explain how to compute in the cohomology ring of the pavilion blow-up.

\subsection{Pavilion blow up and its geometric structures}
\label{sct:pav}

\begin{notation} Given a symplectic manifold \((W,\omega)\)
with a convex (respectively concave) boundary, let
\(\overline{W}\) denote its symplectic completion, obtained by
gluing on copies of the positive (respectively negative) half of
the symplectisation of the boundary. If \(J\) is an
\(\omega\)-compatible almost complex structure on \(W\) which is
cylindrical near the ends then let \(\overline{J}\) denote the
unique almost complex structure on \(\overline{W}\) extending
\(J\) cylindrically to the ends. If \(W\) has both convex and
concave ends then define \(\overline{W}_+\) (respectively
\(\overline{W}_-\)) to be the result of completing only the
convex (respectively concave) ends.
\end{notation}

\begin{definition}\label{dfn:N}
  A neighbourhood \(N\) of the boundary of
  \(E_{p,q}(\alpha,\beta)\) is symplectomorphic to a
  neighbourhood of the boundary of the girdled orbifold
  \(\Orb_{p,q}(\alpha,\beta)\), namely the moment-preimage of
  the shaded region in \(\Tri_{p,q}(\alpha,\beta)\) shown in
  Figure \ref{fig:toric_shaded}. Since this toric orbifold is a
  finite quotient of a compact set in \(\CC^2\), it has a
  complex structure. Let \(J_N\) be the pullback of this complex
  structure to \(N\).
\end{definition}

\begin{figure}[htb]
  \begin{center}
    \begin{tikzpicture}
      \filldraw[lightgray,opacity=0.75] (0,1.5) -- (0,0) -- (8,2) -- cycle;
      \filldraw[gray] (0,1) -- (0,1.5) -- (8,2) -- (5.5,1.375) -- cycle;
      \draw[very thick] (0,1.5) -- (0,0) -- (8,2);
      \draw [decorate,decoration={brace,amplitude=5pt,raise=1ex}] (0,0) -- (0,1.5) node[midway,xshift=-3ex]{\(\beta\)};
      \draw [decorate,decoration={brace,amplitude=5pt,raise=1ex}] (8,2) -- (0,0) node[midway,xshift=1.5ex,yshift=-2.5ex]{\(\alpha\)};
      \node at (8,2) [right] {\((\alpha p^2,\alpha(pq-1))\)};
      \draw[->] (0,0.75) -- (0.5,0.75) node [right] {\(\rho_0\)};
      \draw[->] (4,1) -- ++ (-0.25,1) node [above] {\(\rho_{m+1}\)};
      \draw[->] (2,1.625) -- ++ (0.0625,-1);
      \node at (2.3,0.9) {\(\gamma\)};
      \draw[very thick,dash pattern=on 9pt off 4pt] (0,1.5) -- (8,2);
    \end{tikzpicture}
    \caption{The moment image \(\Tri_{p,q}(\alpha,\beta)\) of a
      toric orbifold. The preimage under the moment map of the
      shaded region is symplectomorphic to a neighbourhood of
      \(\partial E_{p,q}(\alpha,\beta)\) in
      \(E_{p,q}(\alpha,\beta)\). The toric boundary comprises
      the two solid edges, {\em not} the top side; the orbifold point is at the vertex
      of the polygon. We have also labelled the inward normals
      \(\rho_0\) and \(\rho_{m+1}\) of the toric boundary and
      \(\gamma\) of the top side (girdle).}
    \label{fig:toric_shaded}
  \end{center}
\end{figure}
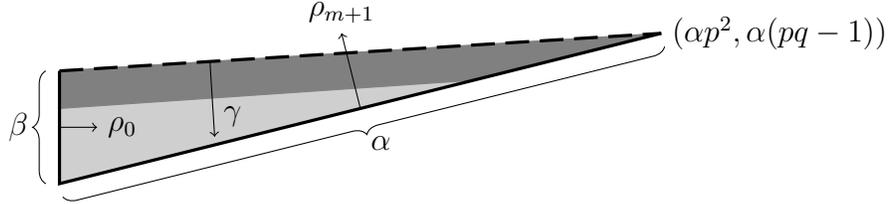

\begin{definition}\label{dfn:JN} Let \((X,\omega)\) be a symplectic manifold
  and suppose that we have a symplectic embedding
  \(\iota\colon E_{p,q}(\alpha,\beta) \hookrightarrow X\). Let
  \(\JJ_N\) be the space of compatible almost complex structures
  on \(X\) which coincide with \(\iota_*J_N\) on \(\iota(N)\).
\end{definition}

These almost complex structures are {\em adjusted} to the neck
region \(N\) in the sense of {\cite[Section 2.2]{BEHWZ}} and so
amenable to {\em neck-stretching}. Neck-stretching (or
splitting) is a specific one-parameter deformation of the almost
complex structure introduced in {\cite[Section 1.3]{EGH}};
neck-stretching for pin-ellipsoids is discussed in detail in
{\cite[Sections 3.2--3.3]{Viannainfty}} and {\cite[Section
  3.1]{ES}}.

\begin{definition}\label{dfn:spaces}
  Given a \(J\in\JJ_N\), we get a neck-stretching sequence of
  compatible almost complex structures \(J_t\) on \(X\) with
  \(J_0=J\); let us write \(U=\iota(E_{p,q}(\alpha,\beta))\) and
  \(V=X\setminus U\) and \(\overline{U}\), \(\overline{V}\) for
  the symplectic completions. Since
  \((\overline{N}_-,\overline{J_N})\) is isomorphic to a
  punctured neighbourhood of the origin in \(\CC^2/G\), we can
  compactify \((\overline{V},\overline{J})\) by adding in a
  point \(x\) to obtain an almost complex orbifold
  \((\ha{X},\ha{J})\). Let
  \(\ha{N}\coloneqq \overline{N}_-\cup\{x\}\); this is an
  embedded copy of the girdled orbifold
  \(\Orb_{p,q}(\alpha,\beta)\).
\end{definition}

We will now resolve \(\ha{X}\) using a pavilion as in Definition
\ref{dfn:pav_local}. We will need to choose a Delzant pavilion
such that the image of the neck \(N\) under the moment map is contained in
\(\Pav_{\bm{\rho},\bm{\lambda}}\left(\Tri_{p,q}(\alpha,\beta)\right)\); in
practice, we will fix the choice of pavilion first, and then
choose \(N\) sufficiently close to the girdle that it satisfies
this condition.

\begin{definition}[Pavilion blow-up]\label{dfn:pavilion_blow_up}
  Given a Delzant pavilion \(\bm{\rho},\bm{\lambda}\) for
  \(E_{p,q}(\alpha,\beta)\) and a symplectic embedding
  \(\iota\colon E_{p,q}(\alpha,\beta)\hookrightarrow X\), we obtain a resolution \(\til{X}\)
  of \(\ha{X}\) by replacing \(\ha{N}\subset \ha{X}\) with the
  girdled resolution \(\til{U}_{\bm{\rho}}\) of
  \(\Orb_{p,q}(\alpha,\beta)\).
  \begin{itemize}
  \item[(i)] Let \(\til{J}\) be the almost
    complex structure on \(\til{X}\) which coincides with \(J\) on
    \(V\) and agrees with the complex structure on
    \(\til{U}_{\bm{\rho}}\).
  \item[(ii)] Note that if \(J\in \JJ_N\), the associated almost
    K\"{a}hler metric \(\omega(-,J-)\) is flat on \(N\). Let
    \(\psi_0\) be the convex function on the moment-image of
    \(N\) which encodes the K\"{a}hler potential of
    \(\omega|_N\). Given a convex function \(\psi\) on
    \(\Pav_{\bm{\rho},\bm{\lambda}}\left(\Tri_{p,q}(\alpha,\beta)\right)\)
    which coincides with \(\psi_0\) on the shaded region
    (moment-image of \(N\)) and has the correct behaviour along
    the edges, we obtain a symplectic form \(\til{\omega}\) on
    \(\til{X}\) which coincides with \(\omega\) on \(V\) and
    \(\til{\omega}_{\bm{\rho},\bm{\lambda},\psi}\) on
    \(\til{U}_{\bm{\rho}}\); this is well-defined because
    \(\psi\) is chosen to coincide with \(\psi_0\) on the
    moment-image of \(N\).
  \end{itemize}
  We call the manifold
  \(\til{X}\coloneqq V\cup \til{U}_{\bm{\rho}}\), equipped with
  the almost complex structure \(\til{J}\) and the symplectic
  form \(\til{\omega}\) the {\em pavilion blow-up} of \(X\)
  along the embedded pin-ellipsoid \(U\). If we want to keep the dependence of \(\til{X}\) on the various choices explicit, we will write \(\Pav_{\bm{\rho},\bm{\lambda}}^\iota(X)\) for \(\til{X}\). We will sometimes abusively drop the subscript on \(\til{U}\).
\end{definition}

\begin{remark}
A pavilion blow-up with minimal Delzant pavilion is just the usual \textit{symplectic rational blow-up}\/ established by Symington \cite{Symington1,Symington2}, which was elaborated upon by Khodorovskiy \cite{Khod1}.
It is the inverse operation of the \textit{rational blow-down}, introduced in the smooth category by Fintushel and Stern \cite{FS}; see also {\cite[Section 9.2]{ELTF}} for an exposition in the symplectic category.
In the last decade the symplectic rational blow-up was used by many authors to study Lagrangian pin-wheels and pin-ellipsoids. 
See for example the work of 
Borman--Li--Wu~\cite{BormanLiWu}, 
Shevchishin--Smirnov~\cite{ShevSmir} and Adaloglou~\cite{Adal1, Ada25} on Lagrangian embeddings of~\(\RR\PP^2\); 
the work of Khodorovskiy~\cite{Khod2} and Evans--Smith~\cite{ES2} on bounding Wahl singularities; 
the work of Brendel--Schlenk on Lagrangian barriers~\cite{BrendelSchlenk} generalising Biran's work~\cite{BiranBarriers}; 
and the work of Adaloglou~\cite{Adal1} and Adaloglou--Hauber~\cite{AdalHauber} on symplectic embeddings of~\(B_{n,1}\).
\end{remark}

Figure \ref{fig:spaces} illustrates the
spaces \(X,\overline{V},\ha{X},\til{X}\), the subspaces
\(U,V,\til{U},N,\overline{N}_-,\ha{N}\), and the relations
between them all.

\begin{figure}[htb]
  \begin{center}
    \begin{tikzpicture}
      \filldraw[draw=none,fill=lightgray] (0,0)  -- (-0.5,0.5) -- (1.5,0.5) -- (1,0) -- cycle;
      \draw (0,0) -- (-0.5,0.5) to[out=135,in=180] (0.5,2.5) to[out=0,in=45] (1.5,0.5) -- (1,0) to[out=-135,in=0]
      (0.5,-0.5) to[out=180,in=-45] (0,0);
      \node at (0.5,0.25) {\(N\)};
      \draw[decorate,decoration={brace,amplitude=5pt}] (-0.9,0.5) -- (-0.9,2.5) node[midway,xshift=-1em] {\(V\)};
      \draw[decorate,decoration={brace,amplitude=5pt}] (-0.9,-0.5) -- (-0.9,0.5) node[midway,xshift=-1em] {\(U\)};
      \node at (1.7,2.5) {\((X,J_0)\)};
      \begin{scope}[shift={(3.5,0)}]
        \filldraw[draw=none,fill=lightgray] (0.5,-0.5)  -- (-0.5,0.5) -- (1.5,0.5) -- (1,0) -- cycle;
        \draw (0,0) -- (-0.5,0.5) to[out=135,in=180] (0.5,2.5) to[out=0,in=45] (1.5,0.5) -- (1,0) -- (0.5,-0.5) -- cycle;
        \node at (0.5,0.15) {\(\overline{N}_-\)};
        \node at (1.7,2.5) {\((\overline{V},\overline{J})\)};
        \node[white] at (0.5,-0.5) {\(\bullet\)};
      \end{scope}
      \begin{scope}[shift={(7,0)}]
        \filldraw[draw=none,fill=lightgray] (0.5,-0.5)  -- (-0.5,0.5) -- (1.5,0.5) -- (1,0) -- cycle;
        \draw (0,0) -- (-0.5,0.5) to[out=135,in=180] (0.5,2.5) to[out=0,in=45] (1.5,0.5) -- (1,0) -- (0.5,-0.5) -- cycle;
        \node at (1.7,2.5) {\((\ha{X},\ha{J})\)};
        \node (p) at (0.5,-0.5) {\(\bullet\)};
        \node at (0.5,0.15) {\(\ha{N}\)};
        \node at (p) [below] {\(x\)};
      \end{scope}
      \begin{scope}[shift={(10.5,0)}]
        \filldraw[draw=none,fill=lightgray] (1.5,0.5) to[out=-135,in=90] (1.5,-0.5) --
        (1,-0.75) -- (0.5,-0.5) -- (0,-0.75) -- (-0.5,-0.5) to[out=90,in=-45] (-0.5,0.5);
        \draw (-0.5,0.5) to[out=135,in=180] (0.5,2.5)
        to[out=0,in=45] (1.5,0.5) to[out=-135,in=90] (1.5,-0.5) --
        (1,-0.75) -- (0.5,-0.5) -- (0,-0.75) -- (-0.5,-0.5) to[out=90,in=-45] (-0.5,0.5);
        \node at (1.7,2.5) {\((\til{X},\til{J})\)};
        \draw (-0.5-0.2,-0.5+0.1) -- (0+0.2,-0.75-0.1) node
        [pos=0.3,below] {\(C_1\)};
        \draw (0-0.2,-0.75-0.1) -- (0.5+0.2,-0.5+0.1);
        \draw (0.5-0.2,-0.5+0.1) -- (1+0.2,-0.75-0.1);
        \draw (1-0.2,-0.75-0.1) -- (1.5+0.2,-0.5+0.1) node
        [pos=0.8,below] {\(C_m\)};
        \node at (0.5,-0.5) [below] {\(\cdots\)};
        \draw[decorate,decoration={brace,amplitude=5pt}] (2.4,2.5) -- (2.4,0.5) node[midway,xshift=+1em] {\(V\)};
        \draw[decorate,decoration={brace,amplitude=5pt}] (2.4,0.5) -- (2.4,-0.8) node[midway,xshift=1em] {\(\til{U}\)};
      \end{scope}
      \draw[->] (9.8,0) -- (8.5,0) node [midway,above] {\(\pi\)};
      \draw[decoration=snake,decorate,->] (1.5,0) -- (3,0) node
      [midway,above] {stretch};
      \node at (2.25,0) [below] {\(J_t\)};
      \draw[right hook->] (5,0) -- (6.5,0);
    \end{tikzpicture}
    \caption{The spaces defined in
      Section \ref{sct:pav}}
    \label{fig:spaces}
  \end{center}
\end{figure}
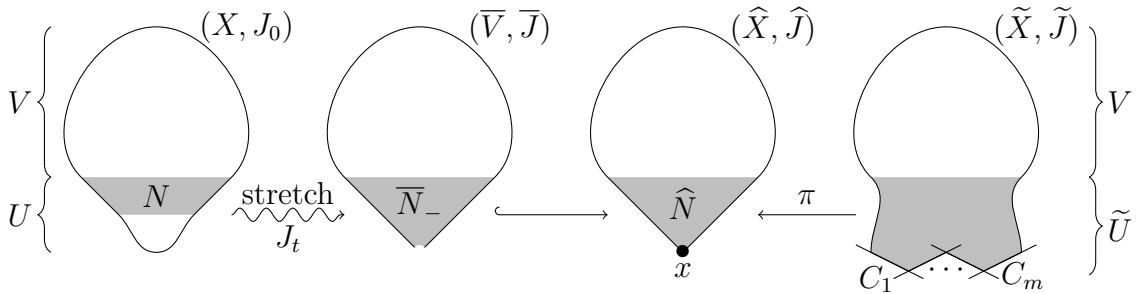

\begin{lemma}\label{lma:bijections}
  There are bijections between
  the following collections of objects:
  \begin{itemize}
  \item irreducible finite energy punctured
    \(\overline{J}\)-holomorphic curves \(C\) in \(\overline{V}\),
  \item irreducible orbifold \(\ha{J}\)-holomorphic curves
    \(\ha{C}\) in \(\ha{X}\),
  \item irreducible \(\til{J}\)-holomorphic curves \(\til{C}\)
    in \(\til{X}\) other than \(C_1,\ldots,C_m\).
  \end{itemize}
\end{lemma}
\begin{proof}
  Given a finite energy punctured \(\overline{J}\)-holomorphic
  curve \(C\subset\overline{V}\), we write \(\ha{C}\) for
  its closure in \(\ha{X}\). The inverse operation is
  puncturing (removing \(x\)). Given an orbifold curve
  \(\ha{C}\) in \(\ha{X}\), write \(\til{C}\)
  for the proper transform in \(\til{X}\). The inverse
  operation is projection along~\(\pi\), which sends any curve
  other than \(C_1,\ldots,C_m\) to an orbifold curve.
\end{proof}

\begin{lemma}\label{lma:j_generic}
\begin{itemize}
\item[(a)] Any irreducible
  \(\til{J}\)-holomorphic curve in \(\til{X}\) which is
  not one of the \(C_i\) must enter the interior of~\(V\).
\item[(b)] By choosing \(J\) generically on \(V\), we can ensure
  that all curves which enter \(V\) are regular\footnote{A curve
    is called {\em regular} if its linearised Cauchy--Riemann
    operator is surjective; see {\cite[Section
      3.1]{McDuffSalamon}}. For an embedded holomorphic sphere
    \(C\) in a 4-manifold, {\cite[Lemma 3.3.3]{McDuffSalamon}}
    implies that regularity is equivalent to the condition
    \(C^2\geq -1\).}, so that the only irregular
  \(\til{J}\)-holomorphic curves are amongst\footnote{If the
    resolution is minimal then {\em all}\/ the \(C_i\) are
    irregular (\(C_i^2\leq -2\)). Otherwise there are also some
    \(-1\)-spheres amongst them.} the \(C_1,\ldots,C_m\).
\end{itemize}
\end{lemma}
\begin{proof}
  (a) If \(\til{C}\subset \til{X}\) is a \(\til{J}\)-curve other
  than \(C_1,\ldots,C_m\) which does not enter the interior of~\(V\) then the associated punctured curve \(C\) is contained
  in \(\overline{N}_-\). Since the positive boundary of
  \(\overline{N}_-\) is convex, this contradicts the maximum
  principle.

  (b) This follows from {\cite[Proposition 3.2.1 and Remark
    3.2.3]{McDuffSalamon}}.
\end{proof}

\begin{remark}
  Note that everything that was introduced and discussed within
  this section makes sense in the case where there are several
  disjoint pin-ellipsoids.
\end{remark}

\subsection{Topology of the pavilion blow-up} \label{ss:top}

Given a symplectic embedding
\(\iota\colon E_{p,q}(\alpha,\beta)\hookrightarrow X\) and a choice of pavilion,
we will need to understand how to compute intersection numbers
of curves in the pavilion blow-up \(\til{X}\) by computing
locally in \(\til{U}\) and \(V = X \setminus \til{U}\)
separately, and adding. In what follows, if \(A\) is a class in
integral homology, we will write \(A_\QQ\) for the corresponding
class in rational homology. Denote the interface between
$\til{U}$ and $V$ by $\Sigma$ and recall that $\Sigma$ is
topologically a lens space $L(p^2,pq-1)$.

\begin{lemma} We have \(H_1(U,\Sigma;\ZZ)=H_3(U,\Sigma;\ZZ)=0\)
  and \(H_2(U,\Sigma;\ZZ)\cong\ZZ_p\).
\end{lemma}
\begin{proof}
  This follows immediately from the long exact sequence of the pair \((U,\Sigma)\):

  \begin{center}
    \begin{tikzpicture}
      \tikzmath{\x = 4; \y = 1.5;}
      \node (11) at (0,0) {\(\cdots\)};
      \node (12) at (\x,0) {\(H_3(U;\ZZ)\)};
      \node (13) at (2*\x,0) {\(H_3(U,\Sigma;\ZZ)\)};
      \node (21) at (0,-\y) {\(H_2(\Sigma;\ZZ)\)};
      \node (22) at (\x,-\y) {\(H_2(U;\ZZ)\)};
      \node (23) at (2*\x,-\y) {\(H_2(U,\Sigma;\ZZ)\)};
      \node (31) at (0,-2*\y) {\(H_1(\Sigma;\ZZ)\)};
      \node (32) at (\x,-2*\y) {\(H_1(U;\ZZ)\)};
      \node (33) at (2*\x,-2*\y) {\(H_1(U,\Sigma;\ZZ)\)};
      \node (34) at (2.5*\x,-2*\y) {\(0.\)};
      \draw[->] (11) -- (12);
      \draw[->] (21) -- (22);
      \draw[->] (31) -- (32) node [midway,below] {\(\mod p\)};
      \draw[->] (12) -- (13);
      \draw[->] (22) -- (23);
      \draw[->] (32) -- (33);
      \draw[->] (13) -- (21);
      \draw[->] (23) -- (31);
      \draw[->] (33) -- (34);
      \node (12l) at (12) [below=0.5cm, left=1cm] {\(0\)};
      \draw[draw=none] (12l) -- (12) node [midway,sloped] {\(=\)};
      \node (21l) at (21) [below=0.5cm, left=1cm] {\(0\)};
      \draw[draw=none] (21l) -- (21) node [midway,sloped] {\(=\)};
      \node (22l) at (22) [below=0.5cm, left=1cm] {\(0\)};
      \draw[draw=none] (22l) -- (22) node [midway,sloped] {\(=\)};
      \node (31l) at (31) [below=0.5cm, left=1cm] {\(\ZZ_{p^2}\)};
      \draw[draw=none] (31l) -- (31) node [midway,sloped] {\(=\)};
      \node (32l) at (32) [below=0.5cm, right=1cm] {\(\ZZ_p\)};
      \draw[draw=none] (32l) -- (32) node [midway,sloped] {\(=\)};
    \end{tikzpicture}
  \end{center}\vspace{-\baselineskip}\qedhere
\end{proof}

\begin{lemma}\label{lma:new_splitting} Given a class
  \(A\in H_2(\til{X};\ZZ)\), there exist rational numbers
  \(a_1,\ldots,a_m\) and a class \(A_X\in H_2(X;\ZZ)\) such that
  \begin{equation}\label{eq:new_splitting}A_\QQ =
    \frac{1}{p}A_X+\sum_{i=1}^ma_i[C_i].\end{equation}
  Here we are
  identifying \(H_2(X;\QQ)\) as a subset of \(H_2(\til{X};\QQ)\)
  using the fact that \(H_2(X;\QQ)\cong H_2(V;\QQ)\). Moreover, the class \(pA_X\) is an integral class which can be represented by a cycle contained in \(V\).
\end{lemma}
\begin{proof}
  The Mayer--Vietoris sequence says:
  \[0=H_2(\Sigma;\ZZ)\to H_2(\til{U};\ZZ)\oplus H_2(V;\ZZ)\to
    H_2(\til{X};\ZZ)\to H_1(\Sigma;\ZZ)=\ZZ_{p^2}\to\cdots,\] so
  if \(A\in H_2(\til{X};\ZZ)\) then \(p^2A\) can be written as
  \(A'_{\til{U}}+A'_V\). Therefore
  \(A_\QQ=\frac{1}{p^2}A'_{\til{U}}+\frac{1}{p^2}A'_V\). We will
  take
  \(A_{\til{U}} :=\frac{1}{p^2}A'_{\til{U}}=\sum_{i=1}^ma_i[C_i]\)
  for some rational numbers \(a_i\in\frac{1}{p^2}\ZZ\). To
  understand~\(A'_V\), consider the exact sequence of the pair
  \((X,V)\):
  \[  0 = H_3(U,\Sigma;\ZZ) \to H_2(V;\ZZ)\to H_2(X;\ZZ)\to
    H_2(X,V;\ZZ)=H_2(U,\Sigma;\ZZ)=\ZZ_p,\] where we used
  excision to identify \(H_2(X,V)\) with \(H_2(U,\Sigma)\). This
  means that \(A'_V=pA_X\) for some \(A_X\in H_2(X;\ZZ)\).
\end{proof}

\begin{remark}\label{rmk:calculate_coeffs}
  We can calculate the numbers \(a_j\) as follows if we know how
  \(A\) intersects the curves \(C_j\). Let
  \(\alpha_j=A\cdot [C_j]\in\ZZ\) and let
  \(\bm{\alpha}=(\alpha_1,\ldots,\alpha_m)^T\) and
  \(\bm{a}=(a_1,\ldots,a_m)\). 
  Then
  \[\alpha_j=\sum_{i=1}^ma_iC_i\cdot C_j,\mbox{ or
    }\bm{\alpha}=M\bm{a},\] where \(M\) is the \(m\)-by-\(m\)
  matrix with entries \(M_{ij}=C_i\cdot C_j\). Therefore
  \(\bm{a}=M^{-1}\bm{\alpha}\). It will be useful to have formulas
  for the entries of \(M^{-1}\).
\end{remark}

\begin{lemma}\label{lma:M_inverse} Suppose that the
  Wahl chain for the minimal resolution of our singularity is
  \[[b_1,\ldots,b_m], \qquad \mbox{ with }b_i=-C_i^2,\] and fix
  an index \(i\). Consider the continued fractions
  \(e_i/e_{i-1} = [b_{i-1},\ldots,b_{1}]\) and
  \(f_i/f_{i+1}=[b_{i+1},\ldots,b_m]\). The entries of
  \(M_{ij}^{-1}\) are
  \[M^{-1}_{ij}=\begin{cases}-(e_if_j)/p^2\mbox{ if }i\leq j\\
      -(e_jf_i)/p^2\mbox{ if }i>j \end{cases}.\]
\end{lemma}
\begin{proof}
  The continued fractions are computed as
  \begin{align*}
    \frac{e_{i+1}}{e_{i}}&=b_i-\frac{1}{[b_{i-1},\ldots,b_1]} & \frac{f_{i-1}}{f_{i}}&=b_i-\frac{1}{[b_{i+1},\ldots,b_m]}\\
                         &=\frac{b_ie_i-e_{i-1}}{e_i},& &=\frac{b_if_i-f_{i+1}}{f_i},
  \end{align*}
  so
  \begin{align*}
    e_{i+1}+e_{i-1}&=b_ie_i,&f_{i+1}+f_{i-1}&=b_if_i,
  \end{align*}
  which are Fibonacci-like recursions with initial conditions
  \[e_0=0,\quad e_1=1,\qquad f_0=p^2,\quad f_1=pq-1.\] 
  For the matrix $M^{-1}$ as postulated in the lemma,
  the off-diagonal entries \((M^{-1}M)_{ij}\) above the diagonal
  are\footnote{The entries in the first and last column have
    only two terms, but the formula still holds true if we
    decide to set \(e_0=f_{m+1}=0\).}
  \[-\frac{1}{p^2}\left(e_if_{j-1}-e_ib_jf_j+e_if_{j+1}\right)=0.\]
  Those below the diagonal are
  \[-\frac{1}{p^2}\left(e_{j-1}f_i-e_jb_jf_i+e_{j+1}f_i\right)=0.\]
  Those on the diagonal are
  \[-\frac{1}{p^2}\left(e_{i-1}f_i-b_ie_if_i+e_if_{i+1}\right) =
    \frac{1}{p^2}\left(e_{i}f_{i-1}-e_{i-1}f_{i}\right).\] We
  can show that \(e_{i}f_{i-1}-e_{i-1}f_{i}=p^2\) by induction:
  it holds when \(i=1\) and
  \[e_{i+1}f_{i}-e_{i}f_{i+1}=\begin{vmatrix} e_{i+1} & f_{i+1} \\
      e_{i} & f_{i}\end{vmatrix} = \begin{vmatrix}
      b_ie_i-e_{i-1} & b_if_i - f_{i-1} \\ e_i &
      f_i \end{vmatrix} = \begin{vmatrix} e_i & f_i \\ e_{i-1}
      & f_{i-1}\end{vmatrix}\] by subtracting \(b_i\) times row
  \(2\) from row \(1\) and swapping the rows.

  This shows that \(M^{-1}M\) is the identity matrix, as required.
\end{proof}

\begin{remark}\label{rmk:ee}
  We will only be concerned with the case when
  \(X=\CP^2\). In this case, we write \(\eE\) for the rational
  class \(\frac{1}{p}H\in H_2(X;\QQ)\), where $H$ is the class of the line, so \(p^2\eE\) can be represented by a cycle in \(V\). We can write any class $A \in H_2(\til{X};\QQ)$ in the form
  \begin{equation}\label{eq:splitting}A=a_0\eE+\sum_{i=1}^ma_i[C_i].\end{equation}
  The square of \(A\) is then given by the formula
  \begin{equation}\label{eq:A_squared}
    A^2=\frac{a_0^2}{p^2}+\bm{a}^TM\bm{a}.
  \end{equation}
\end{remark}

\begin{remark}\label{rmk:classesmultiple}
  In the case of multiple embeddings
  $\iota\colon \bigsqcup_{j=1}^\ell E_{p_j,q_j}(\alpha_j,\beta_j)
  \hookrightarrow X$ we define $\Delta:=\prod_{j=1}^\ell p_j$.
  Then, by performing a pavilion blow-up on all the $\ell$ embedded
  pin-ellipsoids, we obtain spaces $\til{X},\til{V}$
  and~$V$, and Equation~\eqref{eq:new_splitting} becomes
  \[A_{\mathbb{Q}}=\frac{1}{\Delta}A_X + \sum_{j=1}^\ell\sum_{i=1}^{m_j} a_{j,i}[C_{j,i}].\]
  Similarly, in the case that $X=\CP^2$, Equations
  \eqref{eq:splitting} and \eqref{eq:A_squared} become:
    \begin{equation}\label{eq:splitting_mult}
        A=a_0\eE+\sum_{j=1}^\ell\sum_{i=1}^{m_j} a_{j,i}[C_{j,i}] \quad\text{and}\quad A^2=\frac{a_0^2}{\Delta^2}+\sum_{j=1}^\ell \bm{a}_j^T M_j \bm{a}_j.
    \end{equation}
\end{remark}

\section{Ruled 4-manifolds}
\label{sct:ruled_manifolds}

The results of this section are well-known to experts; we thank Weiyi Zhang for informing us that the methods of his paper \cite{ZhangModuli} for irrational ruled surfaces can be adapted to the rational case.

\subsection{Regulations and rulings}

\begin{definition}
    Given a stable \(J\)-holomorphic curve
    \(S\), write \(S=\sum_{i\in I} m_iS_i\) where \(S_i\) are the
    non-constant irreducible components of \(S\), indexed by a set
    \(I\), and \(m_i\) their covering multiplicities. The {\em dual
    graph} of \(S\) is the labelled graph \(\Gamma_S\):
    \begin{itemize}
        \item whose vertex set is \(I\),
        
        \item whose vertex \(i\) is labelled by the self-intersection
        number \(S_i^2\),
        
        \item and where there are \(S_i\cdot S_j\) edges between the
        vertices \(i \neq j\).
    \end{itemize}
\end{definition}

\begin{definition}
Let \((Y,J)\) be a tamed almost complex
4-manifold. Given a homology class \(A\in H_2(Y;\ZZ)\), let
\(\overline{\MM}_{0,n}(A,J)\) denote the moduli space of
\(J\)-holomorphic stable maps in the class \(A\) having \(n\)
marked points. We say that {\em \(Y\) admits a \(J\)-holomorphic
regulation\footnote{This would ordinarily be called a {\em
ruling}, but ruling can also mean one of the curves in the
regulation and this is how we will use it. Regulation is to
{\em regula} (Latin for rule) as a triangulation is to 
{\em triangle}.} in the class \(A\)} if all the following
conditions hold:
\begin{itemize}
    \item[(a)] \(A^2=0\) and \(c_1(Y)\cdot A=2\)
    
    \item[(b)] the evaluation map
      \(\ev\colon\overline{\MM}_{0,1}(A,J)\to Y\) has degree \(1\).
\end{itemize}
The curves in the moduli space \(\overline{\MM}_{0,0}(A,J)\) are
called {\em rulings}. Smooth curves corresponding to points of
the (possibly empty) subspace \(\MM_{0,0}(A,J)\) we call {\em
  smooth rulings}; by the adjunction formula, these are embedded
spheres. Curves corresponding to points of
\(\partial\overline{\MM}_{0,0}(A,J) =
\overline{\MM}_{0,0}(A,J)\setminus\MM_{0,0}(A,J)\) we call {\em
  broken rulings}. We call the regulation of~\(Y\) {\em
  non-degenerate} if there is at least one smooth ruling.
\end{definition}

\begin{example}\label{exm:regulations}
\begin{itemize}
\item[(1)] The simplest example of a ruled surface is
  \(\CP^1\times\CP^1\): this admits two non-degenerate
  regulations, one in the class \([\CP^1\times\{\pt\}]\) and
  one in the class \([\{\pt\}\times\CP^1]\). This ruled
  surface arises as a quadric surface in \(\CP^3\) and the
  rulings in both regulations are lines.
\item[(2)] If we degenerate the quadric to a nodal quadric then
  there is still a regulation (but now only one) consisting of
  lines \(\{F_t \mid t\in\CP^1\}\)
  passing through the node: the two regulations of the smooth surface
  degenerate onto this regulation. If we resolve the node
  (introducing a \(-2\)-curve \(\sigma_\infty\)) then the result
  is the Hirzebruch surface \(\FF_2\). The proper transforms
  \(\tilde{F}_t\) of the rulings of the nodal quadric yield the
  standard non-degenerate regulation of \(\FF_2\); there is
  another, degenerate, regulation whose rulings are
  \(\tilde{F}_t+\sigma_\infty\). 
\end{itemize}
\end{example}

\begin{theorem}\label{thm:existence_of_regulations} If
\(Y\) contains a smoothly embedded
\(J\)-holomorphic curve \(C\) with
\([C]^2=0\) then \(C\) is part of a non-degenerate
\(J\)-holomorphic regulation in the class~\([C]\).
\end{theorem}
\begin{proof}
  We get that \(c_1(Y)\cdot [C]=2\) by adjunction. Since
  \(c_1(Y)\cdot [C]>0\), automatic transversality guarantees that
  the moduli space \(\MM_{0,1}([C],J)\) is a smooth manifold of
  dimension \(4\) and the evaluation map
  \(\ev\colon\MM_{0,1}([C],J)\to Y\) is a pseudocycle. The
  degree of this pseudocycle must be \(1\) because there is at
  most one smooth \(J\)-holomorphic sphere in the class \([C]\)
  passing through any point (by positivity of intersections
  because \([C]^2=0\)) and for all points in a neighbourhood of~\(C\) 
  there is precisely one (namely \(C\) or its small
  deformations, which exist by automatic transversality). In
  particular, \(\ev\colon\overline{\MM}_{0,1}([C],J)\to Y\) has
  degree~\(1\).
\end{proof}

\subsection{Broken rulings} \label{ss:broken}

\begin{proposition}\label{prp:ruled_structure} Let
  \((Y,J)\) be a tamed almost complex manifold admitting a
  non-degenerate regulation in the class 
  \(A\in H_2(Y;\ZZ)\). 
  Let $C$ be a smooth ruling, and let 
  \(S=\sum m_iS_i\) be a broken ruling
  of the regulation.
  \begin{itemize}
      \item[(a)] The curves \(S_i\) are embedded rational curves
        disjoint from \(C\).
        
      \item[(b)] The dual graph \(\Gamma_S\) is a tree.
      
      \item[(c)] The irreducible components \(S_i\) have \([S_i]^2<0\).
      
      \item[(d)] At least one irreducible component is a sphere of
        square \(-1\).
      \end{itemize}
\end{proposition}
\begin{proof}
  (a) and (b). The existence of a smooth ruling of the
  regulation implies that the class \(A\) is \(J\)-nef in the
  sense of Li and Zhang {\cite[Lemma 2.8]{LiZhang}}. This
  implies that the dual graph is a tree and all the components
  are embedded spheres {\cite[Theorem 1.5]{LiZhang}}. Note that
  the components \(S_i\) are disjoint from any smooth ruling
  \(C\), since \(0=A^2=C\cdot\sum m_iS_i=\sum m_i(C\cdot S_i)\)
  and \(C\cdot S_i\geq 0\) for all \(i\), so \(C\cdot S_i=0\)
  for all components \(S_i\).

  (c) To see why \(S_i^2<0\) for all \(i\), we adapt the proof
  of {\cite[Lemma 3.3.1]{McDuffOpshtein}}. Suppose that \(S_i\)
  is a component which appears with multiplicity \(m_i\) in
  \(S\). Since the curves \(\sum_{j\neq i} m_jS_j\) and
  \(C+m_iS_i\) 
  are geometrically distinct, they intersect
  non-negatively; in fact, since the dual graph is connected
  with at least two vertices, \(S_i\) must intersect at least
  one \(S_j\) with \(j\neq i\), so they have positive
  intersection. Therefore
  \[0 < \sum_{j\neq i} m_jS_j\cdot (A+m_iS_i)=(A-m_iS_i)\cdot
    (A+m_iS_i)=-m_i^2S_i^2,\] which implies that \(S_i^2 <
  0\).

  (d) Since \(c_1(A)=2\), there must be at least one \(S_i\)
  with \(c_1(S_i)>0\); since \(c_1(S_i)=S_i^2+2\) by adjunction
  and \(S_i^2<0\), this must be a \(-1\)-sphere.
\end{proof}

\begin{remark}
We thank Weiyi Zhang for pointing out that this proposition can also be deduced from {\cite[Lemma 4.10 and Corollary 4.11]{LiZhang}}.
\end{remark}

\begin{lemma}\label{lma:broken_rulings_disjoint} 
 In a non-degenerate regulation, two broken rulings are either
  identical or disjoint.
\end{lemma}
\begin{proof}
  Suppose that \(S\) and \(S'\) are broken rulings of a
  non-degenerate regulation. Since \([S]=[S']\) and
  \([S]\cdot [S']=[S]^2=0\), the only way for \(S\) and \(S'\)
  to intersect is if they share a common component.
  Suppose that they are not identical (but share a common component). 
  Take a component $T$ of $S$ that is adjacent to a shared component, but is not contained in $S'$. Then $T\cdot S'>0$. 
  But non-degeneracy means that \(S'\) is homologous to a smooth
  ruling \(C\), and \(T\cdot C=0\) by Proposition
  \ref{prp:ruled_structure}(a), so \(T\cdot S'=0\). This yields
  a contradiction.
\end{proof}

\begin{remark} The existence of a smooth ruling is
necessary (see Example \ref{exm:regulations}(2)). We thank Weiyi Zhang for pointing out {\cite[Theorem 1.7]{LiLiWu}} which guarantees the existence of regulations with at least one smooth ruling on any rational symplectic 4-manifold other than \(\CP^2\).
\end{remark}

In the sequel, we look at spheres without marked point. 
Note that by Gromov compactness, there is a finite number of broken rulings 
of a non-degenerate regulation.

\begin{definition}
Given a
broken ruling \(S\in\overline{\MM}_{0,0}(A,J)\), pick a
contractible neighbourhood of \(S\) in
\(\overline{\MM}_{0,0}(A,J)\) which contains no other broken
rulings and let \(N\subset Y\) be the union of all rulings from
this neighbourhood in the moduli space. We call such an \(N\) a {\em
  tubular neighbourhood} of \(S\).
\end{definition}

\begin{lemma}\label{lma:no_new_rulings} Let \(S\) be a
  broken ruling of a non-degenerate regulation in a class \(A\)
  and let \(N\) be a tubular neighbourhood of \(S\). Let \(J'\)
  be a new tame almost complex structure which coincides with
  \(J\) outside \(N\) and for which \(S\) is still
  \(J'\)-holomorphic. Then there is still a non-degenerate
  \(J'\)-holomorphic regulation in the class \(A\) and it has no
  new broken rulings.
\end{lemma}
\begin{proof}
  Since the smooth rulings outside \(N\) are unaffected by the
  modification of the almost complex structure, they are still
  part of a non-degenerate \(J'\)-holomorphic regulation and so
  is \(S\) (because it is still part of the moduli space
  \(\overline{\MM}_{0,0}(A,J')\)). If there are any new broken
  rulings then they must be disjoint from both \(S\) and from
  the rulings outside \(N\), and so they must be contained in
  \(N\setminus S\). But \(N\setminus S\) is a sphere bundle over
  a non-compact surface (the punctured neighbourhood of
  \(S\in\overline{\MM}_{0,0}(A,J')\)) and hence \(H_2(N\setminus
  S;\ZZ)=\ZZ\), generated by the class \(A\). Therefore there
  are no homology classes into which the rulings can break
  (since \(A\) has minimal area amongst multiples of \(A\)).
\end{proof}

\begin{proposition}\label{prp:sikorav} Let \((Y,J)\)
  be an almost complex \(4\)-manifold tamed by a symplectic form
  \(\omega\). Suppose \(Y\) admits a non-degenerate
  \(J\)-holomorphic regulation in the class \(A\). Let
  \(U\subset Y\) be a subset on which \(J\) is integrable. We
  can find another almost complex structure \(J'\) tamed by
  \(\omega\) satisfying:
  \begin{itemize}
      \item[(i)] \(J'|_U=J|_U\).
      
      \item[(ii)] \(Y\) still admits a \(J'\)-holomorphic
        regulation, and has exactly the same broken rulings
        (identical as point sets) as the original regulation.
        
      \item[(iii)]  \(J'\) is integrable on a neighbourhood of all
        the broken rulings.
  \end{itemize}
\end{proposition}
\begin{proof}
  Sikorav {\cite[Theorem 3]{Sikorav}} shows that one can make
  \(J'\) integrable on a neighbourhood of any given holomorphic
  curve; Chen and Zhang {\cite[Appendix A]{ChenZhang}} modified
  this argument to produce a tamed \(J'\). Since we are
  modifying the almost complex structure only on a tubular
  neighbourhood of the broken rulings, Lemma
  \ref{lma:no_new_rulings} guarantees that we do not produce any
  new broken rulings.
\end{proof}

\begin{corollary} The dual graph of a broken ruling in
  a non-degenerate regulation is obtained from the
  graph \begin{tikzpicture}[baseline=0em]\node at (0,0)
    {\(\bullet\)};\node at (0,0) [above]
    {\(0\)};\end{tikzpicture} by a sequence of the following
  moves (combinatorial blow-ups):
  \begin{itemize}
      \item[(1)] Add a new vertex labelled \(-1\) connected by a
        single edge to an existing vertex \(i\), whose label drops
        by \(1\).
        
      \item[(2)] Add a new vertex labelled \(-1\) at the middle of
        an edge connecting two vertices \(i\) and \(j\), whose
        labels both drop by \(1\).
  \end{itemize}
\end{corollary}
\begin{proof}
  We know from Proposition \ref{prp:ruled_structure}(d) that any
  broken ruling \(S\) contains an embedded \(-1\)-sphere
  \(E\). Once we have made the almost complex structure
  integrable near \(S\), we can blow down \(E\) holomorphically
  to obtain a new almost complex manifold with a non-degenerate
  regulation. Iterate this procedure until the ruling becomes
  smooth; then its dual graph is a single vertex labelled
  \(0\). Inverting this process, at each step we blow up a point
  which either lives at a smooth point of the ruling or a node
  of the ruling, which changes the dual graph by a combinatorial
  blow-up of type 1 or 2 respectively.
\end{proof}

The possible dual graphs with at most three vertices are shown in 
Figure~\ref{fig:dual_graphs}. 

\begin{remark}
\label{rmk:dual_graphs_ZCFs} 
Note that by {\cite[Lemma 2.2]{Lisca1}} chain-shaped dual graphs obtained from iterated combinatorial blow-up are precisely the {\em zero continued fractions}; for example \[2-\frac{1}{1-\frac{1}{2}} = 0.\]
\end{remark}
\begin{figure}[htb]
  \begin{center}
    \begin{tikzpicture}
      \node at (0,0) {\(\bullet\)};
      \node at (0,0) [above] {\(0\)};
      \begin{scope}[shift={(2,0)}]
        \node (a1) at (0,0) {\(\bullet\)};
        \node at (0,0) [above] {\(-1\)};
        \node (b1) at (1,0) {\(\bullet\)};
        \node at (1,0) [above] {\(-1\)};
        \draw (a1) -- (b1);
      \end{scope}
      \begin{scope}[shift={(5,0)}]
        \node (a2) at (0,0) {\(\bullet\)};
        \node at (a2) [above] {\(-1\)};
        \node (b2) at (1,0) {\(\bullet\)};
        \node at (b2) [above] {\(-2\)};
        \node (c2) at (2,0) {\(\bullet\)};
        \node at (c2) [above] {\(-1\)};
        \draw (a2) -- (b2);
        \draw (b2) -- (c2);
      \end{scope}
      \begin{scope}[shift={(9,0)}]
        \node (a3) at (0,0) {\(\bullet\)};
        \node at (a3) [above] {\(-2\)};
        \node (b3) at (1,0) {\(\bullet\)};
        \node at (b3) [above] {\(-1\)};
        \node (c3) at (2,0) {\(\bullet\)};
        \node at (c3) [above] {\(-2\)};
        \draw (a3) -- (b3);
        \draw (b3) -- (c3);
      \end{scope}
    \end{tikzpicture}
    \caption{The dual graphs for broken rulings with at most
      three irreducible components.}
    \label{fig:dual_graphs}
  \end{center}
\end{figure}
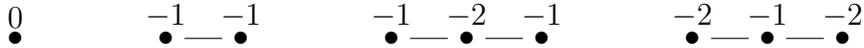

\begin{corollary}\label{cor:ruling_with_one_exc_sphere}
    If \(S\) is a broken ruling of a non-degenerate regulation
    and \(S\) has precisely one \(-1\)-sphere amongst its
     irreducible components, then the same is true of all its
     blow-downs until we
    reach \begin{tikzpicture}[baseline=0ex]\node (a1) at (0,0)
    {\(\bullet\)}; \node at (0,0) [above] {\(-1\)}; \node (b1)
    at (1,0) {\(\bullet\)}; \node at (1,0) [above] {\(-1\)};
    \draw (a1) -- (b1);
    \end{tikzpicture}. Blowing back up, the next dual graph must be \begin{tikzpicture}[baseline=0ex]\node (a1) at (0,0)
    {\(\bullet\)}; \node at (0,0) [above] {\(-2\)}; \node (b1)
    at (1,0) {\(\bullet\)}; \node (c1) at (2,0) {\(\bullet\)}; \node at (1,0) [above] {\(-1\)}; \node at (2,0) [above] {\(-2\)};
    \draw (a1) -- (b1);\draw (b1) -- (c1);
    \end{tikzpicture}, and all subsequent ones are obtained from this by combinatorially blowing-up a point on the \(-1\)-sphere; for example the next step might be \begin{tikzpicture}[baseline=0ex] \node (a) at (0,0) {\(\bullet\)}; \node (b) at (1,0) {\(\bullet\)}; \node (c) at (2,0) {\(\bullet\)}; \node (d) at (3,0) {\(\bullet\)};\draw (a) -- (b); \draw (b) -- (c);\draw (c) -- (d); \node at (a) [above] {\(-2\)};\node at (b) [above] {\(-2\)};\node at (c) [above] {\(-1\)};\node at (d) [above] {\(-3\)};\end{tikzpicture} or  
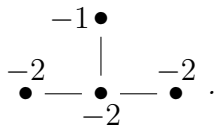
\begin{tikzpicture}[baseline=0ex]\node (a1) at (0,0)
    {\(\bullet\)}; \node at (0,0) [above] {\(-2\)}; \node (b1)
    at (1,0) {\(\bullet\)}; \node (c1) at (2,0) {\(\bullet\)}; \node at (1,0) [below] {\(-2\)}; \node at (2,0) [above] {\(-2\)};
    \draw (a1) -- (b1);\draw (b1) -- (c1);\node (d) at (1,1) {\(\bullet\)};\draw (b1) -- (d);\node at (d) [left] {\(-1\)};
    \end{tikzpicture}.
\end{corollary}

\section{Pin-ellipsoids in the complex projective plane}
\label{sct:pin_in_p2}

\subsection{Goal and strategy}
\label{sct:goals}

In what follows, we specialise to the case where
\(X=\CP^2\) equipped with the Fubini--Study form giving a
line area \(1\). 
In this preliminary section, we will outline the key technical results needed to prove Theorems \ref{thm:uniqueness}, \ref{thm:Markovstairs}, \ref{thm:twoball} and Corollary \ref{cor:threeball}.

Assume that we have a symplectic embedding
\(\iota\colon E_{p,q}(\alpha,\beta)\hookrightarrow X\) for some \(p,q,\alpha,\beta\), write \(U\) for the image of this embedding and \(V=X\setminus U\). 
Choose a Delzant pavilion \(\bm{\rho},\bm{\lambda}\) and equip the pavilion blow-up \(\til{X}\coloneqq \Pav^\iota_{\bm{\rho},\bm{\lambda}}\left(X\right)\)  
with an almost complex structure \(\til{J}\) and symplectic form \(\til{\omega}\) as in Definition \ref{dfn:pavilion_blow_up}.  
Let \(C_1,\ldots,C_m\) be the exceptional curves of the pavilion.

\begin{theorem}\label{thm:input_from_es}
    Assume that \(\til{J}\) is chosen generically on \(V\) and that the Delzant pavilion is minimal. 
    Then there exists a smoothly embedded \(\til{J}\)-holomorphic sphere \(\til{C}\subset\til{X}\) with
    \[\til{C}^2=0\quad\mbox{and}\quad \til{C}\cdot C_{j}=\begin{cases} 0 \mbox{
    unless }j=i_\cul,\\ 1 \mbox{ if }j=i_\cul.\end{cases}\]  
\end{theorem}

\begin{remark}
    Note that Theorem \ref{thm:input_from_es} is obvious for a visible embedding $\iota^{\vis}$: one can then find $\widetilde C$ over a segment perpendicular 
    to the line of $C_{i_{\cul}}$ and to the edge $\edge_1$, cf.\ Figure~\ref{fig:wps}.
    Theorem \ref{thm:input_from_es} follows easily from the results of Evans and
    Smith {\cite[Theorem~4.15]{ES}}, 
    as we will explain in Appendix~\ref{app:es_proof}, 
    but we will give an alternative proof below
    which avoids working with orbifold holomorphic curves. 
    As a consequence we will give a more direct, orbifold-free proof of
    the fact that \(p\) is a Markov number and \(q\) is one of its companion numbers
    (though, upon comparison with the orbifold proof from~\cite{ES}, it is clear
    that the two proofs are doing the same thing in different language).
\end{remark}

\begin{remark}
    In the case of two disjointly embedded pin-ellipsoids, we perform a pavilion blow-up as above at both of them to obtain a space~$\til{X}$. By similar arguments, we obtain the existence of a holomorphic sphere~\(\til{C}\) satisfying:
    \begin{itemize}
        \item \(\til{C}^2=-1\),
        
        \item \(\til{C}=c_0\eE+\sum
        c_{1,i}C_{1,i}+\sum c_{2,i}C_{2,i}\) where \(c_0\) is the
        smaller of the two solutions to
        \[c_0^2+p_1^2+p_2^2=3c_0p_1p_2.\]
        
        \item \(\til{C}\) intersects each Wahl chain in the pavilion blow-up once transversely at a point belonging to one of the curves at the end of the chain.
    \end{itemize}
    This curve will be used in Section \ref{sct:twoball} to prove Theorem \ref{thm:twoball} and Corollary \ref{cor:threeball}. Again, this curve is visible (living over an edge) if both pin-ellipsoids are visible.
\end{remark}

Theorem \ref{thm:input_from_es} implies that \(\til{X}\) admits a regulation. The next result gives some properties of this regulation.

\begin{theorem}\label{thm:regulation_properties}
  Assume that \(\til{J}\) is chosen generically on \(V\) and that the Delzant pavilion is minimal. Then the curve \(\til{C}\) from Theorem \ref{thm:input_from_es} is part of a \(\til{J}\)-holomorphic regulation with the following properties:
  \begin{itemize}
      \item[(a)] The culet curve \(C_{i_\cul}\) is a section.
      \item[(b)]
      \begin{itemize}
          \item[(i)] If \(C_{i_\cul}^2=-4\) then there are no broken rulings.
          \item[(ii)] If \(C_{i_\cul}^2=-7\) then \(i_\cul=1\) (or \(m\)) and there is one broken ruling consisting of \(C_2\cup\cdots\cup C_m\cup E\) (respectively \(C_1\cup\cdots\cup C_{m-1}\cup E\)) where \(E\) is an embedded \(-1\)-sphere intersecting precisely one of the curves \(C_i\) in the broken ruling once transversely.
          \item[(iii)] If \(C_{i_\cul}^2=-10\) then there are two broken rulings:
          \[C_1\cup\cdots\cup C_{i_{\cul}-1}\cup E_1\quad \text{and} \quad C_{i_\cul+1}\cup\cdots\cup C_m\cup E_2\]
          where \(E_1\) and \(E_2\) are embedded \(-1\)-spheres which intersect precisely one of the curves \(C_i\) in the broken ruling once transversely.
      \end{itemize}
  \end{itemize}
  In all cases, the curves \(C_i\) which hit the \(-1\)-curves \(E\), \(E_1\) and \(E_2\) are completely determined by the numbers \(p\) and \(q\).
\end{theorem}

\begin{figure}[htb]
\begin{center}
    \begin{tikzpicture}
    \begin{scope}[shift={(-6,0)}]
    \node at (-1,0) {(a)};
    \draw (-1,-2.7) -- (3,-2.7) node (c) [pos=0.1] {};
    \node at (c) [above] {\(C_1\)};
    \node at (c) [below] {\(-4\)};
    \draw (0.3,0.3) -- (0.3,-3);
    \draw (0.8,0.3) -- (0.8,-3);
    \draw (1.3,-3) -- (1.3,0.3) node [pos=0.6,sloped,above=-1mm] {smooth rulings};
    \draw (1.8,0.3) -- (1.8,-3);
    \draw (2.3,0.3) -- (2.3,-3);
    \end{scope}
    \node at (-1,0) {(b)};
    \draw (-1,-2.7) -- (4,-2.7) node (a) [pos=0.1] {};
    \node at (a) [above] {\(C_1\)};
    \node at (a) [below] {\(-7\)};
    \begin{scope}[shift={(0,0.3)}]
    \draw (3.7,0.3) to[out=-120,in=120] (3.7,-1.3);
    \node at (2.8,-0.5) [right] {\(C_4\)};
    \draw (3.7,-0.7) to[out=-120,in=120] (3.7,-2.3);
    \node at (2.8,-1.2) [right] {\(C_3\)};
    \draw (3.7,-1.7) to[out=-120,in=120] (3.7,-3.3);
    \node at (2.8,-2.5) [right] {\(C_2\)};
    \draw (2.9,-1.5) -- (5,-1.5) node (e) [pos=0.75] {};
    \node at (e) [above] {\(E\)};
    \node at (e) [below] {\(-1\)};
    \end{scope}
    \draw (0.3,0.3) -- (0.3,-3);
    \draw (0.8,0.3) -- (0.8,-3);
    \draw (1.3,-3) -- (1.3,0.3) node [pos=0.6,sloped,above=-1mm] {smooth rulings};
    \draw (1.8,0.3) -- (1.8,-3);
    \draw (2.3,0.3) -- (2.3,-3);
    \node (b) at (1,1) {broken ruling};
    \draw[->] (b) -- (3.4,0.2);
    \begin{scope}[shift={(-6,-4)}]
    \node at (-1,0) {(c)};
    \draw (0,-6) -- (10,-6) node (a) [pos=0.5] {};
    \node at (a) [above] {\(C_7\)};
    \node at (a) [below] {\(-10\)};
    \draw (9.7,0.3) to[out=-120,in=120] (9.7,-2.3);
    \node at (8.4,-1) [right] {\(C_{10}\)};
    \draw (9.7,-1.7) to[out=-120,in=120] (9.7,-4.3);
    \node at (8.7,-2.6) [right] {\(C_9\)};
    \draw (9.7,-3.7) to[out=-120,in=120] (9.7,-6.3);
    \node at (8.6,-5) [right] {\(C_{8}\)};
    \draw (9,-3) -- (11,-3) node (e) [pos=0.75] {};
    \node at (e) [above] {\(E_2\)};
    \node at (e) [below] {\(-1\)};
    \draw (0.3,0.3) to[out=-60,in=60] (0.3,-1.3);
    \draw (0.3,-0.7) to[out=-60,in=60] (0.3,-2.3);
    \draw (0.3,-1.7) to[out=-60,in=60] (0.3,-3.3);
    \draw (0.3,-2.7) to[out=-60,in=60] (0.3,-4.3);
    \draw (0.3,-3.7) to[out=-60,in=60] (0.3,-5.3);
    \draw (0.3,-4.7) to[out=-60,in=60] (0.3,-6.3);
    \node at (0.9,-0.5) {\(C_1\)};\node at (0.8,-1.3) {\(C_2\)};\node at (0.9,-2.5) {\(C_3\)};
    \node at (0.9,-3.5) {\(C_4\)};\node at (0.9,-4.5) {\(C_5\)};
    \node at (0.9,-5.5) {\(C_6\)};\node at (0,-0.5) {\(-5\)};
    \draw (0.8,-1.5) -- (-1,-1.5) node (e1) [pos=0.75] {};
    \node at (e1) [above] {\(E_1\)};
    \node at (e1) [below] {\(-1\)};
    \draw (2.5,-6.3) -- (2.5,-1);
    \draw (3.5,-6.3) -- (3.5,-1);
    \draw (4.5,-6.3) -- (4.5,-1) node [midway,sloped,above] {smooth rulings};
    \draw (5.5,-6.3) -- (5.5,-1);
    \draw (6.5,-6.3) -- (6.5,-1);
    \draw (7.5,-6.3) -- (7.5,-1);
    \node (d) at (5,0) {broken rulings};
    \draw[->] (d) -- (1.5,-0.5);
    \draw[->] (d) -- (8.5,-0.5);
    \end{scope}
    \end{tikzpicture}
    \caption{Examples illustrating the structure of the regulation described in Theorem \ref{thm:regulation_properties} for the cases (a) weight \(4\) (\(p=2,q=1\)), (b) weight \(7\) (\(p=5,q=1\)) and (c) weight \(10\) (\(p=29,q=7\)). 
    Curves are labelled with their name and self-intersection, except vertical straight lines (which represent smooth rulings with square \(0\)) and \(-2\)-spheres whose self-intersection is omitted.}
    \label{fig:regulations}
\end{center}
\end{figure}
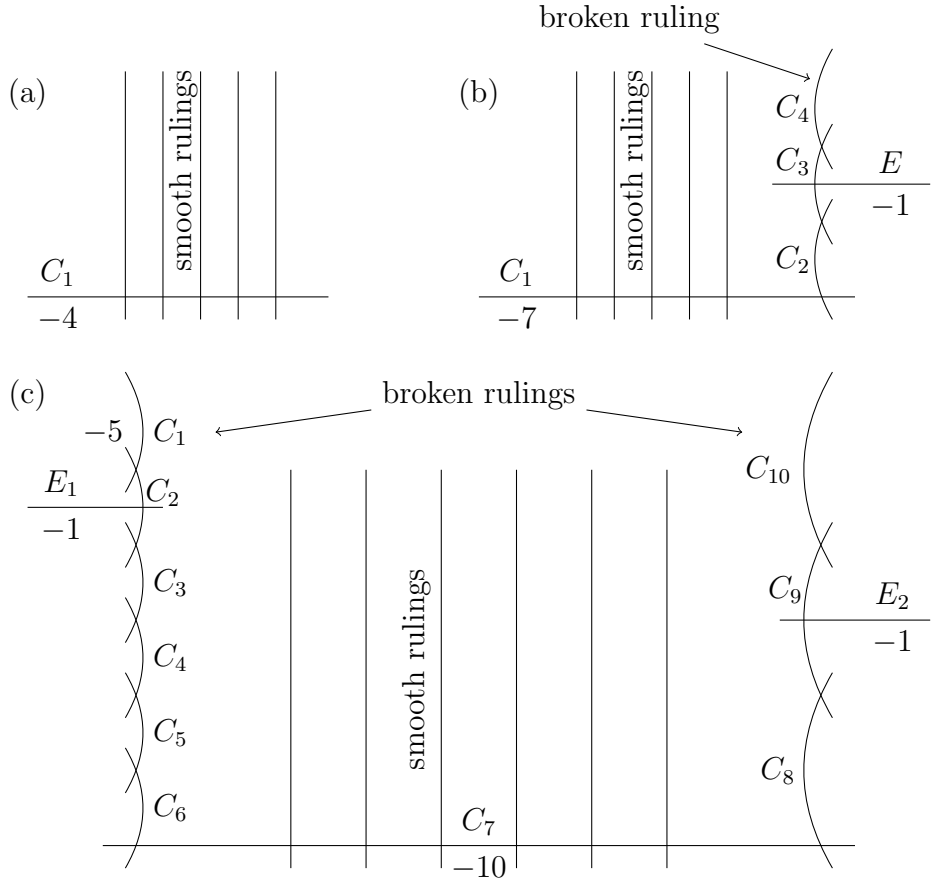

\begin{corollary}\label{cor:non_minimal_regulations}
  If the pavilion is not minimal then we will still find a regulation with at most two broken rulings. 
  The possible structures of the broken rulings is described in the proof and illustrated in Figure~\ref{fig:non_minimal_regulations}.
\end{corollary}
\begin{proof}
     Recall from Remark \ref{rmk:HJ} that a non-minimal pavilion is obtained from a minimal one by a sequence of toric blow-ups, corresponding to adding further rays to the fan~\(\bm{\rho}\). 
     Since these blow-ups occur at intersection points between the curves \(C_1,\ldots,C_m\), they are disjoint from a general ruling of the regulation, so the proper transform of a general ruling is still a holomorphic curve of square zero, so the result is still ruled and the culet curve\footnote{Recall from Remark~\ref{rmk:non_minimal_culet} that the culet curve for a non-minimal resolution is the proper transform of the culet curve from the minimal resolution.} is still a section. 

     The broken rulings are given by taking the total transforms of either an existing broken ruling (we call this Type I) or of a smooth ruling (which we call Type II). 
     Since the blow-ups are toric, the Type II broken rulings are chain-shaped and their dual graphs correspond to zero continued fractions as in Remark~\ref{rmk:dual_graphs_ZCFs}. 
     The dual graphs for Type I broken rulings still have the form of a tree comprising a chain of spheres from the pavilion together with a \(-1\)-sphere attached somewhere along the chain, but now the resolution curves from the pavilion can be \(-1\)-curves. 
     See Figure \ref{fig:non_minimal_regulations} for an example of a non-minimal pavilion blow-up with one Type I and one Type II broken ruling.
\end{proof}

\begin{figure}[htb]
\centering
\begin{tikzpicture}
  \filldraw[fill=lightgray] (0,0) -- (1,-1) -- ++ (1,-0.8)  -- ++ (1,-0.7) -- ++ (1,-0.6) -- ++ (1,-0.5)  -- ++ (1,-0.4) -- ++ (1,-0.3) -- ++ (1,0) -- ++ (1,0.8) -- ++ (1,1.3) -- ++ (1,1.5) -- ++ (0.2,0.7) -- cycle;
  \draw (0,0) -- (1,-1) node {\(\bullet\)} -- ++ (1,-0.8) node {\(\bullet\)} -- ++ (1,-0.7) node {\(\bullet\)} -- ++ (1,-0.6) node {\(\bullet\)} -- ++ (1,-0.5) node {\(\bullet\)} -- ++ (1,-0.4) node {\(\bullet\)} -- ++ (1,-0.3) node {\(\bullet\)} -- ++ (1,0) node {\(\bullet\)} -- ++ (1,0.8) node {\(\bullet\)} -- ++ (1,1.3) node {\(\bullet\)} -- ++ (1,1.5) node {\(\bullet\)} -- ++ (0.2,0.7) node (a) {} -- cycle;
  \draw[dashed] (0,0) -- (2.6,-0.7*2.6) node (b) {\(\times\)};
  \draw[dashed] (a.center) -- ++ (-2,-1.3*2) node (c) {\(\times\)};
  \draw[very thick] (1,-1) -- ++ (1,-0.8) node [midway,below] {\(C_1\)} -- ++ (1,-0.7) node [midway,below] {\(C_2\)} -- ++ (1,-0.6) node [midway,below] {\(C_3\)} -- ++ (1,-0.5) node [midway,below] {\(C_4\)} -- ++ (1,-0.4) node [midway,below] {\(C_5\)} -- ++ (1,-0.3) node [midway,below] {\(C_6\)} -- ++ (1,0) node [midway,below] {\(C_7\)} -- ++ (1,0.8) node [midway,below right] {\(C_8\)} -- ++ (1,1.3) node [midway,right] {\(C_9\)} -- ++ (1,1.5) node [midway,right] {\(C_{10}\)};
  \draw[very thick] (b.center) -- ++ (-0.24,-0.24);
  \draw[very thick] (c.center) -- ++ (0.32,-0.22) node (d) {};
  \node (e1) at (1.4,-2.9) {\(E_1\)};
  \node (e2) at (10.5,-3.5) {\(E_2\)};
  \draw[->] (e1) to[out=60,in=180] (2.45,-1.85);
  \draw[->] (e2) to[out=180,in=-120] (9.3,-2.8);
  \end{tikzpicture}
  \caption{An almost toric base diagram for a minimal pavilion blow-up of a visible \(E_{29,7}(\alpha,\beta)\subset\CP^2\). This is almost toric representation of Figure~\ref{fig:regulations}~(c).}
\label{fig:visible_contact_hypersurface_2}
\end{figure}
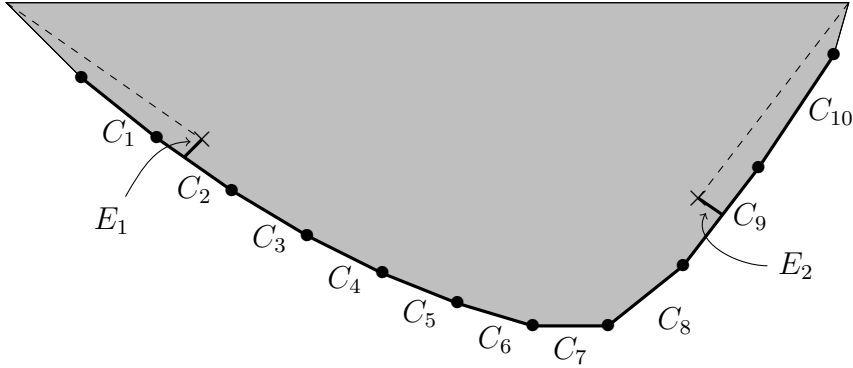

\begin{definition}\label{dfn:tree}
    We will define a tree \(\ct\) of symplectic curves in \(\til{X}\) as follows.
    \begin{itemize}
    \item If there are two broken rulings then \(\ct\) consists of \(C_{i_\cul}\) and the broken rulings.
    \item If there is precisely one broken ruling then we choose an unbroken ruling \(F\) and take \(\ct\) to be the union of \(F\), \(C_{i_\cul}\) and the broken ruling.
    \item If there are no broken rulings (which can happen only if \(w=4\) and the pavilion is minimal) then we choose two unbroken rulings \(F_1\) and \(F_2\) and take \(\ct\) to be the union \(F_1\cup F_2\cup C_1\).
    \end{itemize}
    Denote this splitting of $\ct$ into three parts by $\ct = \cf_1 \cup C_{i_{\cul}} \cup \cf_2$.
\end{definition}

\begin{corollary}\label{cor:mimimal_complement}
 Let \(\ct\) be the tree of curves from Definition \ref{dfn:tree}. The complement $\til{X} \setminus \ct$ is a minimal symplectic manifold.
\end{corollary}

\begin{proof}
    The space $\til{X}$ is obtained from a Hirzebruch surface by a sequence of blow-ups, 
    since by Corollary~\ref{cor:ruling_with_one_exc_sphere} we can consecutively blow down the $-1$-curves contained in the broken rulings until we arrive at a situation in which there are no broken rulings. 
    Then, $\til{X} \setminus \ct$ is identified with a subset in the complement of a section in a Hirzebruch surface.
    Since we can always arrange the consecutive blow-downs to happen away from the culet curve, 
    $\til{X}\setminus \ct$ lives in the normal neighbourhood of a symplectic sphere with positive self-intersection number.
    Therefore, $\til{X} \setminus \ct$ is minimal.
\end{proof}

\begin{remark}\label{rmk:min_mod}
It is possible to compute the minimal model that we mentioned in the proof of Corollary~\ref{cor:mimimal_complement}: it is diffeomorphic to the Hirzebruch surface \(\FF_w\) whose negative section has square \(-w\). This is diffeomorphic to either \(S^2\times S^2\) (if \(w\) is even) or \(\CP^2\#\overline{\CP}^2\) (if \(w\) is odd).
\end{remark}

\begin{figure}[htb]
\begin{center}
    \begin{tikzpicture}
    \node at (-1,0) {};
    \draw (-1,-2.7) -- (4,-2.7) node (a) [pos=0.53] {};
    \node at (a) [above] {\(C_4\)};
    \node at (a) [below] {\(-9\)};
    \begin{scope}[shift={(0,1.3)}]
    \draw (3.7,0.3) to[out=-120,in=120] (3.7,-1.3);
    \node at (3.1,-0.5) {\(C_8\)};
    \draw (3.7,-0.7) to[out=-120,in=120] (3.7,-2.3);
    \node at (3.1,-1.2) {\(C_7\)};
    \draw (3.7,-1.7) to[out=-120,in=120] (3.7,-3.3);
    \node at (3.1,-2.5) {\(C_6\)};
    \node at (3.8,-2.5) {\(-3\)};
    \draw (3.7,-2.7) to[out=-120,in=120] (3.7,-4.3);
    \node at (3.1,-3.5) {\(C_5\)};
    \node at (3.8,-3.5) {\(-1\)};
    \draw (2.9,-1.5) -- (5,-1.5) node (e) [pos=0.75] {};
    \node at (e) [above] {\(E\)};
    \node at (e) [below] {\(-1\)};
    \end{scope}
    \begin{scope}[shift={(0,0.3)}]
    \draw (-0.7,0.3) to[out=-60,in=60] (-0.7,-1.3);
    \node at (-0.1,-0.5) {\(C_1\)};    
    \draw (-0.7,-0.7) to[out=-60,in=60] (-0.7,-2.3);
    \node at (-1,-1.5) {\(C_2\)};
    \node at (-0.1,-1.5) {\(-1\)};
    \draw (-0.7,-1.7) to[out=-60,in=60] (-0.7,-3.3);
    \node at (-0.1,-2.5) {\(C_3\)};
    \end{scope}
    \begin{scope}[shift={(0.3,0)}]
    \draw (0.3,1.3) -- (0.3,-3);
    \draw (0.8,-3) -- (0.8,1.3) node [pos=0.6,sloped,above=-1mm] {smooth rulings};
    \draw (1.8,1.3) -- (1.8,-3);
    \draw (2.3,1.3) -- (2.3,-3);
    \end{scope}
    \node at (3.2,2) {Type I};
    \node at (-0.75,1) {Type II};
    \end{tikzpicture}
    \caption{An example illustrating the structure of the regulation for a non-minimal pavilion blow-up, as described in Corollary~\ref{cor:non_minimal_regulations}. This example is obtained from Figure~\ref{fig:regulations}(b) by blowing up four times: once where the original broken ruling intersects the culet curve (yielding the Type I broken ruling) and three more times on a general smooth ruling (yielding the Type II broken ruling). Curves are labelled with their name and self-intersection, except vertical straight lines (which represent smooth rulings with square~\(0\)) and~\(-2\)-spheres whose self-intersection is omitted.}
    \label{fig:non_minimal_regulations}
\end{center}
\end{figure}
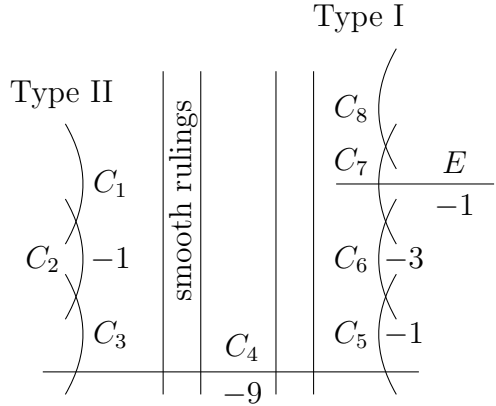

In the next few sections, we will show how Corollary \ref{cor:non_minimal_regulations} implies the Isotopy and Staircase Theorems~\ref{thm:uniqueness} and~\ref{thm:Markovstairs} from the introduction. 
In Section~\ref{subsec:ES}, we will prove Theorem~\ref{thm:input_from_es}.  
In Section~\ref{sct:twoball}, we will modify the proof to handle multiple pin-balls and also give a proof of the Two Pin-Ball Theorem.
Finally, in Section~\ref{subsec:regulation_properties} we will prove Theorem~\ref{thm:regulation_properties}.

\subsection{Proof of the Isotopy Theorem (Theorem \ref{thm:uniqueness})}
\label{sct:isotopy_proof}

In this section, we explain how Corollary \ref{cor:non_minimal_regulations} solves the isotopy problem for pin-ellipsoids in \(\CP^2\).
We begin by stating a further corollary of Corollary~\ref{cor:non_minimal_regulations}.

\begin{corollary}\label{cor:manetti_normal_form}
    Let \(\iota\colon E_{p,q}(\alpha,\beta)\hookrightarrow \CP^2\) be a symplectic embedding, let \(\bm{\rho},\bm{\lambda}\) be a (possibly non-minimal) choice of pavilion and let \(\til{X}\) be the pavilion blow-up along \(\iota\). Let \(\ct\) be the tree of symplectic spheres from Definition \ref{dfn:tree}. Then the complement \(\til{X}\setminus\ct\) is diffeomorphic to a disc bundle over the annulus and the self-intersections and symplectic areas of the spheres in \(\ct\) depend only on \(p,q\) and the choice of pavilion.
\end{corollary}

\begin{proof}
  Since \(\ct\) consists of a section and two rulings of the regulation on \(\til{X}\), including all broken rulings, the restriction of the regulation to \(\til{X}\setminus \ct\) exhibits this complement as a disc bundle over the annulus.

  The tree comprises the linear chain \(C_1\cup \ldots\cup C_m\) together with some additional curves. These additional curves are of the following types:
  \begin{itemize}
  \item a proper transform of a smooth ruling in the minimal model, coming from a broken ruling of Type II;
  \item a \(-1\)-curve attached to one of the curves \(C_{i}\) in the chain, coming from a broken ruling of Type I;
  \item a smooth ruling, if there are fewer than two broken rulings.
  \end{itemize}
  The dual graph of the linear chain is determined by the choice of pavilion, as are the positions in the dual graph where the additional curves attach. The symplectic areas of the curves \(C_i\) are also determined by the pavilion. Writing the additional curves in the form given by Equation \eqref{eq:splitting}, we can find the coefficients by computing their intersection numbers with the curves \(C_i\) and with a smooth ruling;
  then their symplectic areas can be computed using the fact that \(\int_{\eE}\til{\omega}=1/p^2\) and the knowledge of the symplectic areas of the \(C_i\) coming from the pavilion.
\end{proof}

For the rest of this subsection we assume that \(\iota,\iota'\colon E_{p,q}(\alpha,\beta)\hookrightarrow \CP^2\) 
are symplectic embeddings. 
Let \(\bm{\rho},\bm{\lambda}\) be a choice of pavilion for \(E_{p,q}(\alpha,\beta)\). 
As before we denote the symplectic manifolds obtained by performing the pavilion blow-up along these embeddings by $\til{X} = \Pav^{\iota}_{\bm{\rho},\bm{\lambda}}(\CP^2)$ and 
$\til{X}' = \Pav^{\iota'}_{\bm{\rho},\bm{\lambda}}(\CP^2)$.

\begin{corollary}\label{cor:symplecto_of_blow_ups} 
    There exists a symplectomorphism
    \[\Psi \colon \big((\til{X},\omega),\ct\big) \to \big((\til{X}',\omega'),\ct')\] 
    which carries the tree~$\ct$ of symplectic spheres in~$\til{X}$ 
    coming from Corollary~\ref{cor:manetti_normal_form} to the 
    tree~$\ct'$ of symplectic spheres in~$\til{X}'$.
\end{corollary}

\begin{proof}\let\qed\relax
Since the symplectomorphism type of the trees $\ct, \ct'$ of symplectic spheres depends only 
on the pavilion, 
the symplectic neighbourhood theorem gives symplectic embeddings 
\(\psi \colon \nu \hookrightarrow \til{X}\) and \(\psi' \colon \nu \hookrightarrow \til{X}'\) 
of a plumbing $\nu$ of symplectic disk bundles over symplectic spheres into both pavilion blow-ups $\til X, \til X'$.
Our goal is to extend the symplectomorphism
$\psi' \circ \psi^{-1} \colon \psi (\nu) \to \psi'(\nu)$ to a symplectomorphism
$\Psi \colon (\til{X},\omega) \to (\til{X}', \omega')$.
Recall from Definition~\ref{dfn:tree} that the trees split as a union \(\ct = \cf_1 \cup C_{i_\cul} \cup \cf_2\) 
and \(\ct' = \cf_1' \cup C_{i_\cul}' \cup \cf_2'\)
of two rulings and the culet curve (a section). 
By the proof of Corollary~\ref{cor:non_minimal_regulations}, not only the two trees but also these two decompositions are homeomorphic. 
In particular, $\psi' \circ \psi^{-1}$ takes
\(C := C_{i_\cul}\) to \(C' := C'_{i_\cul}\).

{\bf Step 1: Choice of regulations.}
Recall that the regulation of $\widetilde{X}$ was constructed with respect to an almost complex structure~$J$. 
Denote the restriction of this almost complex structure to ~$\psi (\nu)$ by~$j$. 
Extend $(\psi' \circ \psi^{-1})_* j$ to an almost complex structure~$J'$ on~$\til{X}'$ and consider the regulation of~$\til{X}'$ associated to~$J'$.
Note that $\psi' \circ \psi^{-1}$ takes rulings to rulings in a neighbourhood of $\cf_1$ and~$\cf_2$.

{\bf Step 2: Splitting $\til X$ by splitting $C$.}
Since \(C\) is a section, the regulation determines a continuous projection map $\pi \colon \til X \to C$ that is smooth away from the rulings $\cf_1$, \(\cf_2\). 
Let $p_1,p_2$ be the images under~$\pi$ of $\cf_1$ and~$\cf_2$. 
Choose small disjoint embedded closed curves $\gamma_1, \gamma_2$ in $C$ around~$p_1,p_2$. 
They decompose $C$ into a closed disc $D_1$ around~$p_1$, an annulus~$A$, and a closed disc~$D_2$ around~$p_2$.
Since $\pi$ is continuous, we can choose the $\gamma_i$ so close to~$p_i$
that both $\pi^{-1}(D_i)$ are contained in~$\psi(\nu)$.
In this way we obtain a splitting 
$$
\til{X} = \pi^{-1}(D_1) \cup \pi^{-1}(A) \cup \pi^{-1}(D_2) .
$$
Choosing $\gamma_i'$ to be the curves $(\psi' \circ \psi^{-1}) (\gamma_i)$ in the section $C'$ of $\til{X}'$,  we obtain analogous splittings $C' = D_1' \cup A' \cup D_2'$ and 
$$
\til{X}' = (\pi')^{-1}(D_1') \cup (\pi')^{-1}(A) \cup (\pi')^{-1}(D_2) .
$$
Note that $\psi' \circ \psi^{-1}$ restricts to a fibre-preserving map
$\pi^{-1}(D_i) \to \pi^{-1}(D_i')$
and in particular to a fibre-preserving map
$\pi^{-1}(\gamma_i) \to (\pi')^{-1}(\gamma_i')$.

{\bf Step 3: A good chart for $\omega$ near $\pi^{-1}(\gamma_i)$.}
Fix $i \in \{1,2\}$.
The $S^2$-bundle $\pi \colon \pi^{-1}(\gamma_i) \to \gamma_i$ 
is the trivial $S^2$-bundle over the circle $\gamma_i$, because the monodromy around~$\gamma_i$ is symplectic and hence
preserves the orientation of the fibre.
The same argument proves the first of the following two assertions. 

\begin{itemize}
\item[(i)]
The monodromy of the bundle $\gamma_i \times S^2 \to \gamma_i$ 
in $\pi_0 (\Diff (S^2))$ is trivial.

\item[(ii)]
The symplectic normal bundle $\cn \to \gamma_i \times S^2$
with fibre $\{\theta\} \times (T_qS^2)^{\perp \omega}$ over 
$(\theta,q) \in \gamma_i \times S^2$ is trivial.
\end{itemize}

\noindent
{\it Proof of (ii):}
If $\cf_i$ is a broken ruling we successively blow down the exceptional divisors of the broken ruling $\pi^{-1}(p_i)$ as in $\S$~\ref{ss:broken}.
The blow-downs are supported away from the wall~$\pi^{-1}(\gamma_i)$.
Furthermore, the broken ruling~$\pi^{-1}(p_i)$ becomes one smooth fibre, 
and so the regulation over~$\pi^{-1}(D_i)$ becomes 
a trivial symplectic sphere bundle over a disc.
This implies assertion~(ii).
\end{proof}

The following lemma gives a one-sided normal form neighbourhood for the above loops of symplectic spheres.

\begin{lemma} \label{le:omega0}
Consider the closed disc $D$ with symplectic form $\omega_D = dx \wedge dy$,
the sphere~$S^2$ with an area form $\omega_{S^2}$ whose total area agrees with the $\omega$-area of a smooth fibre of~$\pi$, 
and take the product symplectic form 
$\omega_D \oplus \omega_{S^2}$ on~$D \times S^2$.
Then there exists a neighbourhood~$U$ of $\partial D \times S^2$ in~$D \times S^2$
and a diffeomorphism $\alpha$ from~$U$
to a neighbourhood of $\pi^{-1}(\gamma_i)$
in $\pi^{-1}(D_i)$ that pulls back $\omega$ to 
$\omega_D \oplus \omega_{S^2}$.
\end{lemma}

\begin{proof}
The claim would be well-known if we only had to deal with one sphere $\{\theta\} \times S^2$ instead of $\gamma_i \times S^2$.
In our case, we use assertions (i) and~(ii) above
to construct a smooth embedding $\underline \alpha  \colon \partial D \times S^2 \to X$ onto $\pi^{-1}(\gamma_i)$
such that $\underline \alpha^* \omega = (\omega_{D} \oplus \omega_{S^2}) |_{T(\partial D \times S^2)}$.
The neighbourhood theorem for hypersurfaces from~\cite{Go82} (or also Exercise~3.4.17 in~\cite{McSa17})
now implies that $\underline \alpha$ extends to the claimed 
diffeomorphism~$\alpha$ on a neighbourhood~$U$.  
\end{proof}

{\bf Step 4: Extending $\psi' \circ \psi^{-1}$.}
We first extend the symplectomorphism 
$$
\psi' \circ \psi^{-1} \colon 
\pi^{-1}(D_1) \cup \nu(C) \cup \pi^{-1}(D_2) \,\to\, \pi^{-1}(D_1') \cup \nu(C') \cup \pi^{-1}(D_2')
$$ 
to a diffeomorphism 
$\Psi_0 \colon \til X \to \til{X}'$. 
(This is possible since $\til X \setminus \ct$ and $\til{X}' \setminus \ct'$
are both disc bundles over an annulus by Corollary \ref{cor:manetti_normal_form}. We could assume that $\Psi_0$ preserves 
the fibres of $\pi$ and~$\pi'$, but this is not needed in the sequel.)
We then have the two symplectic forms $\omega$ and $\Psi_0^* \omega'$
on~$\til X$, that agree on $\pi^{-1}(D_1) \cup \nu(C) \cup \pi^{-1}(D_2)$.
We are left with showing that there exists a diffeomorphism $\rho$ of~$\til X$
that maps~$\ct$ to~$\ct'$ and satisfies
$\rho^* \Psi_0^* \omega' = \omega$.

View $S^2 \times S^2$ as $C \times S^2 = (D_1 \cup A \cup D_2) \times S^2$ where \(C\) is equipped with the symplectic form pulled back from \(\omega\).
By Lemma~\ref{le:omega0} we can construct two symplectic forms $\omega_1$ and~$\omega_2$ 
on~$S^2 \times S^2$ by taking on $A \times S^2$
the forms $\omega$ and~$\Psi_0^* \omega'$,  
but taking on $(D_1 \cup D_2) \times S^2$ the split form~$\omega_{D} \times \omega_{S^2}$.
By a Moser argument, we can find smaller closed discs $D_i^{<} \subset D_i$ and a diffeomorphism~$\mu$ of~$S^2 \times S^2$ such that:
\begin{itemize}
\item \(\mu\) is the identity on $(D_1^{<} \cup D_2^{<}) \times S^2$,
\item \(\mu\) preserves \(C\subset S^2\times S^2\), and
\item \(\mu^*\omega_1\) is a split form near \(C\). Since \(\omega_1\) and \(\omega_2\) coincide near \(C\), this is also true of \(\mu^*\omega_2\).
\end{itemize}
Choose a closed neighbourhood $K$ of $C$ that is contained in this neighbourhood, 
and choose smaller closed discs $D_i^{\ll} \subset D_i^{<}$ around $p_i$. 
By Theorem~\ref{t:forms} below, 
there exists a diffeomorphism $\Phi$ of~$S^2 \times S^2$
that is the identity on $K \cup ((D_1^{\ll} \cup D_2^{\ll}) \times S^2)$ and pulls back 
$\mu^*\omega_2$ to~$\mu^*\omega_1$. 
We now take the diffeomorphism~$\rho$ of~$\til X$ to be 
$\mu \circ \Phi \circ \mu^{-1}$ on $\til X \setminus \pi^{-1}(p_1 \cup p_2) = (S^2 \times S^2) \setminus (\{p_1 \cup p_2\} \times S^2)$
and equal to the identity on the broken rulings $\pi^{-1}(p_1 \cup p_2)$.
Then $\Psi_0 \circ \rho$ still takes $\ct$ to~$\ct'$, and it pulls back $\omega'$ to~$\omega$.

We are left with proving the following improvement 
of the Gromov--McDuff Theorem {\cite[Theorem 9.4.7~(ii)]{McDuffSalamon}}. 

\begin{theorem} \label{t:forms}
Let $\omega_{a,b}$ be the usual split symplectic form on $S^2 \times S^2$
that gives the first factor area~$a$ and the second factor area~$b$. 
Represent $(S^2 \times S^2,\omega_{a,b})$ by its moment image as in Figure~\ref{fig:gromov_mcduff}.
Let $p_N$ and $p_S$ the North and South poles of~$S^2$,
and consider the three spheres 
$$
S_1 = S^2 \times \{p_S\}, \quad S_2^N = \{p_N\} \times S^2, \quad S_2^S = \{p_S\} \times S^2
$$ 
in~$S^2 \times S^2$.
Let $\omega$ be a symplectic form on $S^2 \times S^2$ that agrees with $\omega_{a,b}$
on a neighbourhood of $S_1 \cup S_2^N \cup S_2^S$. 
Then there exists a symplectomorphism $\Phi \colon (S^2 \times S^2, \omega_{a,b}) \to (S^2 \times S^2, \omega)$
that is the identity on a neighbourhood of $S_1 \cup S_2^N \cup S_2^S$. 
\end{theorem}

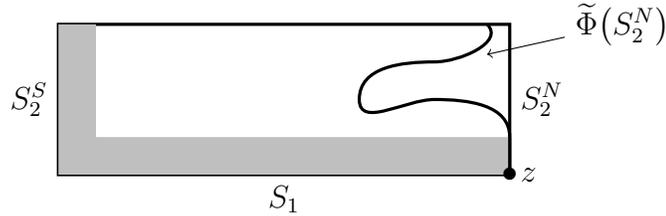
\begin{figure}[htb]
\centering
\begin{tikzpicture}
    \draw[very thick] (-3,0) -- (-3,-2) -- (3,-2) -- (3,0) -- cycle;
    \filldraw[fill=lightgray,draw=none] (-3,0) -- (-2.5,0) -- (-2.5,-1.5) -- (3,-1.5) -- (3,-2) -- (-3,-2) -- cycle;
    \draw[very thick] (3,-2) to[out=90,in=-90] (3,-1.5) to[out=90,in=0] (2,-1) to[out=180,in=-90] (1,-1) to[out=90,in=180] (2,-0.5) to[out=0,in=-45] (2.7,0);
    \node (a) at (4.5,0) {\(\widetilde{\Phi}\!\left(S^N_2\right)\)};
    \node (p) at (-4.5,0) {\phantom{\(\widetilde{\Phi}\!\left(S^N_2\right)\)}};
    \node at (3,-1) [right] {\(S^N_2\)};
    \node at (-3,-1) [left] {\(S^S_2\)};
    \node at (0,-2) [below] {\(S_1\)};
    \draw[->] (a) -- (2.7,-0.4);
    \node (b) at (3,-2) {\(\bullet\)};
    \node at (b) [right] {\(z\)};
  \end{tikzpicture}
  \caption{The moment rectangle for the manifold \(S^2\times S^2\), illustrating the proof of 
  Theorem~\ref{t:forms}. The symplectomorphism \(\widetilde{\Phi}\) coming from the Gromov--McDuff Theorem, fixes the shaded region pointwise. We must isotope it to something fixing a neighbourhood of~\(S^N_2\) 
  by first isotoping \(\widetilde{\Phi}\left(S^N_2\right)\) back to~\(S_2^N\), then fixing the symplectic normal bundle of~\(S^N_2\), 
  and finally straightening the symplectomorphism in a neighbourhood of~\(S^N_2\).}
\label{fig:gromov_mcduff}
\end{figure}    

We note that this theorem (and in fact also a stronger version with four spheres taken out) 
has been proved independently in~\cite[Lemma~3.1.3]{CGMM25}.

\begin{proof}[Proof of Theorem~\ref{t:forms}.]
According to \cite[Theorem 9.4.7]{McDuffSalamon} there exists a symplectomorphism 
$\widetilde \Phi \colon (S^2 \times S^2, \omega_{a,b}) \to (S^2 \times S^2, \omega)$
that is the identity on a neighbourhood of $S_1 \cup S_2^S$. 
Set $z = (p_N,p_S)= S_1 \cap S_2^N$.    
For $t \in [0,1]$ choose a smooth path~$J_t$ of $\omega$-compatible almost
complex structures such that
\begin{itemize}
\item[(i)]
$S_1$ and $S_2^S$ are $J_t$-holomorphic for all $t$,     
\item[(ii)]
$\widetilde \Phi (S_2^N)$ is $J_0$-holomorphic,
\item[(iii)]
$S_2^N$ is $J_1$-holomorphic.
\end{itemize}
Such a path exists since $S_1, S_2^S, S_2^N, \widetilde \Phi (S_2^N)$
are embedded $\go$-symplectic spheres and 
since $S_2^S$ is disjoint from $\widetilde \Phi (S_2^N)$ and 
$S_1$ intersects $S_2^N$ and $\widetilde \Phi (S_2^N)$
transversely, positively, and only in~$z$.

Let $S_2^t$ be the unique $J_t$-holomorphic sphere in class $A := [S_2^N] \in H_2(S^2\times S^2;\Z)$ passing through~$z$. 
This exists because \(A^2=0\) so the Gromov--Taubes invariant counting holomorphic spheres in this class with a single point constraint 
is equal to \(1\) and curves in this class cannot bubble, \(S_2^S\) 
being \(J_t\)-holomorphic.\footnote{This is because a \(J_t\)-holomorphic bubble curve will necessarily remain disjoint from \(S^S_2\), 
because $z$ is not contained in $S_2^S$ and by positivity of intersection, 
and must therefore live in a homology class which is a positive multiple of \(A\). 
But the area cannot be larger than that of \(S^N_2\), so there can be only one irreducible component.} 

The spheres $S_2^t$ depend smoothly on~$t$, because the path $J_t$ is smooth in~$t$ and 
each~$S_2^t$ is automatically regular (being a square zero sphere).
Now $\{S_2^t\}$ is a smooth path of embedded $\omega$-symplectic spheres
such that $S_2^0 = \widetilde \Phi (S_2^N)$ and $S_2^1 = S_2^N$.
By~(i), each sphere $S_2^t$ intersects $S_1$ transversally at~$z$ and is disjoint from~$S_2^S$.
We can therefore modify the path $\{S_2^t\}$ near~$z$ such that all $S_2^t$
coincide with $S_2^N$ near~$z$, 
and such that the $S_2^t$ are still disjoint from $S_2^S$ and intersect $S_1$ only in~$z$.

By~\cite[Proposition 0.3]{SiebertTian05} or~\cite[Lemma 5.16]{McLean12} there exists 
a Hamiltonian isotopy $\gf^t$ of $(S^2 \times S^2, \go)$ whose support
is contained in any given neighbourhood of $\bigcup_{t \in [0,1]} S_2^t$
such that $\gf^t (S_2^0) = S_2^t$.
In particular, we can assume that the support of~$\gf^t$ is disjoint from~$S_2^S$.
Since the $S_2^t$ already agree near~$z$, we can also assume that the support is disjoint
from~$S_1$.
Hence $\gf^1 \circ \widetilde \Phi$ is still the identity near $S_1 \cup S_2^S$
and takes $S_2^N$ to itself.
Let $\gs$ be the restriction of $\gf^1 \circ \widetilde \Phi$ to $S_2^N$.
Since $\gs$ fixes a neighbourhood of~$z$, it is the time-1 map of a Hamiltonian
isotopy $\gs^t$ of~$S_2^N$ that fixes a neighbourhood of~$z$.
Extend $\gs^t$ to a Hamiltonian isotopy~$\hat \sigma_t$ of~$(S^2 \times S^2, \go)$ 
that has support near $S_2^N$ and is the identity near~$z$.
Then $\widehat \Phi := \hat \sigma_1^{-1} \circ \gf^1 \circ \widetilde \Phi$
is a symplectomorphism $(S^2 \times S^2, \go_{a,b}) \to (S^2 \times S^2, \go)$ 
that is the identity near $S_2^S \cup S_1$ and is the identity along~$S_2^N$.

The differential of $\widehat \Phi$ at a point $p\in S_2^N$ with respect to the
symplectic splitting of $T_p (S^2 \times S^2)$ is of the form 
$$
d_p \widehat \Phi = 
\begin{pmatrix}
\id_2 & B \\ 0_2 & D 
\end{pmatrix} \colon
T_p S_2^N \oplus (T_p S_2^N)^{\omega} \to T_p S_2^N \oplus (T_p S_2^N)^{\omega}.
$$
Since $d_p \widehat \Phi$ is symplectic, it leaves $(T_p S_2^N)^{\omega}$
invariant. Hence $B=0_2$ and $D$ is symplectic.
The space ${\rm Symp}(2;\RR)$
deformation retracts onto $U(1) = S^1$.
Since the space of pointed maps from the 2-sphere to the circle is connected, 
we can isotopy $d\widehat \Phi |_{S_2^N}$ at the bundle level to 
the identity along~$S_2^N$. 
Using Moser's method as in the proof of the symplectic neighbourhood theorem, 
we can then find a symplectic isotopy from $\widehat \Phi$ to the identity {\it near}\/ $S_2^N$. 
Since the tubular neighbourhood of $S_2^N$ has trivial first homology, we can make this 
isotopy Hamiltonian. We finally multiply the Hamiltonian function $H_t$, 
that we may choose to vanish near~$z$,
with a function~$\chi$ that is $1$ on a small tubular neighbourhood of $S_2^N$
and vanishes outside a slightly larger tubular neighbourhood. 
The composition $\Phi := \phi_{\chi H} \circ \widehat \Phi$
is then a symplectomorphism $(S^2 \times S^2, \go_{a,b}) \to (S^2 \times S^2, \go)$
that is the identity {\it near}~$S_2^N$.
The proof of Theorem~\ref{t:forms} is complete.
\end{proof}

We are now ready to prove Theorem \ref{thm:uniqueness}, that we restate as:

\begin{corollary}[Hamiltonian uniqueness of pin-ellipsoids] \label{cor:deformation_uniqueness}
    Given two symplectic embeddings \(\iota, \iota' \colon E_{p,q}(\alpha,\beta) \hookrightarrow \CP^2\)
    there exists a Hamiltonian diffeomorphism~$\Phi$ of \(\CP^2\) such that $\iota = \Phi \circ \iota'$.
\end{corollary}

\begin{proof}
    We first give the idea of the proof.
    Two embeddings $\iota, \iota'$ as in the corollary yield splittings of~$\CP^2$, namely 
    $$
    \CP^2 = \left( \CP^2\setminus U \right) \cup U
    $$ 
    for \(U=\iota(E_{p,q}(\alpha,\beta))\),
    and analogously for~$\iota'$.
    By Corollary~\ref{cor:symplecto_of_blow_ups} the two complements can be symplectically identified, 
    and the embedded ellipsoids are trivially symplectically identified through the embeddings.
    We will adjust the first identification near the boundaries so that it agrees with the second one, 
    and hence obtain a symplectomorphism~$\Phi$ of $(\CP^2,\omega_{\text{FS}})$ that intertwines the ellipsoid embeddings.
    This symplectomorphism is Hamiltonian thanks to Gromov's result from~\cite{Gro85} 
    that the symplectomorphim group of $(\CP^2,\omega_{\text{FS}})$ is connected.

    We now come to the actual proof.
    By definition of ``symplectic embedding'', there exist $\hat  \alpha>\alpha$,
    $\hat \beta > \beta$ such that the maps $\iota,\iota'$ are restrictions of symplectic embeddings
    $$
    \hat \iota, \hat \iota' \colon E_{p,q}(\hat \alpha,\hat \beta) \hookrightarrow  \CP^2 .
    $$
    Choose the almost toric offcut $\Off_{\bm{\rho},\bm{\lambda}}(E_{p,q}(\hat \alpha, \hat \beta))$
    such that 
    $$
    E_{p,q}(\alpha, \beta) \subset \Off_{\bm{\rho},\bm{\lambda}}(E_{p,q}(\hat \alpha, \hat \beta))
    \subset E_{p,q}(\hat \alpha, \hat \beta)
    $$
    and such that $\overline{\Off}_{\bm{\rho},\bm{\lambda}}(E_{p,q}(\hat \alpha, \hat \beta))
    \subset \Int \bigl( E_{p,q}(\hat \alpha, \hat \beta) \bigr)$.\footnote{Even though the edges of the offcut have rational slopes,
    we can clearly choose $\bm{\rho},\bm{\lambda}$ such that the offcut has these properties.}
    Doing the pavilion blow-up, as introduced in Definition~\ref{dfn:pavilion_blow_up}, we obtain two closed symplectic manifolds, which we denote by~$(\til{X},\til{\omega})$ and $(\til{X}', \til{\omega}')$. 
    Denote by $\cc_{\vis}$ the linear chain of symplectic spheres in the pavilion blow-up $\Pav_{\bm \rho, \bm \lambda}\bigl(E_{p,q}(\hat \ga, \hat \gb)\bigr)$, 
    and let $\cn (\cc_{\vis})$ be a neighbourhood of~$\cc_{\vis}$ in~$\Pav_{\bm \rho, \bm \lambda}\bigl(E_{p,q}(\hat \ga, \hat \gb)\bigr)$.\footnote{The chain of spheres $\cc_{\vis}$ is visible in the almost toric diagram as in Figure~\ref{fig:offcut}~(b). 
    The combinatorial data of~$\cc_{\vis}$ is exactly that of~$\ct$ and~$\ct'$, except that $\ct$ and~$\ct'$ may in addition contain one or two $(-1)$-spheres, depending on the Manetti weight. Recall that these $(-1$)-spheres stem from the ``ambient" topology of~$\CP^2$ and are not contained in the linear chain corresponding to the choice of pavilion. 
    Therefore, the ``interface" between the embeddings of the pin-ellipsoids and the symplectomorphism of their complements will exactly be~$\cc_{\vis}$.}
    In the following we denote the linear subchains of~$\ct$ and~$\ct'$ corresponding to $\cc_{\vis}$ by $\cc$ and~$\cc'$. 
    Trivially, the complements of the trees $\ct$ and~$\ct'$ of symplectic spheres in these pavilion blow-ups are identified with the complement of the offcuts ``downstairs" in~$\CP^2$.
    We write 
    \begin{equation*}
        \jmath \colon \CP^2 \setminus \hat \iota \big(\overline{\text{Off}}_{\bm{\rho},\bm{\lambda}}(E_{p,q}(\hat \alpha, \hat \beta))\big) \to (\til{X} \setminus \cc,\til{\omega})
    \end{equation*}
    and analogously $\jmath'$ for these identifications.
    By Corollary~\ref{cor:symplecto_of_blow_ups}, 
    there exists a symplectomorphism $\psi \colon (\til{X}',\til{\omega}') \to (\til{X},\til{\omega})$ that takes $\ct'$ to~$\ct$. In particular, $\psi$ takes $\cc$ to $\cc'$.
    Now consider the diagram of symplectic maps shown in Figure~\ref{fig:comdiag}.
    Since the embedding $\hat \iota$ is defined on all of $E_{p,q}(\hat \ga, \hat \gb)$,
    the embedding 
    $$
    \jmath \circ \hat\iota \colon \cn(\cc_{\vis}) \setminus \cc_{\vis} \hookrightarrow \widetilde X
    $$
    uniquely extends to a smooth embedding $\gs \colon \cn (\cc_{\vis}) \hookrightarrow\ \widetilde X$. 
    Since $\jmath \circ \hat \iota$ is symplectic, $\gs$ is also symplectic, 
    and $\sigma(\cc_{\vis}) = \cc$.
    Using in addition that $\psi$ is a symplectomorphism  $\cc'$ to~$\cc$,
    we see in the same way that
    \begin{equation} \label{e:sips}
    \psi \circ \jmath' \circ \hat\iota' \colon \cn(\cc_{\vis}) \setminus \cc_{\vis} \hookrightarrow \widetilde X
    \end{equation}
    extends to a symplectic embedding $\gs' \colon \cn(\cc_{\vis}) \hookrightarrow \widetilde X$
    that also takes $\cc_{\vis}$ to~$\cc$.

    \begin{figure}[htb]
        \centering
            \begin{tikzpicture}
                \node (Epq) at (0,-3) {$E_{p,q}(\hat{\alpha},\hat{\beta})$};
                \node (CP20) at (-4,0) 
                {\begin{tabular}{c} 
                    \(\CP^2 \setminus \hat \iota' \big(\overline{\text{Off}}_{\bm{\rho},\bm{\lambda}}(E_{p,q}(\hat \alpha, \hat \beta))\big) \) \\ 
                    \(\cap\) \\
                    \(\CP^2\) \\
                    \(\cup\) \\
                    \(\hat\iota'(E_{p,q}(\hat{\alpha},\hat{\beta})) \) 
                \end{tabular}};
                \node (CP21) at (4,0)
                {\begin{tabular}{c} 
                    \(\CP^2 \setminus \hat\iota \big(\overline{\text{Off}}_{\bm{\rho},\bm{\lambda}}(E_{p,q}(\hat \alpha, \hat \beta))\big) \) \\ 
                    \(\cap\) \\
                    \(\CP^2\) \\
                    \(\cup\) \\
                    \(\hat\iota(E_{p,q}(\hat{\alpha},\hat{\beta})) \) 
                \end{tabular}};
                \node (X0) at (-4,3) {$\big((\til{X}',\til{\omega}'),\cc')$};
                \node (X1) at (4,3) {$\big((\til{X},\til{\omega}),\cc)$};
                \draw[->] 
                    (Epq)  to node[below left, midway]{$\hat{\iota}'$} (CP20.south); 
                \draw[->] 
                    (Epq)  to node[below right, midway]{$\hat{\iota}$} (CP21.south);
                \draw[->] 
                    (CP20)  to node[left, midway]{$\jmath'$} (X0);
                \draw[->] 
                    (CP21)  to node[right, midway]{$\jmath$} (X1); 
                \draw[->] 
                    (X0)  to node[above, midway]{$\psi$} (X1); 
                \end{tikzpicture}
  \caption{The diagram of symplectic maps used in the proof of Corollary~\ref{cor:deformation_uniqueness}.}
\label{fig:comdiag}
\end{figure}    

    \begin{lemma} \label{le:h}
    There exists a compactly supported Hamiltonian diffeomorphism $h$ of a neighbourhood~$\cn(\cc)$ such that near~$\cc$,
    \begin{equation} \label{e:shid}
        h \circ \sigma' = \sigma.
    \end{equation}
    \end{lemma}

    \begin{proof}
    Set $\tau = \sigma \circ (\sigma')^{-1} |_{\cc}$.
    Then $\tau (C_j)= C_j$ for all spheres~$C_j$ in~$\cc$, and $\tau$ fixes the transverse intersection points $C_j \cap C_{j+1}$.
    The group of orientation preserving diffeomorphisms of a sphere that fix one or two points is connected. 
    We therefore find a smooth path $\tau_t$ from $\id_{\cc}$ to~$\tau$.
    By Moser's deformation argument, we can assume that the~$\tau_t$ are symplectic.
    Using Weinstein's symplectic neighbourhood theorem, we can extend~$\tau_t$ to a smooth path of symplectic embeddings 
    $\hat \tau_t \colon \cn_t \to \cn(\cc)$
    of tubular neighbourhoods~$\cn_t$ of~$\cc$ that 
    starts at the inclusion. 
    Since $H_1(\cn_t)=0$, this path is generated by a Hamiltonian vector field. Multiplying the Hamiltonian function by a cut-off function 
    that is $1$ near~$\cc$
    and 0 outside a small neighbourhood of~$\cc$, we obtain a Hamiltonian diffeomorphism~$h$ supported near~$\ct$ 
    that agrees with $\sigma \circ (\sigma')^{-1}$ near~$\cc$.
    %Take $h$ to be its inverse.
    \end{proof}
    
    In view of Equations \eqref{e:sips} and \eqref{e:shid}
    we can define the symplectomorphism $\Phi \in \text{Symp}(\CP^2,\omega_{\FS})$ by
    $$
    \CP^2 \setminus \hat \iota' \big(\overline{\Off}_{\bm{\rho},\bm{\lambda}}(E_{p,q}(\hat \alpha, \hat \beta)) 
    \xrightarrow{\jmath^{-1} \circ h \circ \psi \circ \jmath' \;} 
    \CP^2 \setminus \hat \iota \big(\overline{\Off}_{\bm{\rho},\bm{\lambda}}(E_{p,q}(\hat \alpha, \hat \beta))
    $$
    and
    $$
    \hat \iota' (\cn) \xrightarrow{\hat\iota \circ (\hat\iota')^{-1}\,} 
    \hat \iota(\cn),
    $$
    where $\cn$ is a sufficiently small neighbourhood
    of 
    $\Off_{\bm{\rho},\bm{\lambda}}(E_{p,q}(\hat \alpha, \hat \beta))$ in $E_{p,q}(\hat \alpha, \hat \beta)$.
    By Lemma~\ref{le:h} this symplectomorphism is well-defined.
    Due to Gromov's theorem from~\cite{Gro85} that $\text{Symp}(\CP^2,\omega_{\FS})$ is connected, $\Phi$ is a Hamiltonian diffeomorphism, and hence the map we were looking for.
\end{proof}

\subsection{Proof of the Staircase Theorem (Theorem \ref{thm:Markovstairs})}
\label{subsec:staircase_proof}

In this section, we will use Corollary \ref{cor:deformation_uniqueness} to prove the Staircase Theorem, which we restate in the following form:

\begin{theorem}[Markov
  staircases]\label{thm:markov_staircases} Let
  \((p_1,p_2,p_3)\) be a Markov triple with
  \(q_1\coloneqq 3p_2p_3^{-1}\mod p_1\) and let
  \((p'_1,p'_2,p'_3)=(p_1,p_3,3p_3p_1-p_2)\) be its mutation at
  \(p_2\). There is no symplectic embedding
  \(E_{p_1,q_1}(\alpha,\beta)\hookrightarrow \CP^2\) if both
  \(\alpha>\frac{p_3}{p_1p_2}\) and
  \(\beta > \frac{p'_2}{p'_3p'_1}\).
  \end{theorem}

Fix the triple \((p_1,p_2,p_3)\), companion \(q_1\) and mutant triple \(p'_1,p'_2,p'_3\) as in the statement of Theorem \ref{thm:markov_staircases}. Consider the Vianna triangles
\(\Vianna(p_1,p_2,p_3)\) and \(\Vianna(p'_1,p'_2,p'_3)\) whose edges
\(\edge_i\) and \(\edge'_i\) have affine integral lengths
\begin{align*}
  |\edge_1| &= \frac{p_1}{p_2p_3} & |\edge_2| &= \frac{p_2}{p_3p_1} & |\edge_3| &= \frac{p_3}{p_1p_2} \\
  |\edge'_1| &= \frac{p'_1}{p'_2p'_3} & |\edge'_2| &= \frac{p'_2}{p'_3p'_1} & |\edge'_3| &= \frac{p'_3}{p'_1p'_2}.
\end{align*}
It suffices to show that if
\(E_{p_1,q_1}(\alpha,\beta)\) admits a symplectic embedding and
\(\alpha > |\edge_3|\) then \(\beta \leq |\edge'_2|\). If
we draw the two Vianna triangles \(\Vianna(p_1,p_2,p_3)\) and
\(\Vianna(p'_1,p'_2,p'_3)\) superimposed then we can see they both
contain \(\ATF_{p_1,q_1}(|\edge_3|,|\edge'_2|)\) (Figure \ref{fig:mutant_vianna_triangles}).

\begin{figure}[htb]
  \begin{center}
    \begin{tikzpicture}
      \filldraw[fill=lightgray,opacity=0.7] (0,0) -- (3,3) -- (0,5) -- cycle;
      \filldraw[fill=lightgray,opacity=0.7] (0,0) -- (4.5,4.5) -- (0,3) -- cycle;
      \filldraw[fill=gray,opacity=0.7] (0,0) -- (3,3) -- (0,3) -- cycle;
      \draw (0,0) -- (3,3) -- (0,5) -- cycle;
      \draw (0,0) -- (4.5,4.5) -- (0,3) -- cycle;
      \draw[dashed] (0,0) -- ++ (0.16*2,0.5*2) node (a) {\(\times\)};
      \draw[dashed] (3,3) -- ++ (-1.5,0) node (b) {\(\times\)};
      \draw[dashed] (0,5) -- ++ (0.25*2,-0.5*2) node (c) {\(\times\)};
      \draw[dashed] (4.5,4.5) -- ++ (-1.16,-0.5) node (d) {\(\times\)};
      \node at (2,3) [above] {\(\bc_2\)};
      \draw [decorate,decoration={brace,amplitude=5pt,raise=1ex}] (0,0) -- (0,3) node[midway,xshift=-4ex]{\(|\edge'_2|\)};
      \draw [decorate,decoration={brace,amplitude=20pt,raise=4ex}] (0,0) -- (0,5) node[midway,xshift=-10ex]{\(|\edge_2|\)};
      \draw [decorate,decoration={brace,amplitude=5pt,raise=1ex}] (3,3) -- (0,0) node[midway,xshift=3ex,yshift=-2ex]{\(|\edge_3|\)};
      \draw [decorate,decoration={brace,amplitude=20pt,raise=4ex}] (4.5,4.5) -- (0,0) node[midway,xshift=6ex,yshift=-7ex]{\(|\edge'_3|\)};
      \node (p) at (6,2) {\(\ATF_{p_1,q_1}(|\edge_3|,|\edge'_2|)\)};
      \node (q) at (-3,4) {\(\Vianna(p_1,p_2,p_3)\)};
      \node (r) at (7,4) {\(\Vianna(p'_1,p'_2,p'_3)\)};
      \draw[->] (p) -- (1,1.8);
      \draw[->] (q) -- (0.5,3.5);
      \draw[->] (r) -- (3,3.5);
    \end{tikzpicture}
    \caption{The Vianna triangles \(\Vianna(p_1,p_2,p_3)\) and
      \(\Vianna(p'_1,p'_2,p'_3)\) superimposed. The dark shaded region
      is a copy of \(\ATF_{p_1,q_1}(|\edge_3|,|\edge'_2|)\) in
      their intersection.}
    \label{fig:mutant_vianna_triangles}
  \end{center}
\end{figure}
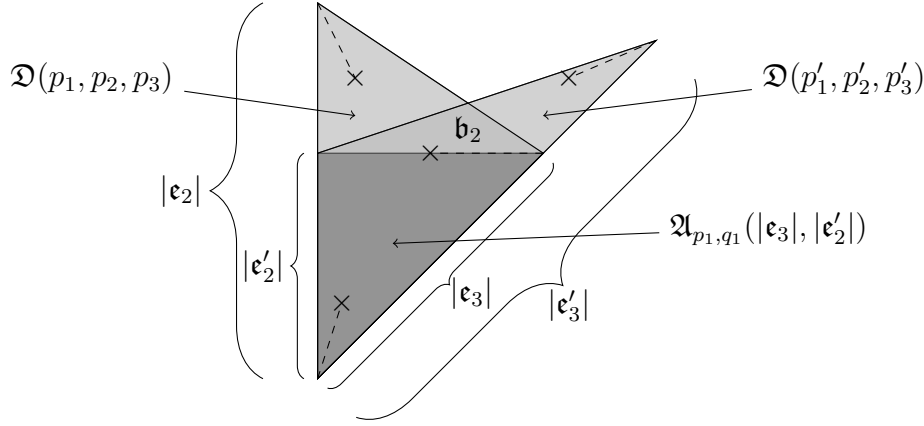

\begin{lemma}\label{lma:affine_length} 
Consider the
  triangle \(\ATF_{p_1,q_1}(|\edge_3|,|\edge'_2|)\) and recall that
  its girdle is the top side.
  \begin{itemize}
    \item[(a)] The primitive integer vector pointing along the
    girdle is \(\begin{pmatrix}p'_3\\ (p'_3q_1-3p'_2)/p_1\end{pmatrix}\).
    \item[(b)] The affine length of the girdle is equal to
    \(\frac{p_1p_3}{p_2p'_3}\).
    \item[(c)] The affine displacement between the girdle and the vertex~$\vtx_1$ is
    equal to \(p_3/p_1\).
  \end{itemize}
\end{lemma}  
\begin{proof}
  The vertices of \(\ATF_{p,q}(|\edge_3|,|\edge'_2|)\) are at
  \[\begin{pmatrix} 0 \\ 0 \end{pmatrix},\quad |\edge_3|\begin{pmatrix}
  p_1^2\\p_1q_1-1\end{pmatrix},\quad \begin{pmatrix}
    0\\ |\edge'_2|\end{pmatrix},\] so the vector pointing along the
  girdle is
  \[\begin{pmatrix} |\edge_3|p_1^2
      \\ |\edge_3|p_1q_1-|\edge_3|-|\edge'_2|\end{pmatrix} =
    \frac{p_3}{p_1p_2p'_3}\begin{pmatrix}p'_3p_1^2\\
      p'_3p_1q_1-p'_3-p_2\end{pmatrix}=
    \frac{p_3}{p_2p'_3}\begin{pmatrix}p'_3p_1\\
      p'_3q_1-3p_3\end{pmatrix},\] where we used the facts that
  \(p'_1=p_1\), \(p'_2=p_3\) and \(p_3'+p_2=3p_1p_3\).
      
  Since no Markov number is divisible by~$3$, and since the numbers in a
  Markov triple are always pairwise coprime,
  \(\gcd(p'_3,p'_3q_1-3p_3) = \gcd(3p_1p_3-p_2,3p_3) =
  \gcd(p_2,p_3) = 1\).

  We also have
  \(\gcd(p_1,p'_3q_1-3p_3) = \gcd(p_1,(3p_1p_3-p_2)q_1-3p_3) =
  \gcd(p_1,p_2q_1+3p_3) = p_1\) since\footnote{Note that \(p_2^2+p_3^2=3p_1p_2p_3-p_1^2=0\mod p_1\) so \(p_2p_3^{-1}=-p_3p_2^{-1}\mod p_1\).}
  \(q_1=3p_2p_3^{-1}=-3p_3p_2^{-1}\mod p_1\). 
  Therefore the vector along the girdle is
  \[\frac{p_1p_3}{p_2p'_3}\begin{pmatrix}p'_3\\
      (p'_3q_1-3p_3)/p_1\end{pmatrix}\] 
where
  \(\frac{p_1p_3}{p_2p'_3}\) is the integral affine length and 
  \(u\coloneqq\begin{pmatrix}p'_3\\
    (p'_3q_1-3p_3)/p_1\end{pmatrix}\) is a primitive integer
  vector. The affine displacement between the girdle and the
  vertex at the origin is therefore 
  \[u^\perp \cdot \begin{pmatrix}0 \\
      |\edge'_2|\end{pmatrix} = \begin{pmatrix}
      -(p'_3q_1-3p_3)/p_1 \\ p'_3\end{pmatrix}
      \cdot \begin{pmatrix} 0 \\
      p_3/(p_1p'_3)\end{pmatrix} = p_3/p_1. \qedhere\]
\end{proof}

\begin{lemma}\label{lma:pavilion_exists} If
  \(\alpha>|\edge_3|\) and \(\beta>|\edge'_2|\) then there
  exists a Delzant pavilion for \(E_{p_1,q_1}(\alpha,\beta)\) (which introduces new edges \(\ff_1,\ldots,\ff_m\) with inward normals \(\bm{\rho}=(\rho_1,\ldots,\rho_m)\) corresponding to curves \(C_1,\ldots,C_m\)) and an
  index \(i\) such that:
  \begin{itemize}
  \item[(a)] the edge \(\ff_i\) is parallel to the branch cut \(\bc_2\)
    in the Vianna triangle \(\Vianna(p_1,p_2,p_3)\);
  \item[(b)] \(\int_{C_i}\til{\omega}>\frac{p_1p_3}{p_2p'_3}\) ;
  \item[(c)] if we define \[b_i=-C_i^2,\quad e_i/e_{i-1}=[b_{i-1},\ldots,b_1]\mbox{ and }
    f_i/f_{i+1}=[b_{i+1},\ldots,b_m]\] then \(e_i=p'_3\) and \(f_i=p_2\).
  \end{itemize}
  \end{lemma}
\begin{proof}
  The fact that \(\alpha>|\edge_3|\) and \(\beta>|\edge'_2|\)
  means that the almost toric base diagram
  \(\ATF_{p_1,q_1}(|\edge_3|,|\edge'_2|)\) is strictly contained
  in \(\ATF_{p_1,q_1}(\alpha,\beta)\). The top side of this
  subdiagram is precisely where the branch cut \(\bc_2\) would
  run in the Vianna triangle \(\Vianna(p_1,p_2,p_3)\) (see Figure~\ref{fig:mutant_vianna_triangles}). 
  Since the
  containment of the subdiagram is strict, this means that we can
  make a pavilion cut parallel to \(\bc_2\) just above this top
  side (with further very small cuts near the ends of this side
  to ensure the pavilion is Delzant). The symplectic area of the
  corresponding curve~\(C_i\) is given by the affine length of
  the side. Since our cut is slightly above the top side of
  \(\ATF_{p_1,q_1}(|\edge_3|,|\edge'_2|)\) (whose affine length
  is \(\frac{p_1p_3}{p_2p'_3}\) by Lemma
  \ref{lma:affine_length}) and since we can take our remaining
  cuts to be arbitrarily small, we can ensure that
  \(\int_{C_i}\til{\omega}>\frac{p_1p_3}{p_2p'_3}\).

  Finally, since the integer vector along the girdle is \(\begin{pmatrix}p'_3\\
    (p'_3q_1-3p_3)/p_1\end{pmatrix}\), we can compute \(e_i\) and \(f_i\) as
  \begin{align*}
    e_i&=\begin{pmatrix}p'_3\\
    (p'_3q_1-3p_3)/p_1\end{pmatrix}\wedge \begin{pmatrix} 0 \\
    1 \end{pmatrix} = p'_3,\\
    f_i &= \begin{pmatrix}p'_3\\
    (p'_3q_1-3p_3)/p_1\end{pmatrix} \wedge \begin{pmatrix} p_1^2 \\
    p_1q_1-1\end{pmatrix}\\&
    = p'_3(p_1q_1-1)-p_1(p'_3q_1-3p_3)=3p_3p_1-p'_3=p_2.\qedhere
  \end{align*}
\end{proof}

We now complete the proof. Let \(\bm{\rho}\) be the normal
vectors and \(\bm{\lambda}\) the choice of constants giving the
pavilion from Lemma \ref{lma:pavilion_exists}. Let
\(\bm{\lambda}_t=(t\lambda_1,\ldots,t\lambda_m)\) for
\(t\in[0,1]\) and consider the family of pavilion cuts
\(\left(\Pav^\iota_{\bm{\rho},\bm{\lambda}_t}(\CP^2),\til{\omega}_t\right)\). When
\(t\) is sufficiently small, there is a visible symplectic
embedding 
\(\iota_{\vis} \colon \Off_{\bm{\rho},\bm{\lambda}_t}(E_{p,q}(\alpha,\beta))\hookrightarrow \CP^2\) given by the neighbourhood of a vertex of a Vianna
triangle. By Corollary~\ref{cor:deformation_uniqueness}, the
restriction
$$\iota|_{\Off_{\bm{\rho},\bm{\lambda}_t}(E_{p,q}(\alpha,\beta))}\colon \Off_{\bm{\rho},\bm{\lambda}_t}(E_{p,q}(\alpha,\beta))\hookrightarrow\CP^2$$ 
is Hamiltonian
isotopic to~$\iota_{\vis}$, so \(\Pav^\iota_{\bm{\rho},\bm{\lambda}_t}(\CP^2)\) is symplectomorphic to \(\Pav^{\iota_{\vis}}_{\bm{\rho},\bm{\lambda}_t} (\CP^2)\). 
In the pavilion blow-up \(\Pav^{\iota_{\vis}}_{\bm{\rho},\bm{\lambda}_t}(\CP^2)\) along the visible offcut, there is a visible symplectic \(-1\)-sphere \(S_t\), living over an arc
in the almost toric base diagram which connects the node coming
out of \(v_2\) to the edge \(\ff_i\) corresponding to \(C_i\). This is because \(\ff_i\) is parallel to \(\bc_2\); see {\cite[Example 9.1 and Figure 9.1]{ELTF}} and Figure \ref{fig:visible_sphere}.

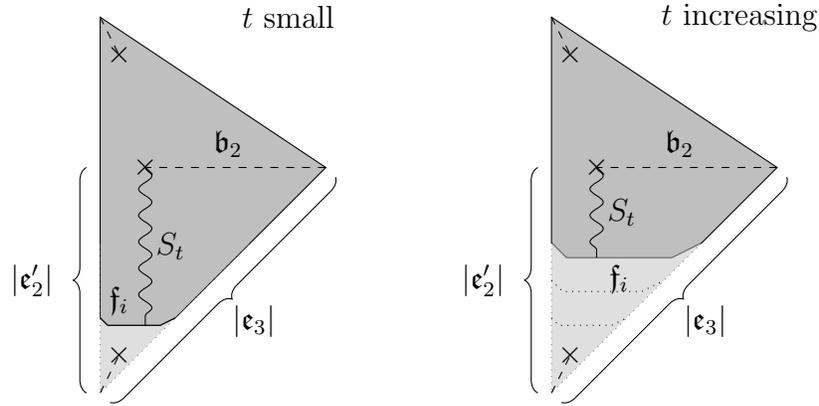
\begin{figure}[htb]
  \begin{center}
    \begin{tikzpicture}
      \node at (2.5,5) {\(t\) small};
      \filldraw[fill=lightgray,draw=black] (0,1) -- (0.1,0.9) -- (0.8,0.9) -- (1,1) -- (3,3) -- (0,5) -- cycle;
      \filldraw[fill=lightgray,draw=black,dotted,opacity=0.5,] (0,2) -- (0,0) -- (2,2);
      \draw[dashed] (0,0) -- (0.25,0.5) node {\(\times\)};
      \draw[dashed] (3,3) -- (0.6,3) node {\(\times\)};
      \draw[dashed] (0,5) -- (0.25,4.5) node {\(\times\)};
      \draw (0,1) -- (0.1,0.9) -- (0.8,0.9) -- (1,1);
      \draw[decoration=snake,decorate] (0.6,3) -- (0.6,0.9) node [midway,right] {\(S_t\)};
      \node at (1.7,3) [above] {\(\bc_2\)};
      \draw [decorate,decoration={brace,amplitude=5pt,raise=1ex}] (3,3) -- (0,0) node[midway,xshift=3ex,yshift=-3ex]{\(|\edge_3|\)};
      \draw [decorate,decoration={brace,amplitude=5pt,raise=1ex}] (0,0) -- (0,3) node[midway,xshift=-5ex]{\(|\edge'_2|\)};
      \node at (0.25,0.9) [above] {\(\ff_i\)};
      \begin{scope}[shift={(6,0)}]
        \node at (2.5,5) {\(t\) increasing};
        \filldraw[fill=lightgray,draw=black] (0,2) -- (0.2,1.8) -- (1.6,1.8) -- (2,2) -- (3,3) -- (0,5) -- cycle;
        \filldraw[fill=lightgray,draw=black,dotted,opacity=0.5,] (0,2) -- (0,0) -- (2,2);
        \draw[dashed] (0,0) -- (0.25,0.5) node {\(\times\)};
        \draw[dashed] (3,3) -- (0.6,3) node {\(\times\)};
        \draw[dashed] (0,5) -- (0.25,4.5) node {\(\times\)};
        \draw[decoration=snake,decorate] (0.6,3) -- (0.6,1.8) node [midway,right] {\(S_t\)};
        \node at (1.7,3) [above] {\(\bc_2\)};
        \draw [decorate,decoration={brace,amplitude=5pt,raise=1ex}] (3,3) -- (0,0) node[midway,xshift=3ex,yshift=-3ex]{\(|\edge_3|\)};
        \draw [decorate,decoration={brace,amplitude=5pt,raise=1ex}] (0,0) -- (0,3) node[midway,xshift=-5ex]{\(|\edge'_2|\)};
        \node at (0.9,1.85) [below] {\(\ff_i\)};
        \begin{scope}[scale=1.5]
          \draw[dotted] (0,1) -- (0.1,0.9) -- (0.8,0.9) -- (1,1);
        \end{scope}
        \begin{scope}[scale=1]
          \draw[dotted] (0,1) -- (0.1,0.9) -- (0.8,0.9) -- (1,1);
        \end{scope}
      \end{scope}
    \end{tikzpicture}
    \caption{The pavilion blow-up
      \(\Pav^{\iota_{\vis}}_{\bm{\rho},\bm{\lambda}_t}(\CP^2)\) along a
      visible \(E_{p,q}(\alpha,\beta)\) in \(\CP^2\) for two
      different values of \(t\). The visible sphere \(S_t\)
      lives over an arc connecting a node to the edge~\(\ff_i\). 
      As \(t\) increases, the pavilion cuts move
      upwards and the symplectic area of \(S\) gets smaller. If
      \(\alpha>|\edge_3|\) and \(\beta>|\edge'_2|\), we could
      push~\(\ff_i\) up above the branch cut \(\bc_2\) and \(S_t\)
      would end up with negative area.}
    \label{fig:visible_sphere}
  \end{center}
\end{figure}

This symplectic sphere $S_t$ intersects \(C_i\) and none of the other
curves \(C_j\). As in Equation \eqref{eq:splitting}, we can
write
\[
[S_t]=s_0 \, \eE+\sum_{j=1}^m s_j[C_j] \in H_2(\Pav^\iota_{\bm{\rho},\bm{\lambda}_t}(\CP^2);\QQ)
\] 
and
\(\sigma_j := [S_t]\cdot [C_j]\), so
\(\bm{\sigma}=(\sigma_1,\ldots,\sigma_m)=(0,\ldots,0,1,0,\ldots,0)\)
with the \(1\) in the \(i\)th place. Since \(S_t\) is a
symplectic \(-1\)-sphere, the spherical Gromov--Taubes invariant
\(\OP{Gr}([S_t])\) counting holomorphic spheres in
\(\Pav^\iota_{\bm{\rho},\bm{\lambda}_t}(\CP^2)\) in the class
\([S_t]\) is equal to \(1\),
see for instance \cite[Proposition~4.1]{McDlectures}.
Since Gromov--Taubes invariants are
deformation invariant, there is still a holomorphic sphere in
this homology class for
\(\Pav^\iota_{\bm{\rho},\bm{\lambda}_1}(\CP^2)\). 
This means that
\[0<\int_{[S_1]}\til{\omega}_1=s_0\int_{\eE}\til{\omega}_1+\sum_{j=1}^ms_j\int_{C_j}\til{\omega}_1.\]
Since \(\bm{s}=M^{-1}\bm{\sigma}\), and since by Lemma~\ref{lma:M_inverse} 
the entries of
\(M^{-1}\) are all negative, we get
\[0<s_0\int_{\eE}\til{\omega}_1+s_i\int_{C_i}\til{\omega}_1.\]
Recall from Remark \ref{rmk:ee} that $p_1^2 \eE$ can be represented by a cycle
in \(V\) (the complement of our embedding).
Hence \(s_0\int_{\eE}\til{\omega}_t\) is constant in \(t\);
in particular, it is equal to the limit of the areas of the
visible \(\til{\omega}_t\)-symplectic spheres \(S_t\) as
\(t\to 0\); by {\cite[Proposition 7.8]{Symington3}}, this equals the affine displacement~$\frac{p_3}{p_1}$ between the
girdle and the bottom vertex computed in Lemma~\ref{lma:affine_length}(c).
Therefore, 
\[0<\frac{p_3}{p_1} + s_i\int_{C_i}\til{\omega}_1.\] 
By Lemma~\ref{lma:pavilion_exists}(b), this tells us that
\[-s_i \, \frac{p_1p_3}{p_2p'_3} < \frac{p_3}{p_1}.\] 
But
\(s_i = M^{-1}_{ii} = -\frac{e_if_i}{p_1^2} = -\frac{p_2p'_3}{p_1^2}\) by
Lemma~\ref{lma:M_inverse} and 
Lemma \ref{lma:pavilion_exists}(c), so we get
\[\frac{p_3}{p_1}<\frac{p_3}{p_1},\] which is a contradiction. This
shows that we cannot have both \(\alpha>\frac{p_3}{p_1p_2}\) and
\(\beta>\frac{p'_2}{p'_3p'_1}\), as required.

\subsection{Proof of Theorem \ref{thm:input_from_es} (Evans--Smith {\em sans} orbifolds)}
\label{subsec:ES}

Recall the notation from Definitions \ref{dfn:spaces}--\ref{dfn:pavilion_blow_up}: given a symplectic embedding
\[\iota\colon E_{p,q}(\alpha,\beta)\hookrightarrow X\coloneqq \CP^2,\] we let \(U=\iota(E_{p,q}(\alpha,\beta))\), let \(V=X\setminus U\), and let \(\overline{U}\) and \(\overline{V}\) be the symplectic completions of \(U\) and \(V\); we also write
\(\Sigma=\partial U\).
We pick a neck-stretching sequence of almost complex structures \(J_t\) as in Definition \ref{dfn:spaces}. 

\begin{lemma}
    For each \(t\), pick a \(J_t\)-holomorphic line \(L_t\) in the class 
    \(H\in H_2(\CP^2;\ZZ)\).
    As \(t\to \infty\), there is a subsequence \(t_i\to\infty\) such that the sequence \(L_{t_i}\) converges (in the sense of Gromov--Hofer convergence {\cite[Section 9.1]{BEHWZ}}) to a holomorphic limit building with a nonempty component in \(\overline{U}\) and a nonempty component in \(\overline{V}\) (and potentially components in intermediate symplectisation levels).
\end{lemma}
\begin{proof}
Convergence to a limit building is the content of the SFT compactness theorem \cite{BEHWZ,CieliebakMohnke}. 
There must be a component in \(\overline{V}\) because
\(\overline{U}\) is exact and therefore unable to contain a
closed holomorphic curve. There must be components in
\(\overline{U}\) because the anticanonical bundle of \(X\)
restricted to the complement of a line can be trivialised, whilst
the first Chern class of \(\overline{U}\) is nontrivial (though
it is torsion; see {\cite[Lemma 2.13]{ES}}).
\end{proof}

Let \(C\) be one of
the components of the SFT limit curve living in \(\overline{V}\).
Fix a minimal Delzant pavilion \(\bm{\rho},\bm{\lambda}\), let \(\til{X}\) be the pavilion blow-up and let \(\til{U}\subset\til{X}\) be the girdled resolution with complement \(V\). Let \(\til{C}\) be the corresponding closed curve in \(\til{X}\) given by Lemma \ref{lma:bijections}. 
Our strategy will be to carefully analyse the adjunction
formula for \(\til{C}\), using the discrepancies of the
singularity to compute \(K\cdot \til{C}\)
where $K = - \mbox{PD} (c_1(\til{X}))$ is the canonical class.

For that purpose, we
write \(K_\QQ\) and \(\til{C}_\QQ\), as in Equation
\eqref{eq:splitting}, in the form
\[K_\QQ=-3p\eE+\sum k_iC_i,\qquad \til{C}_\QQ=c_0\eE+\sum
  c_iC_i.\] 
The coefficient of \(\eE\) in \(K_\QQ\) comes from
the facts that \(K_{\CP^2}=-3H\) and \(H=p\eE\). The numbers
\(k_i\) are called the discrepancies of the singularity and can
be computed to be \(k_i=-1+\frac{e_i+f_i}{p^2}\), see
{\cite[Section 2.1]{HTU}}. Here, $e_i$ and $f_i$ are defined as in Lemma \ref{lma:M_inverse}. Note that these all lie in the
interval \((-1,0)\) since we have taken the minimal resolution;
if we had taken a non-minimal resolution then some of the
discrepancies would have been positive.

\begin{lemma}\label{lma:D_interval} We have
  \(0 < c_0 \leq p\).
\end{lemma}
\begin{proof}
  The symplectic form on \(\overline{V}\) extends to a closed 2-form
  \(\varpi\) on \(\til{X}\) which vanishes on~\(\bigcup_{i=1}^m
  C_i\) (this is essentially the form on the pavilion blow-up with \(\bm{\lambda}=\bm{0}\)). The integral \(\int_{\til{C}}\varpi\) agrees with the
  symplectic area of \(C\subset\overline{V}\), which is bounded above
  by the symplectic area of a line in \(\CP^2\) (which we are
  normalising to be \(1\)). We have \(\int_{\til{C}}\varpi=c_0/p\)
  since the symplectic area of the line (which is homologous to
  \(p\eE\)) is \(1\). Note that \(\varpi\) is non-negative on
  \(\til{J}\)-complex lines in tangent spaces and vanishes only on
  points of \(\bigcup_{i=1}^mC_i\), so
  \(c_0/p=\int_{\til{C}}\varpi>0\).
\end{proof}

\begin{proposition}\label{prp:embedded_curve} The curve \(\til{C}\) is an
  embedded sphere of self-intersection \(\til{C}^2 \leq 0\).
\end{proposition}
\begin{proof}
  The curve \(\til{C}\) has arithmetic genus zero (since we
  obtained it by stretching a holomorphic sphere). The
  adjunction formula for \(\til{C}\) is 
  \[\sum_{x\in\OP{Sing}(\til{C})}
    \delta_x \,=\, \frac{\til{C}^2+K\cdot\til{C}}{2}+1\] where 
  \(\delta_x\) is the
  multiplicity of the singular point \(x\in \til{C}\). We have:
  \begin{align*}
    \til{C}^2 &= c_0^2/p^2+\left(\sum_{i=1}^m c_iC_i\right)^2\\
    \til{C}\cdot K&=\frac{-3c_0}{p}+\sum_{i=1}^m\chi_i k_i,
  \end{align*} where \(\chi_i=\til{C}\cdot C_i\).
  Since the intersection matrix \(M_{ij}=C_i\cdot C_j\) is negative
  definite, \(\left(\sum_{i=1}^m c_iC_i\right)^2\leq 0\).  By
  positivity of intersections, \(\chi_i=\til{C}\cdot C_i\geq 0\) for
  all \(i\), so negativity of the discrepancies implies
  \(\sum_{i=1}^m\chi_i k_i \leq 0\). Therefore the adjunction formula
  becomes
  \[\sum_{x\in \OP{Sing}(\til{C})}\delta_x \leq
    \frac{c_0^2-3pc_0}{2p^2}+1.\] 
  Since \(c_0\) satisfies $0 < c_0 \leq p$ by
  Lemma \ref{lma:D_interval}, we find that $c_0^2 - 3pc_0$ is strictly negative. Since the numbers \(\delta_x\)
  are positive integers, we deduce that
  \(\OP{Sing}(\til{C})=\emptyset\). This shows that \(\til{C}\)
  is an embedded sphere.

  The self-intersection of \(\til{C}\) is
  \(c_0^2/p^2+\left(\sum_{i=1}^mc_iC_i\right)^2\). Since
  \(M_{ij}\) is negative definite, and since \(c_0\leq p\), this
  is \(\leq 1\). 
  Therefore \(\til{C}^2 \leq 1\) and \(\til{C}^2\leq 0\) unless 
  $c_0 = p$ and \(\bm{c}=0\). 
  In the latter case, $\til{C}$ would be homologous to $H$. The following shows that this is impossible. Recall that
  \(U=\iota(E_{p,q}(\alpha,\beta))\) deformation-retracts onto a CW
  complex called a pin-wheel, which has a fundamental class in
  homology with coefficients in \(\ZZ_\ell\) for any prime
  \(\ell\) dividing~\(p\) (note that, despite the name, \(p\) is
  not assumed to be prime). The pin-wheel has nonzero
  self-intersection, so the fundamental class of this pin-wheel
  has nonzero \(\ZZ_\ell\) intersection number with a line
  (which generates \(\ZZ_\ell\)-homology). This shows that
  \(\til{C}\) cannot be both homologous to \(H\) and disjoint
  from \(\bigcup_{i=1}^mC_i\), which concludes the proof that
  \(\til{C}^2 \leq 0\).
\end{proof}

\begin{remark}
    Choosing the almost complex structure on \(V\) generically, we can ensure that the curve $\til{C}$ has square $0$ or $-1$, because the expected dimension of the moduli space is only positive in these two cases.
\end{remark}

\begin{proposition}
At least one of the components
  \(C\) of the SFT limit yields a curve
  \(\til{C}\subset\til{X}\) which has square zero.
\end{proposition}
\begin{proof}
  We can impose an arbitrary point constraint on the line before
  we stretch, so that we find SFT-limit curves through every
  point of \(\overline{V}\). If all of these were to lift to
  curves of negative square then, since there is only a
  countable number of such curves (at most one in each homology
  class), we could not find one through every point. This is a
  contradiction.
\end{proof}

To summarise, we have found an embedded
\(\til{C}\subset\til{X}\) with \(\til{C}^2=0\). Recall that we
write \(\til{C}=c_0\eE+\sum c_iC_i\) with \(0< c_0 \leq p\), and
we define \(\chi_i=\til{C}\cdot C_i\),
\(\bm{c}^T=(c_1,\ldots,c_m)\),
\(\bm{\chi}^T=(\chi_1,\ldots,\chi_m)\) and
\(M_{ij}=C_i\cdot C_j\), so that \(\bm{\chi}=M\bm{c}\). 
A formula for \(M^{-1}_{ij}\) was given in Lemma~\ref{lma:M_inverse} 
in terms of the numbers \(e_i,f_i\) defined
there. The fact that \(\til{C}^2=0\) implies
\[ 0=\til{C}^2=\frac{c_0^2}{p^2}+\bm{\chi}^TM^{-1}\bm{\chi} . \]
Finally, the adjunction formula tells us that
\[0=\frac{\til{C}^2+K\cdot \til{C}}{2}+1,\]
that is
\begin{equation}\label{eq:important}
  1=\frac{3pc_0-c_0^2}{2p^2}-\frac{1}{2}\bm{\chi}^TM^{-1}\bm{\chi}-\frac{1}{2}\bm{k}^T\bm{\chi}
\end{equation}
where \(\bm{k}^T=(k_1,\ldots,k_m)\) is the vector of (negative)
discrepancies defined by \(K=-3p\eE+\sum k_iC_i\).

\begin{lemma}\label{lma:chi} 
We have {\rm (a)} \(\chi_i \leq 1\) for all
  \(i\) and {\rm (b)} \(\chi_i=1\) for precisely one \(i\).
  \end{lemma}
\begin{proof}
  All the entries of \(-M^{-1}\) are positive and also $\chi_i \geq 0$, so the second
  and third terms on the right-hand side of Equation
  \eqref{eq:important} contribute at least
  \begin{equation}
    \label{eq:contributions}\sum_{i=1}^m\left(\frac{1}{2p^2}\chi^2_ie_if_i +
      \frac{1}{2}\chi_i\left(1 - \frac{e_i+f_i}{p^2}\right)\right)
  \end{equation}
  to a sum whose total is \(1\).
  
  If \(\chi_i\geq 2\) then
  \[\frac{1}{2p^2}\chi^2_ie_if_i +
    \frac{1}{2}\chi_i\left(1 - \frac{e_i+f_i}{p^2}\right)\geq
    1+\frac{1}{p^2}\left(2e_if_i-e_i-f_i\right) \geq 1,\] 
    since
  \(e_i,f_i\geq 1\). This is not possible, since 
  the first term in Equation \eqref{eq:important} is positive.
    Therefore \(\chi_i\leq 1\) for all
  \(i\). This proves (a).

  If \(\chi_i=1\) then 
  \[\frac{1}{2p^2}\chi^2_ie_if_i +
    \frac{1}{2}\chi_i\left(1 - \frac{e_i+f_i}{p^2}\right)=
    \frac{1}{2}\left(1+\frac{1}{p^2}\left(e_if_i-e_i-f_i\right)\right)
    \geq \frac{1}{2}(1-1/p^2).\]
  Let \(r\) be the number of indices \(i\) with
  \(\chi_i\neq 0\). Then Equation \eqref{eq:important} implies
  \[1\geq \frac{1}{2p^2}\left(3pc_0-c_0^2\right) +
    \frac{r}{2}\left(1-\frac{1}{p^2}\right).\] The function
  \(3pc_0-c_0^2\) is an increasing function of \(c_0\) on the
  interval \(1\leq c_0\leq p\) and with minimum \(3p-1\), so
  \[1\geq \frac{r}{2}\left(1-\frac{1}{p^2}\right)+\frac{3p-1}{2p^2}.\] 
  If $r=2$, the right hand side is $>1$. If $r\geq 3$, the right hand side is also $>1$,
  since then (using $p \geq 2$) we have $\frac{r}{2} - \frac{r}{2p^2} \geq \frac{3}{2}-\frac{3}{8} = \frac{9}{8} >1$. The only possibility is that \(r=1\), which proves (b).
\end{proof}

\begin{lemma}\label{lma:markov_triple} There exist
  integers \(p_2,p_3\) such that \((p,p_2,p_3)\) is a
  Markov triple.
  \end{lemma}
\begin{proof}
  We now know that the curve \(\til{C}\) intersects precisely
  one of the curves \(C_i\) (and does so once transversely),
  that is \(\chi_i=1\) and \(\chi_j=0\) for \(j\neq
  i\). Therefore,
  \[0=\til{C}^2=\frac{c_0^2}{p^2}+\bm{\chi}^TM^{-1}\bm{\chi}=\frac{1}{p^2}\left(c_0^2
      - e_if_i\right),\] so \[c_0^2=e_if_i.\]
  Equation \eqref{eq:important} becomes
  \[1=
    \frac{3pc_0-c_0^2}{2p^2}+\frac{e_if_i}{2p^2}+\frac{1}{2}-\frac{e_i+f_i}{2p^2},\]
  that is,
  \[p^2 + e_i + f_i = 3p \sqrt{e_if_i}.\] Set \(w_0=p^2\), \(w_1=e_i\)
  and \(w_2=f_i\); we find
  \[\frac{(w_0+w_1+w_2)^2}{w_0w_1w_2}=9.\] Now {\cite[Lemma
    2.8]{HausenKiralyWrobel}} implies that \(w_0,w_1,w_2\) is a
  squared Markov triple, that is \(e_i=p_2^2\) and \(f_i=p_3^2\) for
  some Markov triple \((p,p_2,p_3)\). 
\end{proof}

\begin{corollary}\label{cor:hits_culet} 
  The singularity \(\frac{1}{p^2}(1,pq-1)\) has the form
  \(\frac{1}{p^2}(p_3^2,p_2^2)\), so \(q=\pm 3p_2p_3^{-1}\mod p\),
  and the curve \(C_i\) which \(\til{C}\) hits is the culet
  curve.
  \end{corollary}
\begin{proof}
  Consider the toric singularity \(\frac{1}{p^2}(1,pq-1)\),
  whose cone is bounded by the rays \(\RR_{\geq 0}(1,0)\) and
  \(\RR_{\geq 0}(1-pq,p^2)\). The curve \(C_i\) in the minimal
  resolution corresponds to a ray in this cone, say
  \(\RR_{\geq 0}(-t,s)\). Add the {\em opposite} ray
  \(\RR_{\geq 0}(t,-s)\) into the fan, and instead of getting a
  resolution, we get a compactification of the toric singularity
  (the fan is now complete). In terms of the moment polygon,
  this is equivalent to making a symplectic cut parallel to the
  \((s,t)\)-direction whose inward normal points {\em
    downwards}. See Figure \ref{fig:culet_proof}.

  Since the partial resolution corresponding to the ray
  \(\RR_{\geq 0}(-t,s)\) would introduce the cyclic
  singularities \(\frac{1}{e_i}(1,e_{i-1}^{-1}\mod e_i)\) and
  \(\frac{1}{f_i}(1,f_{i+1}\mod f_i)\) (by definition of \(e_i\)
  and \(f_i\) from Lemma \ref{lma:M_inverse}), the
  compactification will have dual singularities
  \[\frac{1}{e_i}(1,e_i-e_{i-1}^{-1}),\qquad \frac{1}{f_i}(1,f_i-f_{i+1}).\]
  Since \(\gcd(p^2,e_i,f_i)=1\), the result is a weighted
  projective space \(\PP(p^2,e_i,f_i)\). Since \(p^2,e_i,f_i\)
  is a Markov triple \((p_1^2,p_2^2,p_3^2)\), the original
  singularity is \(\frac{1}{p_1^2}(p_3^2,p_2^2)\) and the two
  new singularities are dual to the Wahl singularities
  \(\frac{1}{p_2^2}(p_1^2,p_3^2)\) and
  \(\frac{1}{p_3^2}(p_2^2,p_1^2)\), so by the discussion in
  Section \ref{pg:culet}, the curve \(C_i\) is the culet curve. 
  The fact that \(q=\pm 3p_2p_3^{-1}\mod p\) follows as in Equation \eqref{eq:formula_for_companion}.
\end{proof}

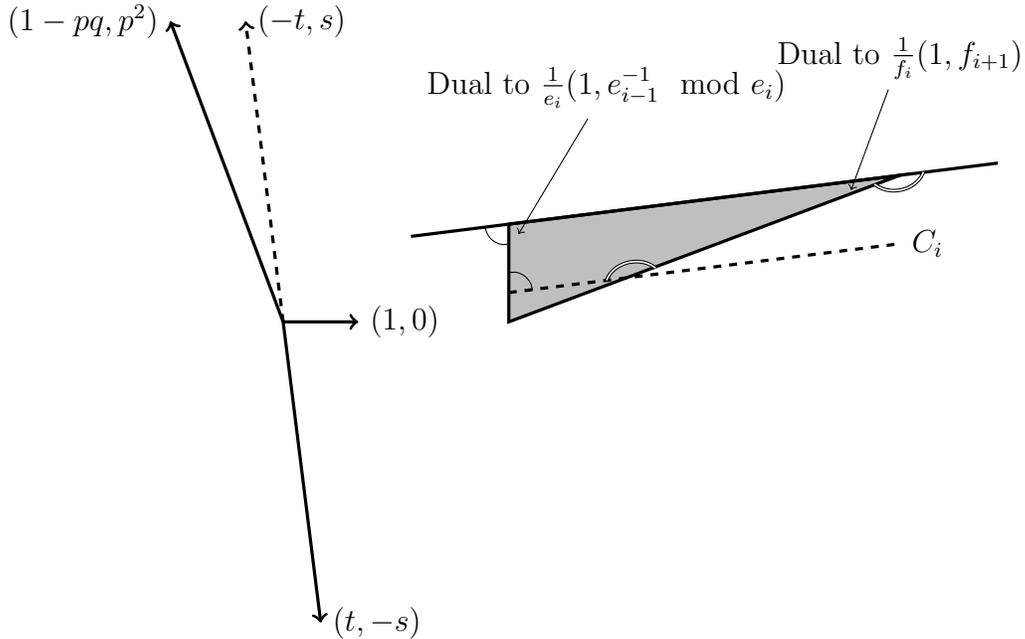
\begin{figure}[htb]
  \begin{center}
    \begin{tikzpicture}
      \draw[->,very thick] (0,0) -- (1,0) node [right] {\((1,0)\)};
      \draw[->,very thick] (0,0) -- (-1.5,4) node [left] {\((1-pq,p^2)\)};
      \draw[->,very thick,dashed] (0,0) -- (-0.5,4) node [right] {\((-t,s)\)};
      \draw[->,very thick] (0,0) -- (0.5,-4) node [right] {\((t,-s)\)};
      \begin{scope}[shift={(3,0)},scale=1.3]
        \draw[fill=lightgray,very thick] (0,1) -- (0,0) -- (4,1.5) -- cycle;
        \draw[very thick] (-1,0.875) -- (5,1.625);
        \draw[dashed,very thick] (0,0.3) -- (4,0.8) node [right] {\(C_i\)};
        \draw (0,0.5) to[out=20,in=90] (0.24,0.32);
        \draw (0,0.8) to[out=-160,in=-90] (-0.24,0.98);
        \draw[double] (1,0.425) to[out=70,in=140] (1.5,0.55);
        \begin{scope}[rotate=180,shift={(-5.24,-1.96)}]
          \draw[double] (1,0.425) to[out=70,in=140] (1.5,0.55);
        \end{scope}
        \node (a) at (1,2.4) {Dual to \(\frac{1}{e_i}(1,e_{i-1}^{-1}\mod e_i)\)};
        \node (b) at (4,2.7) {Dual to \(\frac{1}{f_i}(1,f_{i+1})\)};
        \draw[->] (a) -- (0.1,0.9);
        \draw[->] (b) -- (3.5,1.35);
      \end{scope}
    \end{tikzpicture}
    \caption{Left: The fan for the compactification of the toric
      \(\frac{1}{p^2}(1,pq-1)\) singularity. Right: The moment
      polygon for the compactification. The two new
      singularities (top left and top right corners of the
      triangle) are dual to the singularities
      \(\frac{1}{e_i}(1,e_{i-1}^{-1}\mod e_i)\) and
      \(\frac{1}{f_i}(1,f_{i+1})\) that would be introduced by the
      partial resolution with exceptional locus \(C_i\).}
    \label{fig:culet_proof}
  \end{center}
\end{figure}

This completes the proof of 
Theorem \ref{thm:input_from_es} and also reproves the theorem of Evans and Smith that \(E_{p,q}(\alpha,\beta)\) can only embed in \(\CP^2\) if \(p\) is a Markov number and \(q\) a companion.

\subsection{Proof of the Two Pin-Ball Theorem (Theorem \ref{thm:twoball})}
\label{sct:twoball}
For this subsection we assume that we have a symplectic embedding of two pin-balls 
$$\iota \colon B_{p_1,q_1}(\alpha_1) \sqcup B_{p_2,q_2}(\alpha_2) \hookrightarrow \CP^2,$$
where $1<p_1<p_2$, and we denote $\Delta=p_1p_2$.
We take the same setup as in the beginning of Section~\ref{subsec:ES}, 
that is, we take $C$ to be a component of the SFT limit curve in~$\overline{V}$ obtained by the neck-stretching procedure and denote the corresponding closed curve in the minimal resolution~$\til{X}$ by~$\til{C}$.
As explained in Remark~\ref{rmk:classesmultiple}, we write
\begin{equation}\label{eq:homologymultiple}
    K_\QQ=-3\Delta\eE+\sum_{i=1}^{m_1} k_{1,i}C_{1,i}+\sum_{i=1}^{m_2} k_{2,i}C_{2,i}\quad\text{and}\quad\til{C}_\QQ=c_0\eE+\sum_{i=1}^{m_1} c_{1,i}C_{1,i} + \sum_{i=1}^{m_2} c_{2,i}C_{2,i}.
\end{equation}
Repeating the proofs of Lemma~\ref{lma:D_interval} and Proposition~\ref{prp:embedded_curve} verbatim, 
we obtain $0<c_0 \leq \Delta$ and that $\til{C}$ is an embedded sphere with $\til{C}^2 \leq 0$. 
Writing \(\bm{\chi}_j= M_j \left(c_{j,1},\ldots,c_{j,m_j}\right)\) and 
\(\bm{k}_j=\left(k_{j,1},\ldots,k_{j,m_j}\right)\), the adjunction formula gives us an analogue of Equation~\eqref{eq:important}:
\begin{equation}\label{eq:importantmultiple}
    1=\frac{3\Delta c_0 -c_0^2}{2\Delta^2} 
    - \frac{1}{2}\left(\bm{\chi}^T_1 M_1^{-1}\bm{\chi}_1 + \bm{k}^T_1 \bm{\chi}_1\right) 
    - \frac{1}{2}\left(\bm{\chi}^T_2 M_2^{-1}\bm{\chi}_2 + \bm{k}^T_2 \bm{\chi}_2\right).
\end{equation}
In their proof of {\cite[Theorem 4.16]{ES}} Evans and Smith show the existence of an orbifold curve whose proper transform is the curve \(\til{C}\) we want; it hits each Wahl chain precisely once transversely: the first chain at a point of either \(C_{1,1}\) or \(C_{1,m_1}\), the second at a point of either \(C_{2,1}\) or \(C_{2,m_2}\) (this translates into the property that \(K_z=K_{z'}=0\) in the notation of that paper).
As in the previous section, one could repeat all their arguments in the minimal resolution, without referring to orbifolds, but we omit this derivation. 
Using this we find that Equation \eqref{eq:importantmultiple} for~$\widetilde C$ is equivalent to 
\begin{equation}\label{eq:Markoveqmultiple}
    3c_0p_1p_2=c_0^2+p_1^2+p_2^2,
\end{equation}
which is just the Markov equation.
This means that there are exactly two solutions for $c_0$, and because $c_0<\frac{\Delta}{3p_1}=\frac{p_2}{3}$
by the proof of~\cite[Lemma 4.11.B]{ES}, we know that $c_0$ is the smaller of the two solutions to~\eqref{eq:Markoveqmultiple}.

Now the SFT-limit curve \(C\) has area \(\int_{c_0\eE}\til{\omega}=c_0/(p_1p_2)\).
For \(i=1,2\), let \(\overline{N}_i\) be the negative end of \(\overline{V}\) coming from the neck around \(B_{p_i,q_i}(\alpha_i)\) and let \(D_i\) be the unique non-compact connected component of \(C\cap \overline{N}_i\).
Each \(D_i\) is properly embedded.

\begin{lemma}
 We have \(\alpha_i\leq \OP{Area}(D_i)\) for \(i=1,2\).
\end{lemma}
\begin{proof}
 Add in the orbifold points \(x_1\) and \(x_2\) to obtain girdled orbifolds \(\ha{N}_i\), which are locally modelled on \(\CC^2/\ZZ_{p_i^2}\).
 For each \(i=1,2\), take the uniformising cover \(\CC^2\to \CC^2/\ZZ_{p_i^2}\) and take the preimage of \(D_i\); this is a union of punctured curves in a Euclidean ball \(B\subset \CC^2\).
 In fact, the preimage consists of a single punctured curve \(D'_i\): since \(\til{C}\) intersects one of the end-curves of the Wahl chain, \(D_i\) is asymptotic to the quotient of either the punctured \(z_1\) or \(z_2\) axis in \(\CC^2\), so that \(\pi_1(D_i)\) maps surjectively to \(\ZZ_{p_i^2}\).

 Since \(D'_i\) is a finite-energy curve, removal of singularities tells us that it compactifies to give a properly embedded holomorphic curve \(D''_i\subset B \subset \CC^2\) passing through the origin.
 The monotonicity formula for minimal surfaces now tells us that the area of this holomorphic curve is at least the area of a complex line intersected with~\(B\).
 Therefore the area of~\(D_i\) is bounded from below by the area of the projection of a complex line to the girdled orbifold~\(\ha{N}_i\), which is~\(\alpha_i\).
\end{proof}

As a corollary,
\[\alpha_1+\alpha_2 < \OP{Area}(C)=\frac{c_0}{p_1p_2},\] 
because the curve $C$ passes through the region between the two necks.
As explained in Remark \ref{rmk:Twoballobstruction}, 
this proves the Two Pin-Ball Theorem (Theorem~\ref{thm:twoball}).

\subsection{Proof of Theorem \ref{thm:regulation_properties}} \label{subsec:regulation_properties}

The strategy of the proof is to mimic, as closely as
possible, an argument of Manetti~\cite{Manetti} but staying in
the world of symplectic geometry. Given a symplectic embedding
\(\iota:E_{p,q}(\alpha,\beta)\hookrightarrow X\) for some \(p,q,\alpha,\beta\),
choose a minimal pavilion in the sense of Definition \ref{dfn:pavilion}.

By Theorem \ref{thm:input_from_es}, there exists a square zero
\(\til{J}\)-holomorphic sphere \(\til{C}\subset\til{X}\) which
hits \(C_{i_\cul}\) once transversely and is disjoint from the
other curves \(C_j\), \(j\neq i_\cul\). 
By Theorem~\ref{thm:existence_of_regulations}, $\til{C}$ is part of a
\(\til{J}\)-holomorphic regulation on \(\til{X}\).  By Lemma
\ref{lma:j_generic}, we can choose \(J\) generically on \(V\) so
as to ensure that the only irregular genus zero \(J\)-curves are
\(C_1,\ldots,C_m\): all the other embedded \(J\)-holomorphic
spheres are either \(-1\)-spheres or have non-negative
square. By Proposition \ref{prp:ruled_structure}(c), the broken
rulings can only involve spheres of negative square, so the
broken rulings are built out of the curves \(C_1,\ldots,C_m\)
and some \(-1\)-spheres.

\begin{lemma}\label{lma:breaking}
    Every curve \(C_j\)
    except the culet curve \(C_{i_\cul}\) appears as an
    irreducible component in one of the broken rulings.
\end{lemma} 
  
\begin{proof} 
  The curves \(C_j\) other
  than \(C_{i_\cul}\) are disjoint from \(\til{C}\). Since
  \(\til{C}^2=0\), we can impose an arbitrary point constraint
  \(x\in\til{X}\) and find a ruling passing through \(x\). If
  \(x\in C_j\) then we find a ruling \(S\) intersecting \(C_j\). Since $[S] = [\til{C}]$, we find \(S\cdot C_j=0\), and thus the ruling \(S\) must contain \(C_j\)
  as an irreducible component.
\end{proof}

We now use Proposition \ref{prp:sikorav} to ensure that
\(\til{J}\) is integrable in a neighbourhood of the
\(-1\)-curves in the broken rulings, keeping it integrable along
\(C_1,\ldots,C_m\) (where it is already integrable). 
Recall from Proposition~\ref{prp:ruled_structure}(d) that every broken ruling
contains at least one $-1$-sphere. Now we can successively
holomorphically contract $-1$-spheres in broken rulings
until we reach an almost complex
4-manifold \(Y\) with a non-degenerate holomorphic regulation
with no broken rulings. This must have second Betti number \(2\)
(in fact, it must be diffeomorphic to either \(S^2\times S^2\)
or \(\CP^2\#\overline{\CP}^2\), see Remark \ref{rmk:min_mod}).

\begin{lemma}\label{lma:number_of_contractions}
    The contraction from our resolved surface \(\til{X}\) to the
    minimal model \(Y\) contracts \(m-1\) curves.
\end{lemma}

\begin{proof}
  Let \(k\) be the number of curves contracted in going from
  \(\til{X}\) to \(Y\). The second Betti number of \(\til{X}\)
  is \(k+2\). But the second Betti number of \(\til{X}\) is
  equal to \(1+m\) since it is obtained from \(\CP^2\) by
  pavilion blow-up, introducing the chain of \(m\) spheres
  \(C_1,\ldots,C_m\). Therefore \(k+2=m+1\) and \(k=m-1\).
\end{proof}

\begin{lemma}\label{lma:one_exc_sphere_per_ruling}
    There are at most two broken rulings on \(\til{X}\). More
  precisely,
  \begin{itemize}
  \item If \(i_{\cul}=1=m\) then there are no broken rulings.
  \item If \(i_{\cul}=1<m\) (respectively \(1<i_\cul=m\)) then
    there is precisely one broken ruling comprising
    \(C_2,\ldots,C_{m}\) (respectively
    \(C_1,\ldots,C_{m-1}\)) and one additional
    \(\til{J}\)-holomorphic \(-1\)-sphere \(E\).
  \item If \(1<i_\cul <m\) then there are precisely two broken
    rulings: one comprising \(C_1,\ldots,C_{i_\cul-1}\) and an
    additional \(-1\)-sphere \(E_1\), the other comprising
    \(C_{i_\cul+1},\ldots,C_{m}\) and an additional \(-1\)-sphere
    \(E_2\). 
  \end{itemize}
\end{lemma}
  
\begin{proof}
  By Proposition \ref{prp:ruled_structure}, each broken ruling
  is a tree of embedded \(\til{J}\)-holomorphic spheres with negative
  self-intersection, which are therefore either amongst the
  curves \(C_1,\ldots,C_{m}\) or else \(-1\)-spheres. All of
  the curves \(C_i\) other than the culet curve \(C_{i_\cul}\)
  appear amongst the irreducible components of some broken
  rulings by Lemma~\ref{lma:breaking}.
  
  If \(m=1\) then \(i_\cul=1\) and by Lemma~\ref{lma:number_of_contractions} 
  we have \(\til{X}=Y\), so there are no broken rulings.
  This situation arises only for pin-ellipsoids with $p=2$.
  
  If \(i_\cul=1<m\) (respectively \(1<i_\cul=m\)) then
  by Lemma~\ref{lma:broken_rulings_disjoint} the curves
  \(C_2,\ldots,C_{m}\) (respectively \(C_1,\ldots,C_{m-1}\))
  will belong to a single broken ruling. Since these curves all
  have \(C_i^2\leq -2\), 
  Proposition~\ref{prp:ruled_structure}(d) shows that 
  this broken ruling must also contain a
  \(-1\)-sphere \(E\). We need to contract all but one of the
  curves in this ruling to get to the minimal model~\(Y\). This
  is \(m-1\) curves, so there can be no other broken rulings by
  Lemma~\ref{lma:number_of_contractions}. Moreover, the broken
  ruling is precisely \(E\cup\bigcup_{i=2}^mC_i\). 
  See Figure~\ref{fig:regulations}(b) for an example; this situation arises 
  precisely for pin-ellipsoids with $p \geq 5$ along the Fibonacci branch of the Markov tree.
  
  If \(i_\cul\neq 1,m\) then there will be two broken rulings
  containing the subchains \(C_1,\ldots,C_{i_\cul-1}\) and
  \(C_{i_\cul+1},\ldots,C_{m}\). Each of these broken rulings
  will also contain a \(-1\)-sphere, say \(E_1\) and \(E_2\). In
  each broken ruling we need to contract all but one of the
  curves, which means we contract a total of \(i_\cul-1\) curves
  in the first and \(m-i_\cul\) curves in the second ruling,
  making a total of \(m-1\) curves. Again, Lemma~\ref{lma:number_of_contractions} tells us there can be no
  further rulings and the two rulings we have found contain no
  further curves. See Figure \ref{fig:regulations}(c) for an example.\qedhere
\end{proof}

\begin{remark}\label{rmk:Multipleviaiso}
    Moreover, we can identify which of the
    curves \(C_1,\ldots,C_{m}\) intersect the \(-1\)-spheres
    \(E_1\) and \(E_2\) in the broken rulings. Since the dual graph
    of the broken ruling is a tree, \(E_1\) and \(E_2\) can each
    intersect at most one of the curves in the chain, and blowing
    down \(E_1\) (respectively \(E_2\)) will turn the subchain
    \(C_1,\ldots,C_{i_\cul-1}\) (respectively
    \(C_{i_\cul+1},\ldots,C_{m}\)) into a zero continued
    fraction by Remark \ref{rmk:dual_graphs_ZCFs}. Recall that \(C_1,\ldots,C_{i_\cul-1}\) and
    \(C_{i_\cul+1},\ldots,C_{m}\) form dual Wahl chains. It is well-known (see for example {\cite[Proposition-Definition 2.3 and Proposition 2.4]{UrzuaZuniga}}) that a dual
    Wahl chain has a unique position where it can be combinatorially
    blown down to obtain a zero-continued fraction, so this gives the
    positions where \(E_1\) and \(E_2\) intersect the subchains.
\end{remark}

This completes the proof of Theorem \ref{thm:regulation_properties}.

\section{Appendix}\label{app:es_proof}

\subsection{Alternative proof of Theorem \ref{thm:input_from_es}}

In this appendix, we explain how
Theorem \ref{thm:input_from_es} follows from {\cite[Theorem
  4.15]{ES}}. Given a symplectic embedding
\(\iota:E_{p,q}(\alpha,\beta)\hookrightarrow \CP^2\), they found a Markov
triple \((p,b,c)\) and constructed a \(\ha{J}\)-holomorphic
orbifold curve \(\ha{C}\) in the orbifold \(\ha{X}\) (see
Figure \ref{fig:spaces} for the notation) satisfying the following
properties:
\begin{itemize}
\item[(ES1)] Near the orbifold point \(p\in\ha{X}\), \(\ha{C}\) has
  precisely one branch (in the terminology of \cite{ES},
  \(|Z|=1\)).
\item[(ES2)] If we look at the lift of \(\ha{C}\) to the local
  uniformising cover \(\CC^2\to\CC^2/G\), there exist local
  coordinates \((z_1,z_2)\) transforming under the action of
  \(\zeta\in G\) as
  \(\left(\zeta^{m_1}z_1,\zeta^{m_2}z_2\right)\) with respect to
  which the lift is given by
  \(z\mapsto \left(z^{b^2},a_1z^{c^2}\right)\). In
  particular, the Puiseux series has length one and the link of
  the singularity of the curve in the cover is a
  \((b^2,c^2)\)-torus knot.
\item[(ES3)] The real dimension of the moduli space containing
  \(\ha{C}\) is equal to its virtual dimension which equals
  \(2\).
\end{itemize}

\begin{remark}
We take this opportunity to correct
{\cite[Remark 3.15]{ES}}. Namely, as described in {\cite[Section
  3.3]{ES}}, the germ of the orbifold curve
\(f\colon \ha{C}\to\ha{X}\) near an orbifold point
\(z\in\ha{C}\) involves a choice of homomorphism
\(\rho_z\colon H_z\to G_z\), where \(H_z\) is the isotropy group
of \(\ha{C}\) at \(z\) and \(G_z\) is the isotropy group of
\(\ha{X}\) near \(f(z)\); a lift \(\tilde{f}_\alpha\) of \(f\)
to the local uniformising cover near \(f(z)\) satisfies the
equivariance condition
{\cite[Eq. (2)]{ES}}
\[\tilde{f}_\alpha(\zeta x) =
  \rho_z(\zeta)\tilde{f}_\alpha(x).\]
In {\cite[Remark 3.15]{ES}}, this was applied to
\(\tilde{f}_\alpha(x)=\left(x^Q,\sum a_ix^{R_i}\right)\) 
and the cyclic quotient singularity \(\frac{1}{p^2}(1,pq-1)\) to
say that
\[\left(\zeta^Qx^Q,\sum a_i\zeta^{R_i}x^{R_i}\right) =
  \left(\zeta x^Q,\zeta^{pq-1}\sum a_ix^{R_i}\right)\] and
deduce that \(Q\equiv 1\mod p^2\), \(R_i\equiv pq-1\mod p^2\)
(or vice versa). However, this assumes that \(\rho_z\) is the
identity. More generally, \(\rho_z\) could have the form
\(\rho_z(\zeta)=\zeta^n\) for some \(n\). If
\(\gcd(n,p^2)=p^2/d\) then at best we can deduce
that \[R_i/Q=(pq-1)^{\pm 1}\mod d.\] In fact, this is all that
was used in what followed. It should have been clear that the
stronger conclusion is false from the fact that eventually
\(Q=b^2\), \(R_1=c^2\), for example if \((p,b,c)=(2,5,29)\) then
neither \(2^2\) nor \(5^2\) is congruent to \(1\) modulo
\(29^2\).
\end{remark} 

\paragraph{The uniformising cover.}\label{pg:uniformising_cover}
Make an integral affine transformation of the Vianna triangle
\(\Vianna(p_1,p_2,p_3)\) as in \ref{pg:culet} to make the edge
\(\edge_1\) horizontal with the rest of the triangle below. Let
\(\bm{w}_2\) and \(\bm{w}_3\) be the primitive integer vectors
pointing out of \(\vtx_1\) along \(\edge_2\) and \(\edge_3\)
respectively; note that \(\bm{w}_2=(-u_2,p_3^2)\) and
\(\bm{w}_3=(u_3,p_2^2)\) for some \(u_2,u_3\), since
\((p_2\bm{w}_2)/(p_3p_1)\) and \((p_3\bm{w}_3)/(p_1p_2)\) must
reach the same height \(p_2p_3/p_1\) (compare with
{\cite[Paragraph~2 of the proof of Proposition~2.5]{BrendelSchlenk}} for the computation of this
height). Consider the uniformising cover \(\CC^2\to\CC^2/G\) for
the orbifold singularity \(\frac{1}{p_1^2}(p_2^2,p_3^2)\). The
cover \(\CC^2\) is also toric, its moment polygon is the
positive quadrant, and the uniformising cover is induced by the
integral affine map given by the matrix whose columns are
\(\bm{w}_2\) and \(\bm{w}_3\). Under this map, a symplectic cut
along the \((p_2^2,-p_3^2)\) direction in the moment image of
\(\CC^2\) maps to a horizontal cut in the moment image of
\(\frac{1}{p_1^2}(p_2^2,p_3^2)\), that is the culet cut.

\begin{proposition}\label{prp:ruling}
    If
  \(\ha{C}\subset\ha{X}\) is a curve with Properties (ES1--3) then the proper transform
  \(\til{C}\subset\til{X}\) satisfies the following properties:
  \begin{itemize}
  \item[(a)] \(\til{C}\) is a smooth sphere which
    intersects the culet curve \(C_{i_\cul}\) once
    transversely and is disjoint from the other curves \(C_j\),
    \(j\neq i_\cul\).
  \item[(b)] \(\til{C}\) lives in a moduli space of real dimension
    equal to its virtual dimension, which equals \(2\). In
    particular, \(\til{C}\) is a square zero
    \(\til{J}\)-holomorphic sphere, and therefore appears as a
    smooth ruling in a \(\til{J}\)-holomorphic regulation.
  \end{itemize}
\end{proposition}

\begin{proof}
  (a) By Property (ES2), the germ of
  the lifted curve in the uniformising cover \(\CC^2\to\CC^2/G\) is parametrised by
  \(z\mapsto \left(z^{b^2},a_1z^{c^2}\right)\). We can resolve
  the singularity of the germ by performing a weighted
  blow-up. Consider the surface
  \[\left\{\left(z_1, z_2, [w_1 : w_2]\right)\, : \,
      z_1^{c^2}w_2^{b^2} = w_1^{c^2}z_2^{b^2}\right\} \subset
    \CC^2 \times \PP^1(b^2,c^2).\] The total transform of our
  germ under the projection of this surface to \(\CC^2\) is
  given by the equations
  \(z_1^{c^2}w_2^{b^2} = w_1^{c^2}z_2^{b^2}\) and
  \(z_1^{c^2}=z_2^{b^2}\). In the chart \(w_2=1\), we have
  \(z_1^{c^2}=w_1^{c^2}z_2^{b^2}=z_2^{b^2}\), so the proper
  transform becomes \(w_1^{c^2}=1\); this looks like a curve
  with \(c^2\) irreducible components \(w_1=\mu\) as \(\mu\)
  runs over the set \(\bm{\mu}_{c^2}\) of \(c^2\)th roots of
  unity, but there is a residual action of \(\bm{\mu}_{c^2}\)
  (the rescalings preserving the condition \(w_2=1\)) which
  relates all of these, so the proper transform of the germ has
  a single, smooth irreducible component in this chart which
  intersects the exceptional curve \(z_2=0\) once transversely;
  we find the same component in the other chart \(w_1=1\). This
  weighted blow-up corresponds to making a symplectic cut of
  \(\CC^2\) along the hypersurface living over the line parallel
  to the vector \((b^2,-c^2)\) , which projects under the
  partial resolution of the uniformising cover to the
  culet curve, as explained in~\ref{pg:uniformising_cover}.

  (b) The same analysis applies to the proper transforms of the
  other curves~\(\ha{C}_t\) in the moduli space of~\(\ha{C}\),
  so that the proper transforms~\(\til{C}_t\) live in the same
  moduli space. Moreover, if \(\til{C}_s\) is in the same
  moduli space of~\(\til{C}\) then its projection to~\(\ha{X}\)
  lives in the same moduli space as~\(\ha{C}\), so the moduli
  spaces are diffeomorphic and the dimension computation follows
  from Property~(ES3). The fact that \(\til{C}\) is a sphere of
  square zero follows from the fact that a sphere of square~\(n\) 
  lives in a moduli space of virtual real dimension~\(2n+2\) 
  (which equals the actual dimension as long as \(n\geq-1\) by automatic transversality).
\end{proof}

This completes the alternative proof of Theorem \ref{thm:input_from_es}.

\bibliographystyle{plain} \bibliography{bib}

\begin{thebibliography}{10}

\bibitem{AbreuToricKaehler}
M.~Abreu.
\newblock K\"{a}hler geometry of toric manifolds in symplectic coordinates.
\newblock In {\em Symplectic and contact topology: interactions and
  perspectives ({T}oronto, {ON}/{M}ontreal, {QC}, 2001)}, volume~35 of {\em
  Fields Inst. Commun.}, pages 1--24. Amer. Math. Soc., Providence, RI, 2003.

\bibitem{Adal1}
N.~Adaloglou.
\newblock Embeddings and disjunction of {L}agrangian pinwheels via rational
  blow-ups.
\newblock {\em arXiv:2405.02110}, 2024.

\bibitem{Ada25}
N.~Adaloglou.
\newblock Uniqueness of {L}agrangians in {$T^*\Bbb RP^2$}.
\newblock {\em Ann. Math. Qu\'{e}.}, 49(1):215--222, 2025.

\bibitem{AdalHauber}
N.~Adaloglou and J.~Hauber.
\newblock Pinwheels in symplectic rational and ruled surfaces and non-squeezing
  of rational homology balls.
\newblock {\em arXiv:2503.16250}, 2025.

\bibitem{Aig}
M.~Aigner.
\newblock {\em Markov's theorem and 100 years of the uniqueness conjecture}.
\newblock Springer, Cham, 2013.
\newblock A mathematical journey from irrational numbers to perfect matchings.

\bibitem{BiranBarriers}
P.~Biran.
\newblock Lagrangian barriers and symplectic embeddings.
\newblock {\em Geom. Funct. Anal.}, 11(3):407--464, 2001.

\bibitem{BormanLiWu}
M.~S. Borman, T.-J. Li, and W.~Wu.
\newblock Spherical {L}agrangians via ball packings and symplectic cutting.
\newblock {\em Selecta Math. (N.S.)}, 20(1):261--283, 2014.

\bibitem{BEHWZ}
F.~Bourgeois, Y.~Eliashberg, H.~Hofer, K.~Wysocki, and E.~Zehnder.
\newblock Compactness results in symplectic field theory.
\newblock {\em Geom. Topol.}, 7:799--888, 2003.

\bibitem{BrendelSchlenk}
J.~Brendel and F.~Schlenk.
\newblock Pinwheels as {L}agrangian barriers.
\newblock {\em Commun. Contemp. Math.}, 26(5):Paper No. 2350020, 21, 2024.

\bibitem{encycbrit}
Encyclopaedia Britannica.
\newblock ``facet''.
\newblock In {\em Micropedia}, volume~4. Encyclopaedia Britannica, Inc., 15th
  edition, 1990.

\bibitem{ChenZhang}
H.~Chen and W.~Zhang.
\newblock Kodaira dimensions of almost complex manifolds {II}.
\newblock {\em Comm. Anal. Geom.}, 32(3):751--790, 2024.

\bibitem{CieliebakMohnke}
K.~Cieliebak and K.~Mohnke.
\newblock Compactness for punctured holomorphic curves.
\newblock {\em J. Symplectic Geom.}, 3(4):589--654, 2005.

\bibitem{CrHi24}
D.~Cristofaro-Gardiner and R.~Hind.
\newblock Boundaries of open symplectic manifolds and the failure of packing
  stability.
\newblock {\em arXiv:2307.01140}, 2023.

\bibitem{CGMM25}
Dan Cristofaro-Gardiner, Nicki Magill, and Dusa McDuff.
\newblock Curvy points, the perimeter, and the complexity of convex toric
  domains, 2025.

\bibitem{EGH}
Y.~Eliashberg, A.~Givental, and H.~Hofer.
\newblock Introduction to {Symplectic} {Field} {Theory}.
\newblock In {\em GAFA 2000}, pages 560--673. Basel: Birkh{\"a}user, 2000.

\bibitem{ELTF}
J.~D. Evans.
\newblock {\em Lectures on {L}agrangian torus fibrations}, volume 105 of {\em
  London Mathematical Society Student Texts}.
\newblock Cambridge University Press, Cambridge, 2023.

\bibitem{ES}
J.~D. Evans and I.~Smith.
\newblock Markov numbers and {L}agrangian cell complexes in the complex
  projective plane.
\newblock {\em Geom. Topol.}, 22(2):1143--1180, 2018.

\bibitem{ES2}
J.~D. Evans and I.~Smith.
\newblock Bounds on {W}ahl singularities from symplectic topology.
\newblock {\em Algebr. Geom.}, 7(1):59--85, 2020.

\bibitem{FS}
R.~Fintushel and R.~J. Stern.
\newblock Rational blowdowns of smooth {$4$}-manifolds.
\newblock {\em J. Differential Geom.}, 46(2):181--235, 1997.

\bibitem{Fulton}
W.~Fulton.
\newblock {\em Introduction to toric varieties}, volume 131 of {\em Annals of
  Mathematics Studies}.
\newblock Princeton University Press, Princeton, NJ, 1993.
\newblock The William H. Roever Lectures in Geometry.

\bibitem{Go82}
Mark~J. Gotay.
\newblock On coisotropic imbeddings of presymplectic manifolds.
\newblock {\em Proc. Amer. Math. Soc.}, 84(1):111--114, 1982.

\bibitem{Gro85}
M.~Gromov.
\newblock Pseudo holomorphic curves in symplectic manifolds.
\newblock {\em Invent. Math.}, 82(2):307--347, 1985.

\bibitem{HacPro10}
P.~Hacking and Y.~Prokhorov.
\newblock Smoothable del {P}ezzo surfaces with quotient singularities.
\newblock {\em Compos. Math.}, 146(1):169--192, 2010.

\bibitem{HTU}
P.~Hacking, J.~Tevelev, and G.~Urz\'{u}a.
\newblock Flipping surfaces.
\newblock {\em J. Algebraic Geom.}, 26(2):279--345, 2017.

\bibitem{HausenKiralyWrobel}
J.~Hausen, K.~Kir\'{a}ly, and M.~Wrobel.
\newblock On degenerations of the projective plane.
\newblock {\em arXiv:2405.04862}, 2024.

\bibitem{Hi10}
R.~Hind.
\newblock Lagrangian isotopies in {S}tein manifolds.
\newblock {\em arxiv:0311093}, 2003.

\bibitem{Khod2}
T.~Khodorovskiy.
\newblock Bounds on embeddings of rational homology balls in symplectic
  4-manifolds.
\newblock {\em arXiv:1307.4321}, 2013.

\bibitem{Khod1}
T.~Khodorovskiy.
\newblock Symplectic rational blow-up.
\newblock {\em arXiv:1303.2581}, 2013.

\bibitem{LiLiWu}
J.~Li, T.-J. Li, and W.~Wu.
\newblock The space of tamed almost complex structures on symplectic
  4-manifolds via symplectic spheres.
\newblock {\em Riv. Math. Univ. Parma (N.S.)}, 13(2):651--670, 2022.

\bibitem{LiWu12}
T.-J.\ Li and W.~Wu.
\newblock Lagrangian spheres, symplectic surfaces and the symplectic mapping
  class group.
\newblock {\em Geom.\ Topol.}, 16(2):1121--1169, 2012.

\bibitem{LiZhang}
Tian-Jun Li and Weiyi Zhang.
\newblock {{\(J\)}}-holomorphic curves in a nef class.
\newblock {\em Int. Math. Res. Not.}, 2015(22):12070--12104, 2015.

\bibitem{Lisca1}
P.~Lisca.
\newblock On symplectic fillings of lens spaces.
\newblock {\em Trans. Amer. Math. Soc.}, 360(2):765--799, 2008.

\bibitem{Manetti}
M.~Manetti.
\newblock Normal degenerations of the complex projective plane.
\newblock {\em J. Reine Angew. Math.}, 419:89--118, 1991.

\bibitem{McDlectures}
D.~McDuff.
\newblock Lectures on {G}romov invariants for symplectic {$4$}-manifolds.
\newblock In {\em Gauge theory and symplectic geometry ({M}ontreal, {PQ},
  1995)}, volume 488 of {\em NATO Adv.\ Sci.\ Inst.\ Ser.\ C: Math.\ Phys.\
  Sci.}, pages 175--210. Kluwer Acad. Publ., Dordrecht, 1997.
\newblock With notes by Wladyslav Lorek.

\bibitem{McDuffOpshtein}
D.~McDuff and E.~Opshtein.
\newblock Nongeneric {$J$}-holomorphic curves and singular inflation.
\newblock {\em Algebr. Geom. Topol.}, 15(1):231--286, 2015.

\bibitem{McDuffSalamon}
D.~McDuff and D.~Salamon.
\newblock {\em {$J$}-holomorphic curves and symplectic topology}, volume~52 of
  {\em American Mathematical Society Colloquium Publications}.
\newblock American Mathematical Society, Providence, RI, 2004.

\bibitem{McDSch12}
D.~McDuff and F.~Schlenk.
\newblock The embedding capacity of 4-dimensional symplectic ellipsoids.
\newblock {\em Ann. of Math. (2)}, 175(3):1191--1282, 2012.

\bibitem{McDuffSiegel24}
D.~McDuff and K.~Siegel.
\newblock Singular algebraic curves and infinite symplectic staircases.
\newblock {\em arXiv:2404.14702}, 2024.

\bibitem{McSa17}
Dusa McDuff and Dietmar Salamon.
\newblock {\em Introduction to symplectic topology}.
\newblock Oxford Graduate Texts in Mathematics. Oxford University Press,
  Oxford, third edition, 2017.

\bibitem{McLean12}
M.~McLean.
\newblock The growth rate of symplectic homology and affine varieties.
\newblock {\em Geom. Funct. Anal.}, 22(2):369--442, 2012.

\bibitem{MikhalkinShkolnikov}
G.~Mikhalkin and M.~Shkolnikov.
\newblock Wave fronts and caustics in the tropical plane.
\newblock {\em arXiv:2310.17269}, 2023.

\bibitem{Sch18}
F.~Schlenk.
\newblock Symplectic embedding problems, old and new.
\newblock {\em Bull. Amer. Math. Soc. (N.S.)}, 55(2):139--182, 2018.

\bibitem{ShevSmir}
V.~Shevchishin and G.~Smirnov.
\newblock Symplectic triangle inequality.
\newblock {\em Proc. Amer. Math. Soc.}, 148(4):1389--1397, 2020.

\bibitem{SiebertTian05}
B.~Siebert and G.~Tian.
\newblock On the holomorphicity of genus two {L}efschetz fibrations.
\newblock {\em Ann.\ of Math.\ (2)}, 161(2):959--1020, 2005.

\bibitem{Sikorav}
J.-C. Sikorav.
\newblock Singularities of {$J$}-holomorphic curves.
\newblock {\em Math. Z.}, 226(3):359--373, 1997.

\bibitem{Symington1}
M.~Symington.
\newblock Symplectic rational blowdowns.
\newblock {\em J. Differential Geom.}, 50(3):505--518, 1998.

\bibitem{Symington2}
M.~Symington.
\newblock Generalized symplectic rational blowdowns.
\newblock {\em Algebr. Geom. Topol.}, 1:503--518, 2001.

\bibitem{Symington3}
M.~Symington.
\newblock Four dimensions from two in symplectic topology.
\newblock In {\em Topology and geometry of manifolds ({A}thens, {GA}, 2001)},
  volume~71 of {\em Proc. Sympos. Pure Math.}, pages 153--208. Amer. Math.
  Soc., Providence, RI, 2003.

\bibitem{UrzuaZuniga}
G.~Urz\'{u}a and J.~P. Z\'{u}\~{n}iga.
\newblock The birational geometry of {M}arkov numbers.
\newblock {\em arXiv:2310.17957}, 2023.

\bibitem{Via14}
R.~Vianna.
\newblock On exotic {L}agrangian tori in {$\Bbb{CP}^2$}.
\newblock {\em Geom. Topol.}, 18(4):2419--2476, 2014.

\bibitem{Viannainfty}
R.~Vianna.
\newblock Infinitely many exotic monotone {L}agrangian tori in {$\Bbb{CP}^2$}.
\newblock {\em J. Topol.}, 9(2):535--551, 2016.

\bibitem{ZhangModuli}
W.~Zhang.
\newblock Moduli space of {$J$}-holomorphic subvarieties.
\newblock {\em Selecta Math. (N.S.)}, 27(2):Paper No. 29, 44, 2021.

\end{thebibliography}

\Addresses
\end{document}